\documentclass{article}

\usepackage{geometry,amsmath,amsthm,amssymb,mathrsfs,graphicx,extarrows,float,mathabx}
\usepackage[numbers]{natbib}
\usepackage{titlesec}
\usepackage[titletoc,toc,title]{appendix}
\usepackage{graphicx}
\usepackage{faktor}
\usepackage{hyperref}
\usepackage{latexsym,fullpage,amsfonts,amssymb,amsmath,amscd,graphics,epic}
\usepackage[all]{xy}
\usepackage{amssymb,amsthm,amsxtra}
\usepackage[usenames]{color}
\usepackage{amscd}
\usepackage{amsthm}
\usepackage{amsfonts}
\usepackage{amssymb}
\usepackage{mathrsfs}
\usepackage{mathtools}
\usepackage{caption}
\usepackage{subcaption}

\newtheorem{thm}{Theorem}[section]
\newtheorem{prop}[thm]{Proposition}
\newtheorem{defin}[thm]{Definition}

\newtheorem{corr}[thm]{Corollary}
\newtheorem{lemma}[thm]{Lemma}

\usepackage{mwe}

\newcommand{\Addresses}{{
		\bigskip
		\footnotesize 
		
		\textsc{Institute of Mathematics of the Polish Academy of Sciences, ul. \'{S}niadekich 8, 00-656 Warszawa, Poland}\par\nopagebreak
		\textit{E-mail address}: \texttt{stm862@gmail.com} }}

\title{Limits of traces of Temperley-Lieb algebras}
\author{Stephen T. Moore}
\date{}

\begin{document}

\maketitle 
	
\begin{abstract}
We review the classification of positive extremal traces on the generic infinite Temperley-Lieb algebra, and then extend the classification to the non-semisimple root of unity case. As a result, we obtain Hilbert space structures on the full infinite Temperley-Lieb algebra at roots of unity.
\end{abstract}	

\section{Introduction}
The Temperley-Lieb algebras have been at the foundations of the connection between a number of areas of mathematics appearing in recent decades. Originally used to define statistical mechanics models \cite{TL}, they were rediscovered in relation to subfactors by Jones \cite{Jones1}, who then used them to define the Jones polynomial knot invariant \cite{Jones2}. The Temperley-Lieb algebras also naturally appear as the Schur-Weyl dual of the quantum group $U_{q}(\mathfrak{sl}_{2})$ \cite{Jimbo, MartinSW}. More recently, their  infinite dimensional limit has appeared in relation to conformal field theory \cite{GS}, and representations of super Lie algebras \cite{Inna}. While the finite Temperley-Lieb algebras are well understood, the representation theory of the infinite Temperley-Lieb algebra is mostly unstudied, and so there is opportunity for a wide variety of approaches \cite{MeTL, Sitaraman}.\\

One fruitful approach to the study of the representation theory of infinite dimensional algebras is via their extremal traces. This approach was started by Thoma \cite{Thoma}, who classified all extremal traces for the infinite symmetric group. The extremal traces can in turn be used to realize large families of representations, either in the form of $\text{II}_{1}$ factors, or as spherical representations of the double $S_{\infty}\times S_{\infty}$. The general approach via extremal traces was further developed by Kerov and Vershik \cite{KV1, KV3, KV2}, among others \cite{BO, Wassermann}, and has been applied to a wide variety of algebras appearing as the limit of finite dimensional semisimple algebras \cite{Neretin, Wahl2, Wahl1}. \\

In the case of the Temperley-Lieb algebra, there are two main studies of its extremal traces to mention. The first is the thesis of Wassermann, \cite{Wassermann}, who classified the extremal traces for the infinite tensor product of the two dimensional $SU(2)$ representation. As this involves the same Bratteli diagram, it also gives the classification of extremal traces for the generic Temperley-Lieb algebra. The second is Jones famous Index for Subfactors paper \cite{Jones1}, this studied the case when the trace is a markov trace, which can be thought of as adding the extra condition that the generators satisfy
\[
tr(xe_{n})=\delta^{-1}tr(x).\]
This defines a unique trace, known as the \textit{Jones trace}, which in turn defines a positive-definite inner product in the generic case, and positive semi-definite inner product in the non-semisimple case.\\

To our knowledge, while a number of algebras studied via extremal traces, such as the Brauer and Hecke algebras, have non-semisimple specializations, no attempt has been made to generalize the approach to the non-semisimple case. The infinite Temperley-Lieb algebra is a natural choice to attempt this generalization, as we can first use its generic semisimple case to aid in the classification of extremal traces, and further its finite dimensional structure is well enough understood to hopefully aid in rectifying any issues that may occur.\\ 

Our main finding is that in the non-semisimple case, while the classification of extremal traces can indeed be extended, the case of positivity of the inner product defined by the trace becomes more complicated. Apart from the Jones trace, other extremal traces become \textit{indefinite}, by which we mean the set of norm zero elements is non-empty and \textit{not} linearly closed. This means that we can not even consider a quotient of the inner product. Instead, by studying the decomposition of the finite dimensional algebras, we show that we can instead define a new involution on the Temperley-Lieb algebras, and that under this new involution, the inner product defined via the trace becomes positive definite, resulting in a Hilbert space structure on the full infinite Temperley-Lieb algebra.\\

An alternative viewpoint of our construction is that we can realize the full infinite Temperley-Lieb algebra inside appropriate completions of 
\[
\bigcup_{n}TL^{S}_{n},\]
where $TL_{n}^{S}$ is the \textit{maximal} semisimplification of the Temperley-Lieb algebra, which is in general much larger than the Jones quotient. \cite{ILZ}\\

The paper is set out as follows: In Section \ref{section: TL introduction} we give the necessary background on the Temperley-Lieb algebra. In Section \ref{section: trace generic}, we review the classification of extremal traces for the generic Temperley-Lieb algebra. We study the regular representation define via the trace in Section \ref{section: Hilber space generic}, and introduce extra representations we call \textit{generalized regular representations}, which realize special values of the trace coefficients. To aid our study of the traces in the root of unity case, in Section \ref{section: Matrix decomposition} we give a method of decomposing both the semisimplification of the Temperley-Lieb algebra, and its Jacobson radical. We begin the classification of extremal traces in the non-semisimple case in Section \ref{section: traces rou}, and study the correction of the inner product in Section \ref{section: inner product rou}. In Section \ref{section: zero involution} we give formula for the corrected involution on the Temperley-Lieb generators in the case $\delta=0$. Finally, we study the decomposition of the regular representation and generalized regular representations in the root of unity case in Section \ref{section: Hilbert space rou}.

\section{The Temperley-Lieb Algebra}\label{section: TL introduction}

\begin{defin}
The \textbf{Temperley-Lieb algebras}, $TL_{n}(\delta)$, are a family of finite dimensional unital algebras depending on a parameter $\delta\in\mathbb{C}$, with generators $e_{i}$, $1\leq i\leq n-1$. The relations are as follows:
\begin{align}
	e_{i}^{2}&=\delta e_{i}\\
	e_{i}e_{i\pm 1}e_{i}&=e_{i}\\
	e_{i}e_{j}&=e_{j}e_{i}, ~ ~ \lvert i-j\rvert>1
\end{align}
\end{defin}
We will often find it convenient to write $\delta=q+q^{-1}$ for $q\in\mathbb{C}^{\times}$. We will also assume throughout that $\delta\geq 0$.\\

While it can be studied as an ordinary algebra, we will prefer to consider the Temperley-Lieb algebra as a \textit{diagrammatic algebra} \cite{Kauffman}. More explicitly, we will view $TL_{n}$ as having a special basis consisting of certain diagrams. A \textit{$TL_{n}(\delta)$ diagram} will consist of a box (whose border we will generally neglect to draw), with $n$ marked points along the top and bottom respectively. Each of these marked points is the end point of a string, and the basis of $TL_{n}$ consists of all such diagrams with non-intersecting strings. The identity element is given by the diagram consisting of $n$ vertical strings, and the generator $e_{i}$ is given by the diagram consisting of a cup and cap joining the $i$th and $i+1$th points long the top and bottom, and vertical strings connecting all other points. Often we will use thick strings to represent multiple strings, and boxes to represent sub-diagrams. For example, the diagrammatic basis of $TL_{4}$ is given by:\\
\begin{figure}[H]
	\centering
	\includegraphics[width=0.7\linewidth]{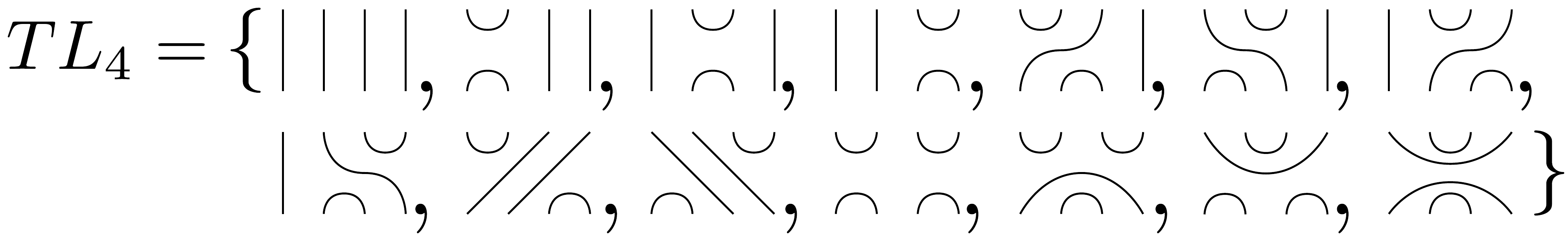}
\end{figure}
Multiplication of diagrams is given by concatenating diagrams vertically, then smoothing out any strings. If a closed circle appears in a resulting diagram, it can be removed and the diagram multiplied by $\delta$. As an example, the relations for the $TL_{4}$ generators are given by:\\
\begin{figure}[H]
	\centering
	\includegraphics[width=0.5\linewidth]{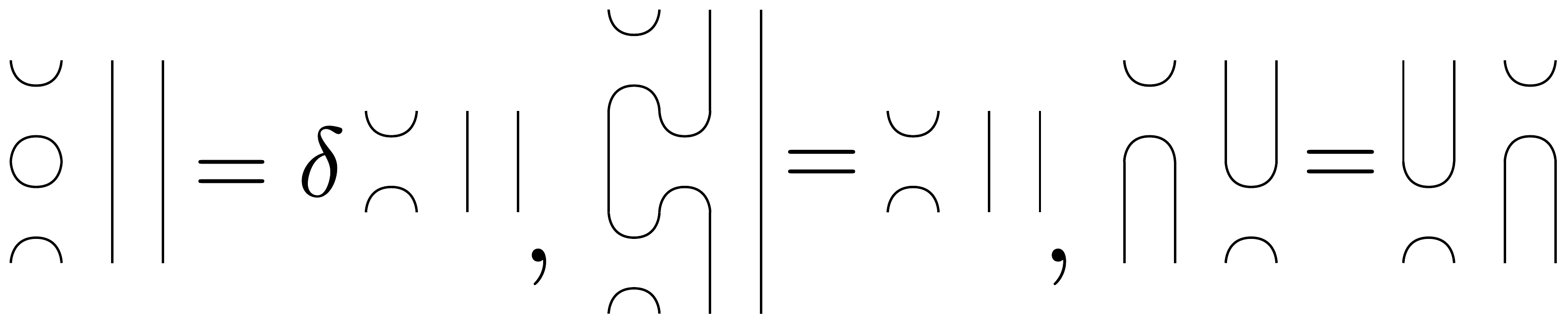}
	\caption{The diagrammatic form of the $TL_{4}(\delta)$ relations $e_{1}^{2}=\delta e_{1}$, $e_{1}e_{2}e_{1}=e_{1}$, $e_{1}e_{3}=e_{3}e_{1}$.}
\end{figure}
It is easy to see that $TL_{n-1}$ is a subalgebra of $TL_{n}$, and so we have a family of inclusions
\[
\mathbb{C}\subset TL_{2}\subset TL_{3}\subset...\]
There is a natural choice of involution on the Temperley-Lieb algebras that commutes with the inclusion of algebras:
\begin{defin}
We define the map $\dagger$ on the basis of $TL_{n}$ diagrams to be given by reflecting diagrams about the horizontal axis. Alternatively, it is given by 
\[
e_{i}^{\dagger} = e_{i}, ~ (ab)^{\dagger} = b^{\dagger}a^{\dagger}\]
for all $i$. We define the involution $\ast$ to be the conjugate linear extension of $\dagger$.
\end{defin}

\subsection{The Structure of $TL_{n}$}
The structure of $TL_{n}$ can be split into two cases, depending on the parameter. When $q$ is not a root of unity, which we refer to as the \textit{generic case}, the algebra is semisimple, and is isomorphic to 
\[
\bigoplus\limits_{p=0}^{\lfloor\frac{n}{2}\rfloor} M_{d_{n,p}}(\mathbb{C})\]
where
\[
d_{n,p}:=\binom{n}{p}-\binom{n}{p-1}.\]
When $q$ is a root of unity, $TL_{n}$ is in general non-semisimple for $n$ large enough. For $q$ a root of unity, we will fix $l\in\mathbb{N}$ as the minimal such that $q^{2l}=1$.\\

\begin{defin}
The \textbf{Jones-Wenzl idempotents}, \cite{Wenzl}, $f_{n}\in TL_{n}$ are defined inductively as follows:
\begin{figure}[H]
	\centering
	\includegraphics[width=0.4\linewidth]{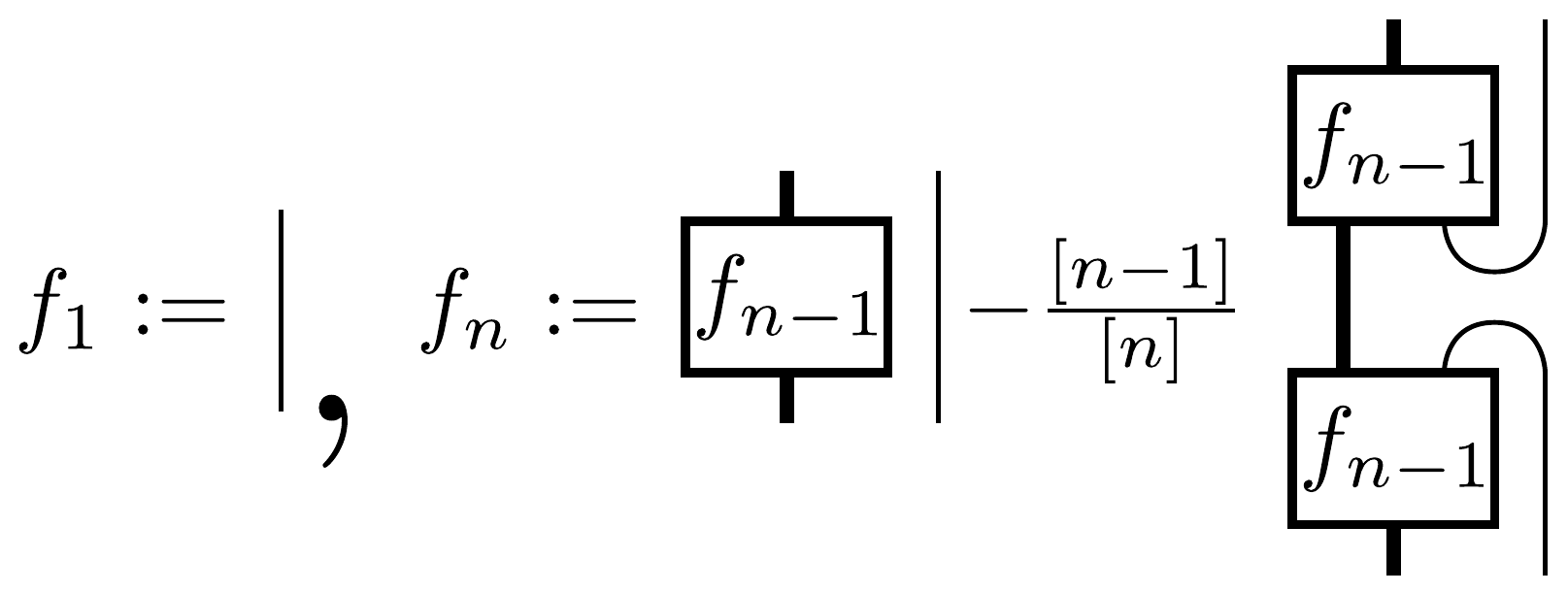}
\end{figure}
\end{defin}
where the \textbf{quantum integers} are defined by $[k]:=\frac{q^{k}-q^{-k}}{q-q^{-1}}$. They can alternatively be defined by 
\[
[0]=0, ~ [1]=1, ~ [2]=\delta, ~ [k]=\delta[k-1]-[k-2].\]
We will need the following properties of the Jones-Wenzl idempotents \cite{Morrison}:
\begin{prop}
The Jones-Wenzl idempotent is the unique idempotent in $TL_{n}$ such that 
\[
e_{i}f_{n}=f_{n}e_{i}=0,\]
for all $1\leq i\leq n-1$.
\end{prop}
\begin{prop}
The \textbf{(Jones) partial trace} of a $TL_{n}$ diagram is given by connecting together the $n$th points at the top and bottom of the diagram. The partial trace of $f_{n}$ is
\begin{figure}[H]
	\centering
	\includegraphics[width=0.2\linewidth]{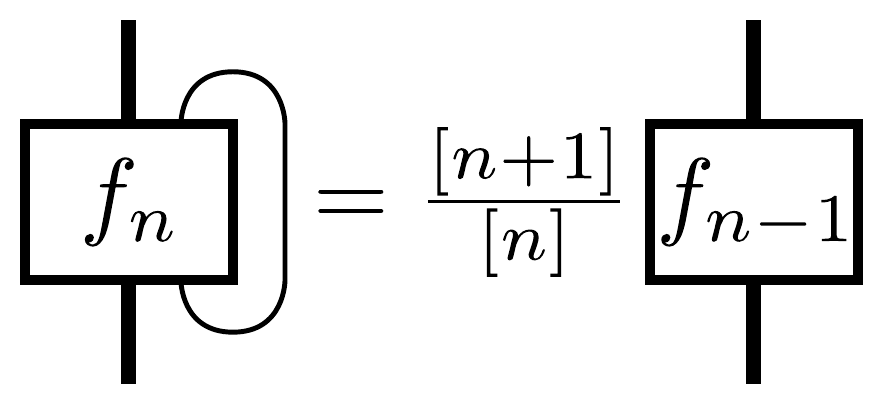}
\end{figure}
\end{prop}
For $q$ generic, the Jones-Wenzl idempotents are well-defined, and $f_{n}$ is the projection onto the unique copy of $\mathbb{C}$ appearing in the decomposition of $TL_{n}$. For $q$ a root of unity, the idempotents are not always well-defined due to certain quantum integers going to zero in the denominator. However, for certain special cases they are still well-defined:
\begin{prop}
For $q$ a root of unity, with $l$ minimal such that $q^{2l}=1$, the Jones-Wenzl idempotent  $f_{n}$ is well-defined if $n<l$, or $n=kl-1$ for some $k\in\mathbb{N}$.
\end{prop}

\subsection{Temperley-Lieb representations}\label{section: TL representations}
The representation theory of the Temperley-Lieb algebras has been studied by a number of authors using various techniques. We refer to \cite{GW, MartinBook, RSA} for the following results.
\begin{defin}
An \textbf{$(n,p)$-link state}, $0\leq p\leq\lfloor\frac{n}{2}\rfloor$, is the top half of a $TL_{n}$ diagram consisting of $p$ cups and $n-2p$ through strands.
\end{defin}
There is a natural action of $TL_{n}$ on an $(n,p)$ link state given by vertical concatenation, where we take the resulting link state to be zero if it has $p^{\prime}>p$ cups.
\begin{defin}
The \textbf{standard representation}, $\mathcal{V}_{n,p}$, is the $TL_{n}$ representation with basis consisting of $(n,p)$-link states.
\end{defin}
For example, $\mathcal{V}_{4,1}$ has basis:\\
\begin{figure}[H]
	\centering
	\includegraphics[width=0.4\linewidth]{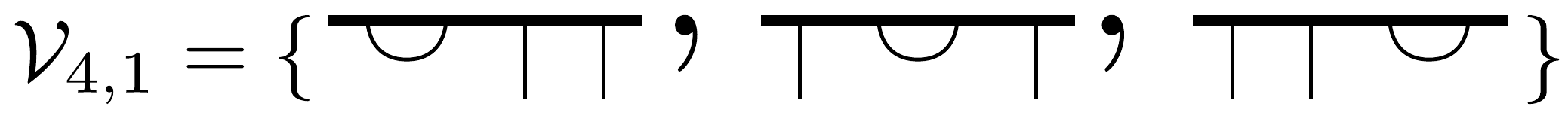}
\end{figure}
The action of $e_{1}$ on $\mathcal{V}_{4,1}$ is as follows:\\
\begin{figure}[H]
	\centering
	\includegraphics[width=0.6\linewidth]{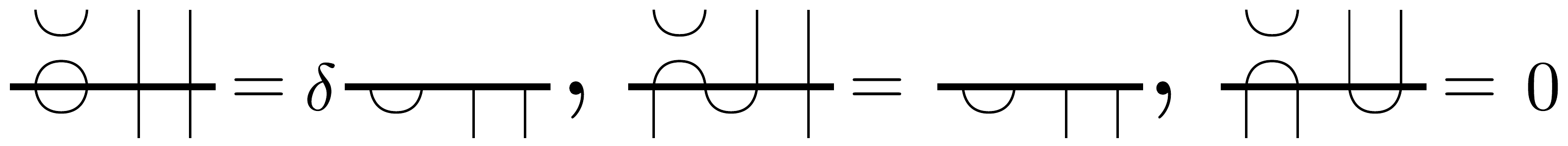}
\end{figure}
We note that 
\[
\text{dim }\mathcal{V}_{n,p} = d_{n,p}.\]
\begin{thm}\label{thm: TL reps generic}
For $q$ generic, the set of representations $\mathcal{V}_{n,p}$, $0\leq p\leq\lfloor\frac{n}{2}\rfloor$, form a complete set of irreducible representations for $TL_{n}$.
\end{thm}
We will often need to consider the restriction of $TL_{n}$ representations to $TL_{n-1}$. For $\mathcal{V}_{n,p}$ this is straightforward:
\begin{prop}\label{prop: restriction of standard reps}
For all $q$, we have
\[
0\rightarrow\mathcal{V}_{n-1,p}\rightarrow\downarrow\mathcal{V}_{n,p}\rightarrow\mathcal{V}_{n-1,p-1}\rightarrow 0\]
Further, if $n-2p-1\neq kl$ for some $k\in\mathbb{N}$, then the above sequence splits.
\end{prop}
For $q$ a root of unity, the representations $\mathcal{V}_{n,p}$ are generally not irreducible, but all $TL_{n}$ irreducibles appear as quotients of them. To describe the irreducibles in more detail, for a given pair $(n,p)$, we write:
\[
n-2p+1 = k(n,p)l+r(n,p), ~ k(n,p)\in\mathbb{N}, ~ 1\leq r(n,p)\leq l\]
\begin{defin}
We call the pair $(n,p)$ \textbf{critical} if $n-2p+1=kl$ for some $k\in\mathbb{N}_{>0}$.
\end{defin}
\begin{defin}
For $q$ a root of unity, we denote by $\mathcal{L}_{n,p}$ the irreducible quotient of $\mathcal{V}_{n,p}$.
\end{defin}
To describe the irreducibles, we first consider the case of $l=2$ separately:
\begin{thm}
For $l=2$ and $n$ odd, the representations $\{\mathcal{V}_{n,p}:0\leq p\leq \frac{n-1}{2}\}$ are irreducible and form a complete set of irreducibles. For $n$ even, we have a non-split short exact sequence
\[
0\rightarrow\mathcal{L}_{n,p-1}\rightarrow\mathcal{V}_{n,p}\rightarrow\mathcal{L}_{n,p}\rightarrow 0,\]
and the representations $\{\mathcal{L}_{n,p}:0\leq p\leq\frac{n}{2}-1\}$ form a complete set of irreducibles.
\end{thm}
For $l>2$, we have:
\begin{thm}
The representations $\{\mathcal{L}_{n,p}:0\leq p\leq\lfloor\frac{n}{2}\rfloor\}$ form a complete set of irreducibles. If $r(n,p)=l$, then $\mathcal{V}_{n,p}\simeq\mathcal{L}_{n,p}$ is irreducible, otherwise we have a non-split short exact sequence
\[
0\rightarrow\mathcal{L}_{n,p-l+r(n,p)}\rightarrow\mathcal{V}_{n,p}\rightarrow\mathcal{L}_{n,p}\rightarrow 0.\]
\end{thm}
Note that we take any representation labelled by $p^{\prime}<0$ or $p^{\prime}>\lfloor\frac{n}{2}\rfloor$ to be zero in the above.\\

For the restriction of irreducibles, we again consider the case of $l=2$ separately:
\begin{prop}
If $n$ is even and $l=2$, then
\[
\downarrow\mathcal{L}_{n,p}\simeq\mathcal{L}_{n-1,p}\simeq\mathcal{V}_{n-1,p}.\]
\end{prop}
As $\mathcal{V}_{n,p}\simeq\mathcal{L}_{n,p}$ for $l=2$ and $n$ odd, the restriction in this case follows from Proposition \ref{prop: restriction of standard reps}. For $l>2$, we have:
\begin{prop}
If $r(n,p)=1$, then 
\[
\downarrow\mathcal{L}_{n,p}\simeq\mathcal{L}_{n-1,p}\simeq\mathcal{V}_{n-1,p}.\]
If $1<r(n,p)<l$, then
\[
\downarrow\mathcal{L}_{n,p}\simeq\mathcal{L}_{n-1,p}\oplus\mathcal{L}_{n-1,p-1}.\]
\end{prop}
For $r(n,p)=l$, the restriction is again given by Proposition \ref{prop: restriction of standard reps}. Finally, we note the following:
\begin{corr}
For $q$ a root of unity and $r(n,p)=l$, the irreducible components of $\downarrow\mathcal{V}_{n,p}$ are:
\[
\{\mathcal{L}_{n-1,p},2\mathcal{L}_{n-1,p-1},\mathcal{L}_{n-1,p-l}\}\]
\end{corr}

For our purposes, we will often consider the irreducible representations for $\bigcup\limits_{n}TL_{n}$ organized as a \textit{Bratteli diagram}. This is a graph with vertices ordered into levels labelled by $\mathbb{N}$. The vertices on a given level will be indexed by the irreducible representations of $TL_{n}$. The number of edges between a vertex $(n,p)$ on level $n$ and a vertex $(n-1,p^{\prime})$ on level $n-1$ will be the multiplicity of the corresponding $TL_{n-1}$ irreducible appearing in the restriction of the $TL_{n}$ irreducible. For $q$ generic, this will just consist of a single edge connecting the vertex $(n,p)$ to $(n-1,p-1)$ and a single edge connecting $(n,p)$ to $(n-1,p)$. The $q$ generic Bratteli diagram is then as follows:
\begin{figure}[H]
	\centering
	\includegraphics[width=0.4\linewidth]{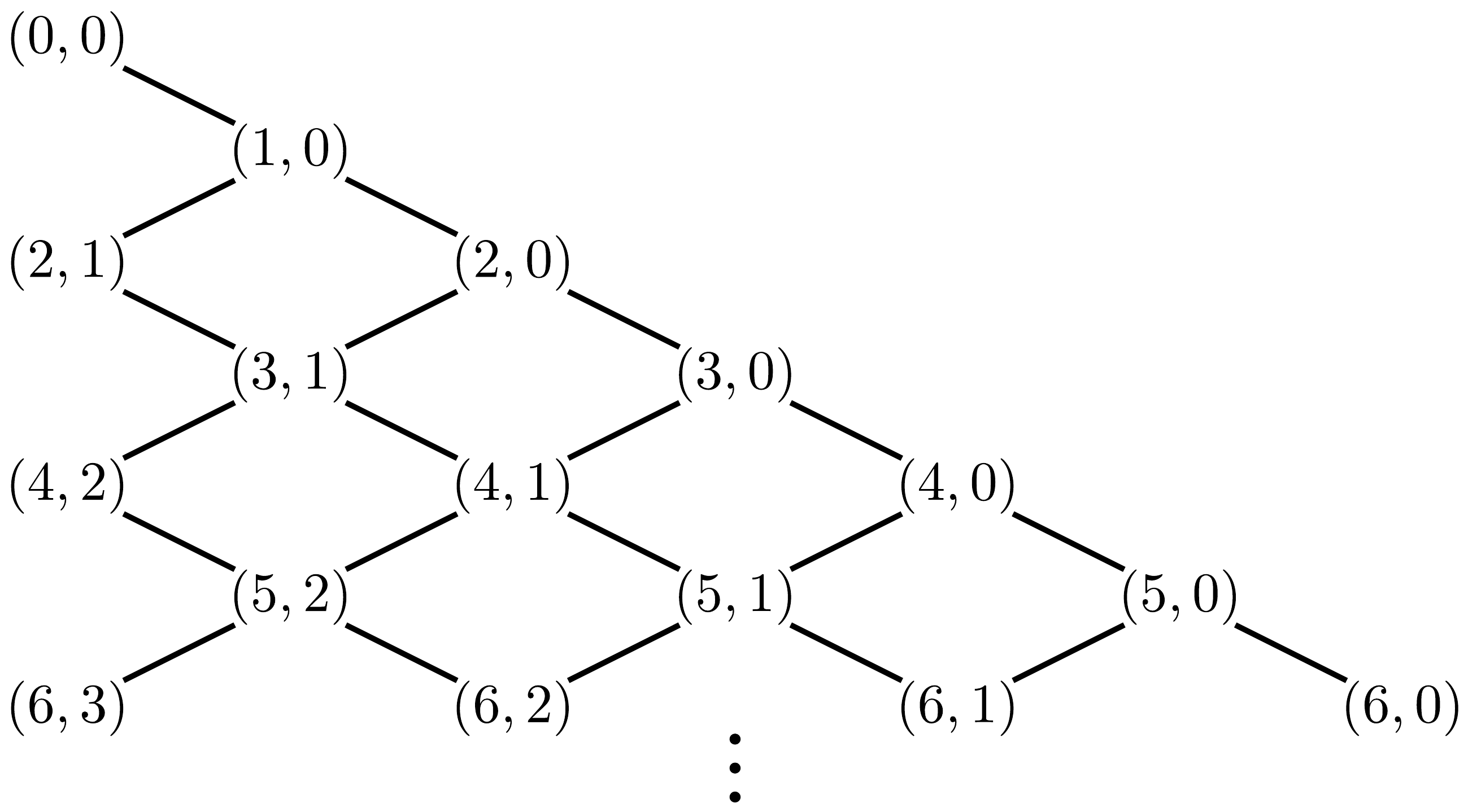}
	\caption{The Bratteli diagram for $q$ generic.}
\end{figure}
For $q$ a root of unity, the graph will depend on the choice of $l$. In this case, we will also include extra information on the graph in the form of red lines passing through vertices $(n,p)$ such that $n-2p+1=kl$. These extra lines will be known as \textit{critical lines}. If $(n,p)$ lies on a critical line, then there will be an edge to $(n-1,p)$, two edges to $(n-1,p-1)$, and an edge to $(n-1,p-l)$. There will be no edge from $(n+1,p+1)$ to $(n,p)$. For other vertices, there will be single edges from $(n,p)$to $(n-1,p)$ and $(n-1,p-1)$. The Bratteli diagrams for $l=2,3$ are as follows:
\begin{figure}[H]
	\begin{subfigure}[b]{0.45\textwidth}
		\includegraphics[width=\textwidth]{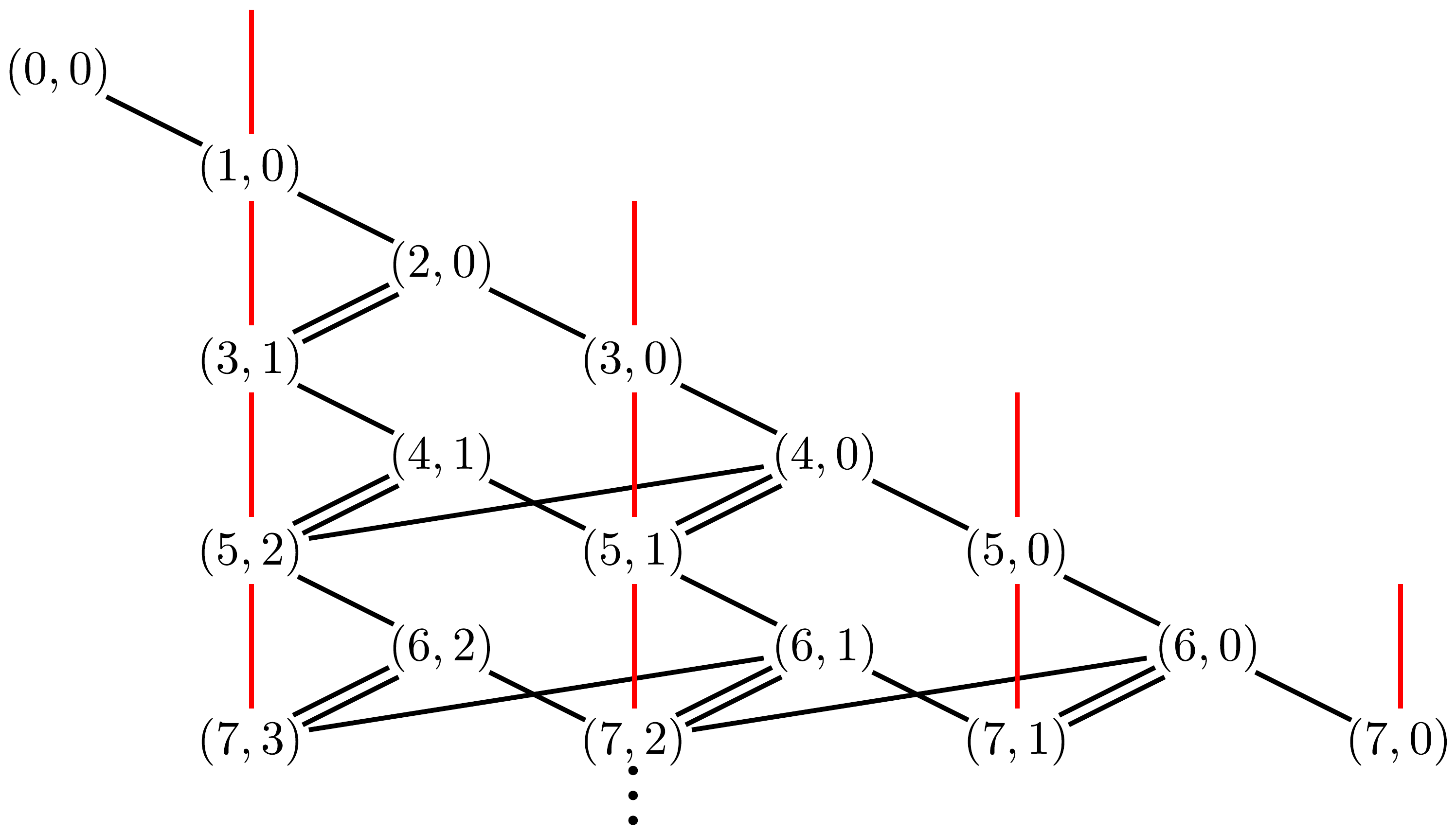}
	\end{subfigure}
	\hfill
	\begin{subfigure}[b]{0.45\textwidth}
		\includegraphics[width=\textwidth]{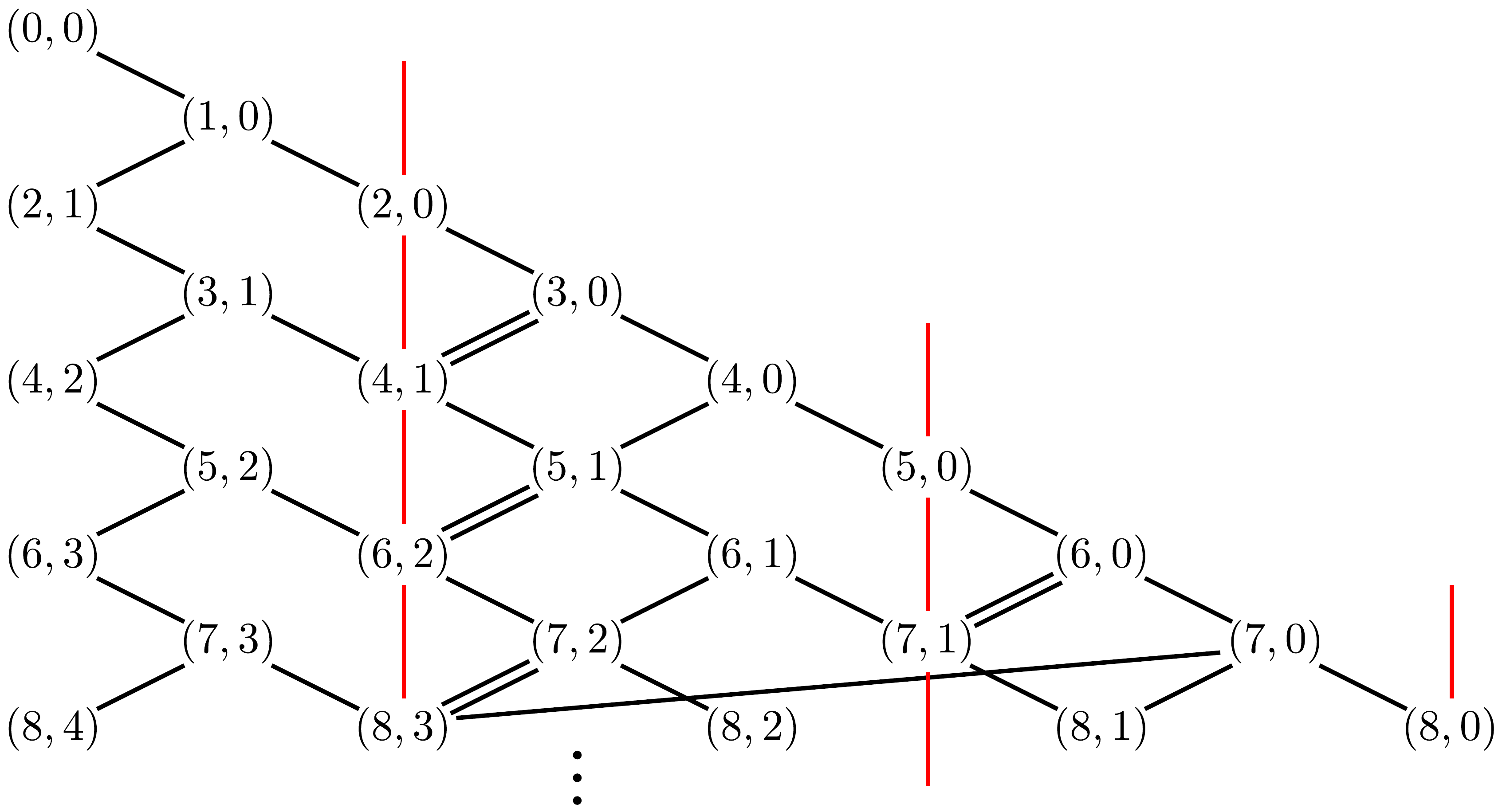}
	\end{subfigure}
	\caption{The Bratteli diagrams for $l=2$ and $l=3$.}
\end{figure}

\section{Limits of traces for $q$ generic}\label{section: trace generic}

\begin{defin}
A \textbf{trace}, $t$, on an algebra $A$, is a linear function $t:A\rightarrow\mathbb{C}$ that satisfies:
\[
t(ab)=t(ba).\]
We call the trace \textbf{normalized} if
\[
t(1)=1,\]
and \textbf{positive} if 
\[t(aa^{\ast})> 0\]
for all $a\in A$, $a\neq 0$, where $\ast$ is an involution on $A$.
\end{defin}
Recall there is a unique positive definite trace, up to a constant, on the algebra $M_{n}(\mathbb{C})$. Given a finite dimensional semisimple algebra, the possible traces on it are then just linear combinations of the traces on its irreducible components.\\

Consider now an algebra $A_{\infty}$ constructed as the limit of a family of semisimple algebras, i.e.
\[
A_{\infty} = \bigcup\limits_{n}A_{n}.\]
Denote the matrix traces on $A_{n}$ by $t_{n,i}$. Then given a trace $t_{\infty}$ on $A_{\infty}$, and considering its  restriction to $A_{n}$, we find:
\[
A_{\infty}^{(n)} = \sum\limits_{i}c_{n,i}t_{n,i}\]
for some $c_{n,i}\in\mathbb{C}$. We can further restrict to $A_{n-1}$ in two ways:
\[
t_{\infty}^{(n-1)} = \sum\limits_{j}c_{n-1,j}t_{n-1,j}\]
\[
t_{\infty}^{(n-1)} = \sum\limits_{i}c_{n,i}\downarrow t_{n,i}\]
where $\downarrow t_{n,i}$ is the restriction of traces coming from restricting the corresponding irreducible matrix component. Comparing the two ways of restricting, as well as forcing $t_{\infty}(1)=1$, we get conditions on the possible traces on $A_{\infty}$. However in general these conditions are not enough to fully determine the possible traces on $A_{\infty}$. In particular, given a set of traces on $A_{\infty}$, we can form a convex sum to obtain another trace. We therefore want to determine which traces on $A_{\infty}$ can't be written as a convex sum:
\begin{defin}
A normalized trace is called \textbf{extremal} if it can not be written as a convex sum of other normalized traces.
\end{defin} 
To determine which traces are extremal, we will use the following:\\

Let $\Gamma$ be a $\mathbb{N}$ graded graph, and $e(\mu,\nu)$ denote the number of edges from $\mu$ to $\nu$. We denote $\mu\nearrow\nu$ if $\mu\in\Gamma_{n}$, $\nu\in\Gamma_{n+1}$, and $e(\mu,\nu)>0$. Further, we assume $\Gamma_{0}$ and $\Gamma_{1}$ both contain a single vertex, which we denote by $v_{0},v_{1}$ respectively.
\begin{defin}
$\Gamma$ is a \textbf{multiplicative graph}, \cite{Borodin-Olshanski}, if there is an $\mathbb{N}$ graded associative commutative $\mathbb{R}-$algebra $A_{\Gamma}$ with basis $\{a_{v}\}$ indexed by vertices $v\in\Gamma$ such that:
\begin{enumerate}
	\item $a_{v_{0}} = 1$,
	\item $\text{deg }a_{\mu} = \lvert \mu\rvert$,
	\item $a_{v_{1}}a_{\mu} = \sum\limits_{\nu:\mu\nearrow\nu}e(\mu,\nu)a_{\nu}$.
\end{enumerate}
\end{defin}
Given a trace $t_{\infty}$ on $A_{\infty}$ as previous, we can construct a graph $\Gamma$, by taking the vertices labelled by irreducible traces of $A_{n}$, and edges denoting restriction. Then if $A_{\Gamma}$ exists, we can construct a linear functional 
\[
\hat{t}_{\infty}:A_{\Gamma}\rightarrow\mathbb{C}\]
by taking
\[\hat{t}_{\infty}:a_{n,i}\mapsto c_{n,i}.\]
\begin{defin}
	The trace $t_{\infty}$ is called \textbf{multiplicative} if $\hat{t}_{\infty}$ is a ring homomorphism on $A_{\Gamma}$.
\end{defin}
\begin{thm}
	\textbf{Kerov-Vershik Ring Theorem:} The (positive normalized) trace $t_{\infty}$ is extremal if and only if it is multiplicative.
\end{thm}
For proof see \cite{GO}. \\

For our considerations, we will be taking the graph $\Gamma$ to be the Bratteli diagrams defined at the end of Section \ref{section: TL introduction}. A similar idea to the above theorem, using the Grothendieck ring structure on almost-finite algebras, was introduced independently by Wassermann.

\subsection{Extremal traces for the generic Temperley-Lieb algebra.}
We can now proceed to the classification of extremal traces of $TL_{\infty}$ for $q$ generic. Let $t_{\infty}$ be a trace on $TL_{\infty}$, and $t^{(n)}_{\infty}$ denote its restriction to $TL_{n}$. From Theorem \ref{thm: TL reps generic}, we know that $TL_{n}$ has irreducibles \[
\mathcal{V}_{n,p}, ~ 0\leq p\leq\lfloor\frac{n}{2}\rfloor.\]
Let $t_{n,p}$ be the corresponding matrix trace. 
we can then write 
\[
t_{\infty}^{(n)} = \sum\limits_{i}c_{n,i}t_{n,i}\]
for some coefficients $c_{n,i}\in\mathbb{C}$. From the  restriction of representations, we have
\[
\downarrow t_{n,i} = t_{n-1,i-1}+t_{n-1,i}.\]
Note that we use the un-normalized traces for $\mathcal{V}_{n,p}$, so $t_{n,p}(1)=d_{n,p}$.\\

Recall the Bratteli diagram for $q$ generic given at the end of Section \ref{section: TL introduction}.
\begin{prop}\label{prop: generic multiplicative graph}
The Bratteli diagram $\Gamma_{TL}$ for generic $TL_{\infty}$ is a multiplicative graph.
\end{prop}
\begin{proof}
Define the algebra
\[
A_{TL}:=\{a_{n,i}:0\leq n\leq\lfloor\frac{n}{2}\rfloor, ~ n\in\mathbb{N}\},\]
with multiplication given by 
\[
a_{n,i}a_{m,j} = a_{n+m,i+j}+a_{n+m,i+j+1}.\]
Clearly this product is commutative and associative, and so $A_{TL}$ satisfies the requirements for $\Gamma_{TL}$ to be a multiplicative graph.
\end{proof}
Before we proceed further, we will need the following:
\begin{defin}
For $r\in\mathbb{C}^{\times}$, we define the functions $[n]_{r}$ by
\[
[n]_{r} = \frac{r^{n}-r^{-n}}{r-r^{-1}}.\]
These satisfy
\[
[0]_{r}=0, ~ [1]_{r}=1, ~ [2]_{r} = r+r^{-1}, ~ [n]_{r} = [2]_{r}[n-1]_{r}-[n-2]_{r}.\]
\end{defin}
\begin{lemma}
The extremal traces of $TL_{\infty}$ for $q$ generic must satisfy the following:
\[
c_{n,0} = \frac{[n+1]_{r}}{([2]_{r})^{n}}, ~ c_{n,i} = \gamma^{i}c_{n-2i,0}, ~ \gamma = ([2]_{r})^{-2}\]
\end{lemma}
\begin{proof}
Considering the restriction of $t_{\infty}$, we get
\[
t_{\infty}^{(n-1)} = \sum\limits_{i}c_{n,i}(t_{n-1,i-1}+t_{n-1,i}) = \sum\limits_{j}c_{n-1,j}t_{n-1,j}.\]
Hence we must have
\[
c_{n-1,i} = c_{n,i}+c_{n,i+1},\]
with $t_{\infty}(1)$ forcing $c_{0,0}=c_{1,0}=1$. To find the extremal traces, we now need to consider whether the trace is multiplicative. Considering the algebra $A_{TL}$ from Proposition \ref{prop: generic multiplicative graph}, we have the product
\[
a_{2n+2,n+1} = a_{2,1}a_{2n,n}.\]
Hence for $t_{\infty}$ to be multiplicative, we require
\[
c_{2n+2,n+1} = c_{2,1}c_{2n,n}.\]
It is straightforward to see that this requirement combined with the previous ones forces $t_{\infty}$ to be fully determined by the value of $c_{2,1}$. From now on, we set
\[
c_{2,1}=\gamma.\]
Multiplicativity then forces
\[
c_{2n-1,n-1} = \gamma^{n-1}, ~ c_{2n,n} = \gamma^{n}.\] Noting that these are the coefficients for $c_{2,1}$ and $c_{1,0}$ multiplied by a power of $\gamma$, then combined with the lower vertices having to satisfy the same recursion relations, it follows that we must have
\[
c_{n,i} = \gamma^{i}c_{n-2i,0}.\]
Hence we only need to determine the values of $c_{n,0}$. We have
\[
c_{n,0} = c_{n-1,0} - c_{n,1} = c_{n-1,0}-\gamma c_{n-2,0}.\]
Then putting
\[
\gamma = ([2]_{r})^{-2}, ~ c_{n,0} = \frac{[n+1]_{r}}{([2]_{r})^{n}},\]
into the above relation, we get
\[
\frac{[n+1]_{r}}{([2]_{r})^{n}} = \frac{[n]_{r}}{([2]_{r})^{n-1}} - \frac{[n-1]_{r}}{([2]_{r})^{n}}\]
which can be seen to satisfy the recursion relation for $[n]_{r}$.
\end{proof}
Writing out the first few terms of the coefficients, we get as follows:
\begin{figure}[H]
	\centering
	\includegraphics[width=0.5\linewidth]{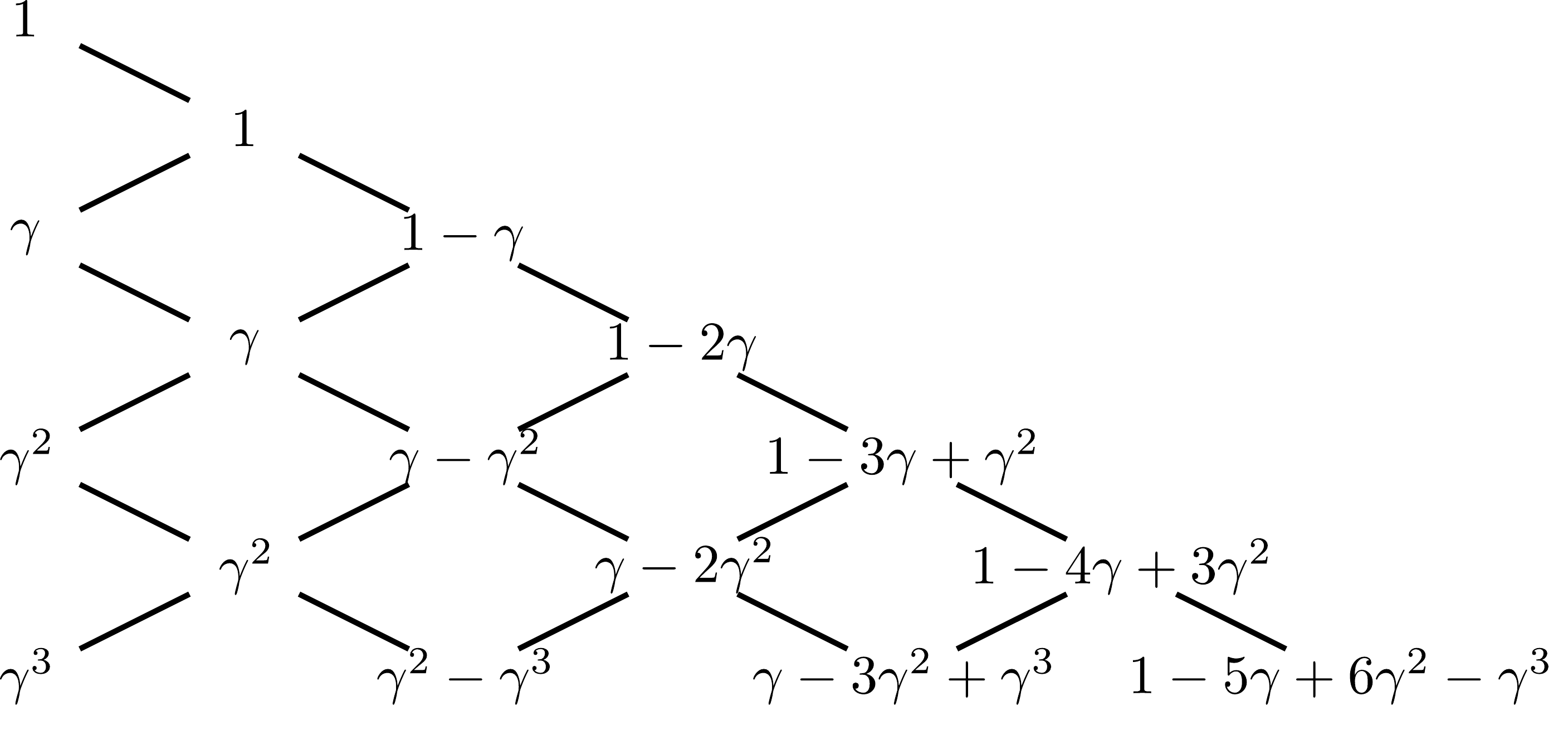}
\end{figure}
We note that there is a straightforward way to check the coefficients are correct, as $t_{n,i}(1)=\text{dim }\mathcal{V}_{n,i}$, we must have
\[
\sum\limits_{i}c_{n,i}d_{n,i} = 1.\]

So far we have neglected to mention positive-definiteness of the trace, i.e. we want to determine when $t_{\infty}(a^{\ast}a)\geq 0$ for all $a\in TL_{\infty}$. Considering the matrix decomposition of $TL_{n}$, with $e^{(i)}_{kl}\in M_{n,i}(\mathbb{C})$, the involution $\ast$ gives $(e^{(i)}_{kl})^{\ast}=e^{(i)}_{lk}$. Hence for $a\in TL_{n}$, we can consider it as 
\[
a = \sum\limits_{i,k,l}\lambda^{(i)}_{kl}e^{(i)}_{kl}\]
for some $\lambda^{(i)}_{kl}\in\mathbb{C}$. As $t_{n,i}(e^{(i)}_{ll})=1$, the trace then gives
\[
t_{\infty}(aa^{\ast}) = \sum_{i,k,l}\lvert\lambda^{(i)}_{kl}\rvert^{2} c_{n,i}.\]
Hence positive-definiteness of the trace is given by determining when the coefficients $c_{n,i}$ are all positive. This is then just a restatement of Jones' famous result concerning the possible values of the index for subfactors \cite{Jones1}:
\begin{thm}\label{thm: positive coeffs}
	\hfill
\begin{itemize}
	\item For $0<\gamma\leq \frac{1}{4}$, we have the coefficients $c_{n,i}>0$ for all $n,i$.
	\item For $\gamma=0$, $c_{n,0}=1$ for all $n$, and $c_{n,i}=0$ for all $i>0$.
	\item For $\gamma\in \{\frac{1}{4}\sec^{2}(\frac{\pi}{k}):k\geq 3,k\in\mathbb{N}\}$, $c_{n,i}>0$ for $i\leq k-2$, and $c_{k-1,0}=0$.
\end{itemize}  
\end{thm}
From the first case, we immediately get:
\begin{corr}
The positive extremal traces on $TL_{\infty}$ for $q$ generic are classified by $0<\gamma\leq\frac{1}{4}$.
\end{corr}
The case $\gamma=0$ corresponds to the trivial representation, i.e. $e_{i}=0$ for all $i$.\\

The case $\gamma=\frac{1}{4}\sec^{2}(\frac{\pi}{k})$ does not correspond to a trace on the generic $TL_{\infty}$, as it would require the quotient $\bigcup\limits_{n}\left(TL_{n}/f_{k-1}\right)$, which is not a $TL_{\infty}$ representation in the generic case. However, we shall see later on that there are alternative $TL_{\infty}$ representations where we can realize these special values.\\ 

Before proceeding further, we want to describe how to calculate the value of the trace on diagram elements. Note that given a diagram $x\in TL_{n}$, we can find a sequence
\[
1\leq i_{1}<...<i_{k}\leq n-1, ~ i_{j}\in\mathbb{N}\]
such that
\[
t_{\infty}(x) = t_{\infty}(e_{i_{1}}e_{i_{2}}...e_{i_{k}}).\]
This follows from a combination of cyclicity of the trace, and that fact that we can write $x$ so that $e_{n-1}$ appears at most once. If a sequence consists of indices increasing by $1$ each time, we have:
\[
t_{\infty}(e_{1}e_{2}...e_{k}) = \begin{cases} \gamma^{\frac{k}{2}} & k\text{ even }\\ \delta\gamma^{\frac{k+1}{2}} & k\text{ odd } \end{cases}.\]
By multiplicativity, if a diagram $y\in TL_{n}$ can be viewed as 
\[
y_{1}\otimes y_{2}\in TL_{n_{1}}\otimes TL_{n_{2}}, ~ n_{1}+n_{2}=n,\]
then
\[
t_{\infty}(y)=t_{\infty}(y_{1})t_{\infty}(y_{2}).\]
It then follows that we can write the trace of any other sequence as a product of traces of the above such sequences.

\section{Hilbert space constructions for $q$ generic}\label{section: Hilber space generic}

Given a positive extremal trace $t_{\infty}$ on $TL_{\infty}$ as in the previous section, we can define an inner product on $TL_{\infty}$ by putting
\[
\langle a,b\rangle_{\gamma} := t_{\infty}(ab^{\ast})\]
for $a,b\in TL_{\infty}$. Then if $0<\gamma\leq\frac{1}{4}$, this inner product is positive definite, and we can take the completion to get a Hilbert space structure on $TL_{\infty}$. We denote this by $\mathcal{H}_{\gamma}$.\\

Recall that two inner products $\langle\cdot,\cdot\rangle_{1}$, $\langle\cdot,\cdot\rangle_{2}$ on a vector space $V$ are equivalent if there is $a,b\in\mathbb{R}$, $0<a\leq b$, such that for all $v\in\ V$ we have
\[
a\langle v,v\rangle_{1}\leq \langle v,v\rangle_{2}\leq b\langle v,v\rangle_{1}.\] 
For the inner products constructed from the trace, we note the following:
\begin{prop}
The inner products $\langle\cdot,\cdot\rangle_{\gamma}$ are inequivalent for different values of $\gamma$.
\end{prop}
\begin{proof}
Without loss of generality, assume $\gamma>\gamma^{\prime}$. We define
\[
h_{i}:=\delta^{-i}e_{1}e_{3}...e_{2i-1}\]
Then
\[
\langle h_{i},h_{i}\rangle_{\gamma} = \gamma^{i}.\]
Hence we need $a,b$ such that 
\[
a\leq\left(\frac{\gamma^{\prime}}{\gamma}\right)^{i}\leq b\]
for all $i\geq 1$. However, considering larger values of $i$, we see it forces $a=0$.
\end{proof}
We can define an action of $x\otimes y\in TL_{\infty}\otimes TL_{\infty}$ on $\mathcal{H}_{\gamma}$ by putting 
\[
(x\otimes y)v = xvy^{\ast}.\]
Then we have commuting actions of $TL_{\infty}\otimes 1$, $1\otimes TL_{\infty}$ which generate von Neumann factors. However our focus is on studying $TL_{\infty}$, not its completion, so we neglect further discussion about von Neumann algebras. We want to describe the general decomposition of the action of $TL_{\infty}\otimes TL_{\infty}$ on $\mathcal{H}_{\gamma}$. Before proceeding further, we will find it useful to construct an orthonormal basis for $\mathcal{H}_{\gamma}$. Our starting point will be the matrix decomposition $TL_{n}$: \\ 

There is a nice method of constructing the decomposition of $TL_{n}$ into matrix elements for $q$ generic via the Schur-Weyl duality of $TL_{n}(q+q^{-1})$ with $U_{q}(\mathfrak{sl}_{2})$. Viewing $TL_{n}$ as the endomorphism algebra of $(\mathcal{X}_{2})^{\otimes n}$, with $\mathcal{X}_{2}$ the standard $U_{q}(\mathfrak{sl}_{2})$ representation, then we can build a matrix decomposition of $TL_{n}$ by using the fusion rules for tensor products of the form $\mathcal{X}_{k}\otimes\mathcal{X}_{2}$, where $\mathcal{X}_{k}$ is the $k$-dimensional $U_{q}(\mathfrak{sl}_{2})$ representation. \\

For $q$ generic, the Jones-Wenzl projection $f_{n}$ is the projection onto the unique copy of $\mathcal{X}_{n+1}$ appearing in $(\mathcal{X}_{2})^{\otimes n}$. The fusion rule $\mathcal{F}_{n}$ for
\[
\mathcal{X}_{n}\otimes\mathcal{X}_{2}\simeq\mathcal{X}_{n+1}\oplus\mathcal{X}_{n-1}\]
is given by
\begin{figure}[H]
	\centering
	\includegraphics[width=0.12\linewidth]{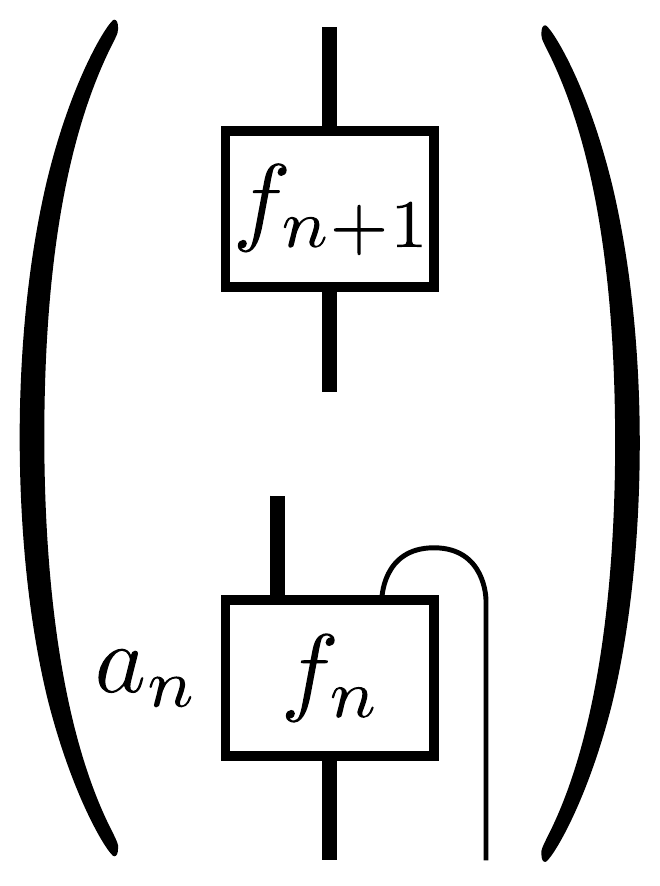}
\end{figure}
where we denote $a_{n}:=\left(\frac{[n]}{[n+1]}\right)^{\frac{1}{2}}$.
We can see that taking $(\mathcal{F}_{n}^{\dagger})^{T}\mathcal{F}_{n}$ and $\mathcal{F}_{n}(\mathcal{F}_{n}^{\dagger})^{T}$ gives the projections onto the left and right hand sides of the above isomorphism.\\

Starting with the fusion rule for $\mathcal{X}_{2}\otimes\mathcal{X}_{2}$, we can tensor it on the right with another copy of $\mathcal{X}_{2}$, then apply fusion rules to the elements of the resulting vector. Note that we use the transpose of the individual diagrams in the fusion rule, as otherwise we would obtain diagrams in $TL_{n-2p}$ instead. Doing this repeatedly, we end up with the fusion rule $\mathcal{F}_{\otimes n}$ for the decomposition of $(\mathcal{X}_{2})^{\otimes n}$. For the first few values of $n$ we have:
\begin{figure}[H]
	\centering
	\includegraphics[width=0.45\linewidth]{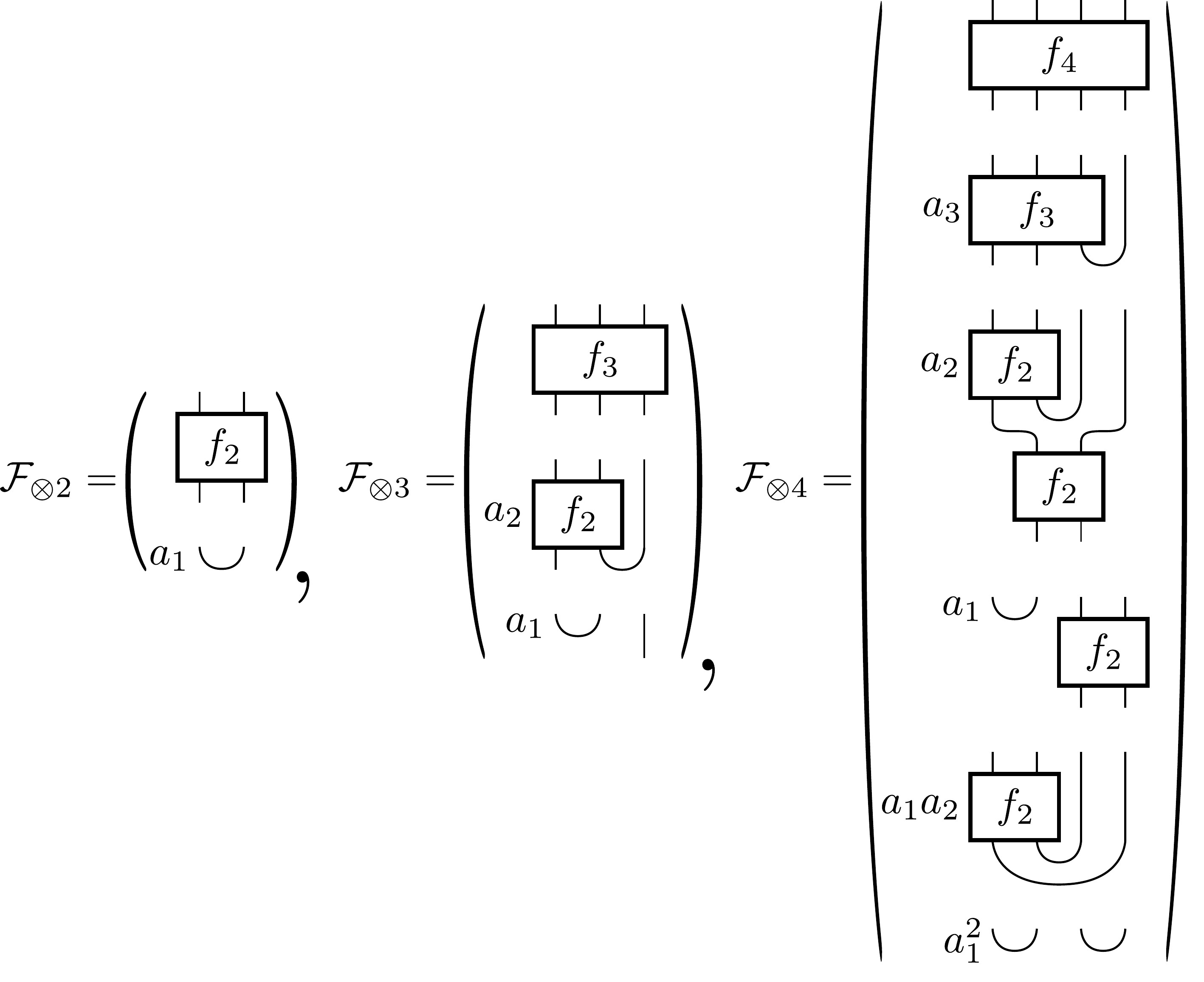}
\end{figure}
Taking \[
\mathcal{F}_{\otimes n}\left(\mathcal{F}_{\otimes n}^{\dagger}\right)^{T}\]
we then get a basis for $TL_{n}$ as a matrix algebra. For example, for $n=3$ we have
\[
TL_{3}\simeq\mathbb{C}\oplus M_{2}(\mathbb{C})\]
which using the above method is given by
\begin{figure}[H]
	\centering
	\includegraphics[width=0.25\linewidth]{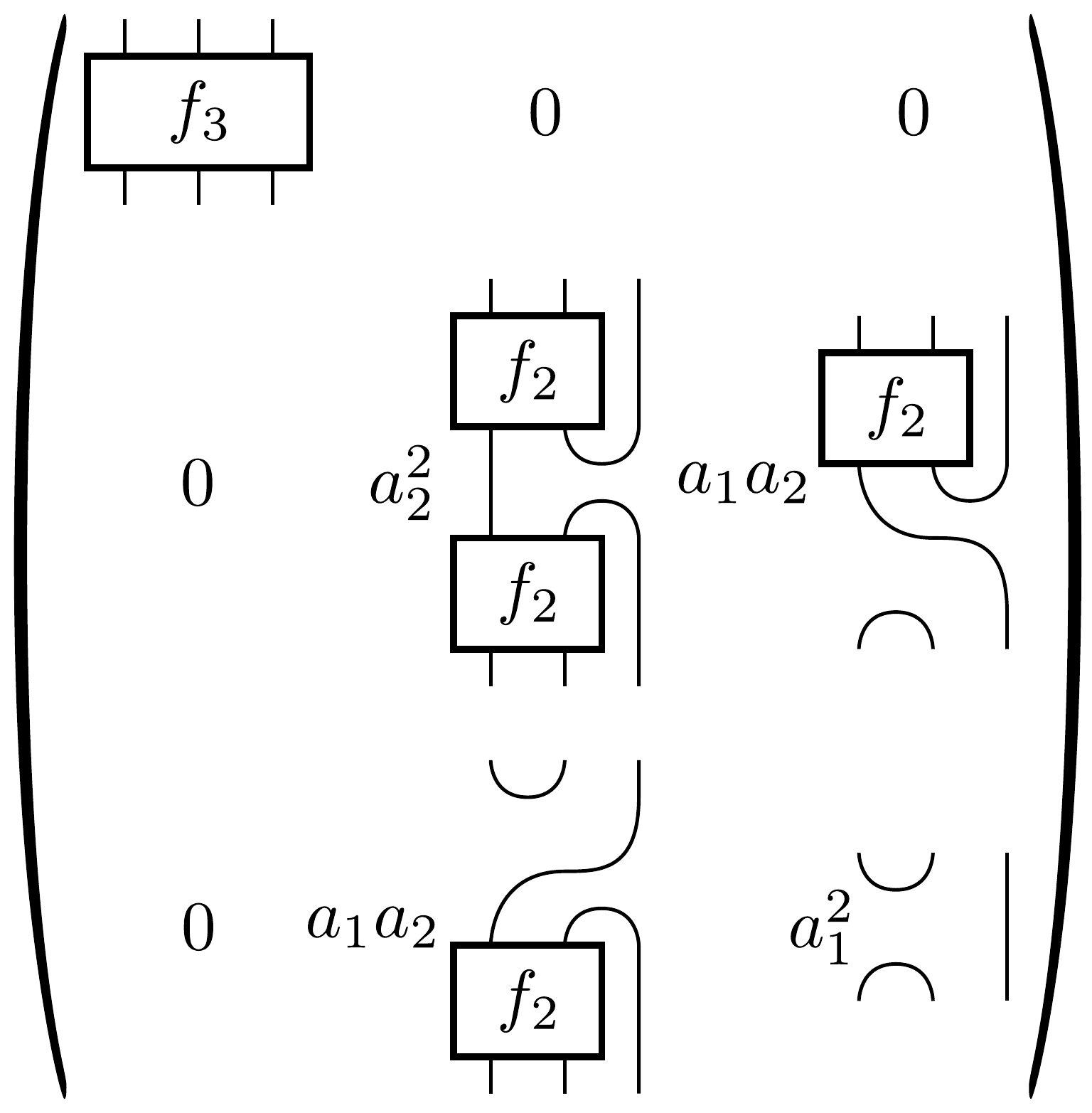}
\end{figure} 
Considering the Bratteli diagram for $TL_{\infty}$, we see that going from $\mathcal{F}_{\otimes n}$ to $\mathcal{F}_{\otimes (n+1)}$ for a given diagram corresponds to choosing an edge on the Bratteli diagram from the $n$th to $(n+1)$th level. Hence given an $n\rightarrow n-2i$ half diagram in the fusion rules, we can associate with it a unique path on the Bratteli diagram from the vertex $(1,0)$ to $(n,i)$. We can then view matrix elements of $TL_{n}$ as a pair of strictly descending paths from the vertex $(1,0)$ to the $n$th level that end at the same vertex. For our construction of the orthonormal basis, we will find it more convenient to work with the matrix elements labelled in this way:
\begin{defin}
We write the matrix elements of $TL_{n}$ as
\[v^{(n)}_{p_{1},p_{2}}\]
where $p_{1},p_{2}$, are paths on the $TL_{\infty}$ Bratteli diagram viewed as follows:
\[
p_{i} = \{p_{1}^{i},p_{2}^{i},...,p_{n}^{i}\}, ~ p_{k}^{i}\in\mathbb{N}, ~ 0\leq p_{k}^{i}\leq\lfloor\frac{k}{2}\rfloor,\]
and satisfying $p_{n}^{1}=p_{n}^{2}$.
\end{defin}
We note the following basic facts about the matrix elements:
\begin{lemma}
\[
\|(v^{(n)}_{p_{1},p_{2}})\|^{2}_{\gamma} = c_{n,p_{n}^{1}}\]
\end{lemma}
\begin{lemma}
For $m\leq n$,
\[
v^{(m)}_{p_{1},p_{2}}v^{(n)}_{p_{3},p_{4}} = \delta_{p_{2},\{p_{1}^{3},...,p_{m}^{3}\}}v^{(n)}_{\{p_{1}^{1},...,p_{m}^{1},p_{m+1}^{3},...,p_{n}^{3}\},p_{4}}\]
\[
v^{(n)}_{p_{1},p_{2}}v^{(m)}_{p_{3},p_{4}} = \delta_{\{p_{1}^{2},...,p_{m}^{2}\},p_{3}}v^{(n)}_{p_{1},\{p_{1}^{4},...,p_{m}^{4},p_{m+1}^{2},...,p_{n}^{2}\}}.\]
In particular,
\[
v^{(n)}_{p_{1},p_{2}}v^{(n)}_{p_{3},p_{4}} = \delta_{p_{2},p_{3}}v^{(n)}_{p_{1},p_{4}}\]
\end{lemma}
In terms of this basis, the involution is just given by
\[
(v^{(n)}_{p_{1},p_{2}})^{\ast} = v^{(n)}_{p_{2},p_{1}}.\]
\begin{defin}
Denote 
\[
b_{n,i}:=\left(\frac{c_{n,i}c_{n-1,i}}{c_{n,i+1}}\right)^{-\frac{1}{2}}.\]
For $n\geq 2$, we define the set 
\[
\mathcal{B}_{n} = \{w^{(n)}_{p_{1},p_{2}}\}\]
as follows:
\[
w^{(n)}_{p_{1},p_{2}} = \begin{cases}
	c_{n,p_{n}^{1}}^{-\frac{1}{2}}v^{(n)}_{p_{1},p_{2}} & n=2\text{ or }p_{n-1}^{1}\neq p_{n-1}^{2}\\
	0 & n>2\text{ even and }p_{n}^{1}=\frac{n}{2}\\
	b_{n,p_{n}^{1}}\left(v^{(n)}_{p_{1},p_{2}}-\frac{c_{n,p_{n}^{1}}}{c_{n,p_{n}^{1}+1}}v^{(n)}_{\{p_{1}^{1},...,p_{n-1}^{1},p_{n}^{1}+1\},\{p_{1}^{2},...,p_{n-1}^{2},p_{n}^{2}+1\}}\right) & \text{ otherwise }.
\end{cases}\]
\end{defin}
\begin{prop}
The set
\[
\mathcal{B}_{\infty}:=\bigcup_{n\geq 2}\mathcal{B}_{n}\]
is an orthonormal basis for $\mathcal{H}_{\gamma}$.
\end{prop}
\begin{proof}
The basis construction can be viewed as taking an orthonormal basis of $TL_{n}$, and adding elements of $TL_{n+1}$ that are orthonormal to this basis. Consider an element $v_{p_{1},p_{2}}^{(n)}\in TL_{n}$, viewed as an element of $TL_{n+1}$, by inclusion it is given as
\[
v_{p_{1},p_{2}}^{(n)} = v_{p_{1}\cup\{p_{n}^{1}\},p_{2}\cup\{p_{n}^{1}\}}^{(n+1)}+ v_{p_{1}\cup\{p_{n}^{1}+1\},p_{2}\cup\{p_{n}^{1}+1\}}^{(n)}.\]	
It follows that neither of the two elements on the right hand side will be orthogonal to the element on the left hand side. Instead consider the element
\[
w:=\mu_{1}v_{p_{1}\cup\{p_{n}^{1}\},p_{2}\cup\{p_{n}^{1}\}}^{(n+1)}-\mu_{2} v_{p_{1}\cup\{p_{n}^{1}+1\},p_{2}\cup\{p_{n}^{1}+1\}}^{(n+1)}.\]
We want to choose coefficients $\mu_{1},\mu_{2}$ so that this element is orthonormal to $v_{p_{1},p_{2}}^{(n)}$. We get
\[
\langle w,v_{p_{1},p_{2}}^{(n)}\rangle_{\gamma} = \mu_{1}c_{n+1,p_{n}^{1}}-\mu_{2}c_{n+1,p_{n}^{1}+1}=0, ~ \langle w,w\rangle = \lvert\mu_{1}\rvert^{2}c_{n+1,p_{n}^{1}}+\lvert\mu_{2}\rvert^{2}c_{n+1,p_{n}^{1}+1}=1.\]
Combining the two, we see that the $\mu$ must satisfy
\[
\mu_{2} = \mu_{1}\frac{c_{n+1,p_{n}^{1}}}{c_{n+1,p_{n}^{1}+1}}, ~ \lvert\mu_{1}\rvert^{2} = \left(c_{n+1,p_{n}^{1}}+\frac{c_{n+1,p_{n}^{1}}^{2}}{c_{n+1,p_{n}^{1}+1}}\right)^{-1} = \left(\frac{c_{n+1,p_{n}^{1}}c_{n,p_{n}^{1}}}{c_{n+1,p_{n}^{1}+1}}\right)^{-1}.\]
For $n>2$, there are two cases where can use $v^{(n)}_{p_{1},p_{2}}$ instead of $w^{(n)}_{p_{1},p_{2}}$: The first is when $p_{n-1}^{1}\neq p^{2}_{n-1}$, as the element in this case will not come from an inclusion of smaller elements. The second case is when $n$ is even and $p^{1}_{n}=\frac{n}{2}$, as in this case we can write the element in terms of 
\[v^{(n-1)}_{p_{1}\setminus \{\frac{n}{2}\},p_{2}\setminus \{\frac{n}{2}\}} ~ \text{ and } ~ w^{(n)}_{(p_{1}\setminus \{\frac{n}{2}\})\cup \{\frac{n}{2}-1\},(p_{2}\setminus \{\frac{n}{2}\})\cup \{\frac{n}{2}-1\}}.\]
\end{proof}	
As every path starts at $(1,0)$, we neglect to write the first vertex of a path. The first few basis vectors of $\mathcal{B}_{\infty}$ are:
\[
\mathcal{B}_{2} = \{c_{2,0}^{\frac{-1}{2}}v_{0,0},c_{2,1}^{-\frac{1}{2}}v_{1,1}\}\]
\[
\mathcal{B}_{3} = \{b_{3,0}(v_{00,00}-\frac{c_{3,0}}{c_{3,1}}v_{01,01}),c_{3,1}^{-\frac{1}{2}}v_{01,11},c_{3,1}^{-\frac{1}{2}}v_{11,01}\}\]
\begin{align*}
\mathcal{B}_{4} =& \{b_{4,0}(v_{000,000}-\frac{c_{4,0}}{c_{4,1}}v_{001,001}),b_{4,1}(v_{011,011}-\frac{c_{4,1}}{c_{4,2}}v_{012,012}),b_{4,1}(v_{111,111}-\frac{c_{4,1}}{c_{4,2}}v_{112,112}),\\
&b_{4,1}(v_{011,111}-\frac{c_{4,1}}{c_{4,2}}v_{012,112}),b_{4,1}(v_{111,011}-\frac{c_{4,1}}{c_{4,2}}v_{112,012}),c_{4,1}^{-\frac{1}{2}}v_{001,011},c_{4,1}^{-\frac{1}{2}}v_{011,001},\\
&c_{4,1}^{-\frac{1}{2}}v_{001,111},c_{4,1}^{-\frac{1}{2}}v_{111,001}\}
\end{align*}
We now want to determine the decomposition of $\mathcal{H}_{\gamma}$.
\begin{thm}
The representation of $TL_{\infty}\otimes TL_{\infty}$ on $\mathcal{H}_{\gamma}$ has no closed invariant subspaces.
\end{thm}
\begin{proof}
Consider the two-sided action of $TL_{\infty}\otimes TL_{\infty}$ on $TL_{\infty}$, it is straightforward to see that $TL_{\infty}$ has a chain of ideals 
\[
TL_{\infty}=\mathcal{I}_{0}\supset \mathcal{I}_{1}\supset\mathcal{I}_{2}...,\]
where $\mathcal{I}_{n}$ consists of the $TL_{\infty}$ diagrams with at least $n$ cups. From \cite{MeTL} it follows that for each $n$,
\[
\hat{\mathcal{I}}_{n}:=\mathcal{I}_{n}/ \mathcal{I}_{n+1}\]
is an irreducible $TL_{\infty}\otimes TL_{\infty}$ representation. Hence we just need to determine whether these irreducibles are closed, which reduces to determining whether they are a direct sum in $\mathcal{H}_{\gamma}$. Consider first $\mathcal{I}_{0}$. If it splits in $\mathcal{H}_{\gamma}$, then there must be an element, which we denote by 
\[f_{\infty}\in\mathcal{H}_{\gamma},\]
such that 
\[
e_{i}f_{\infty} = f_{\infty}e_{i} = 0\]
for all $i\in\mathbb{N}$. Assume such an $f_{\infty}$ exists, then we can approximate it by
\[
x_{n}:= \sum\limits_{\substack{m\leq n\\ w^{(m)}_{p_{1},p_{2}}\in\mathcal{B}_{m}}}\langle w^{(m)}_{p_{1},p_{2}},f_{\infty}\rangle w^{(m)}_{p_{1},p_{2}}.\]
Then 
\[
\langle w^{(m)}_{p_{1},p_{2}},f_{\infty}\rangle = \begin{cases}
	c_{2,0}^{\frac{-1}{2}} & m=2\text{ and }p_{1}=p_{2}=0 \\ b_{m,0} & m\geq 3\text{ and }p_{1}=p_{2} = \{0,...,0\}\\
	0 & \text{ otherwise }
	\end{cases}.\]
It follows that we have 
\[
\| x_{n} \|^{2} = c_{2,0}^{-1}+\sum\limits_{i=3}^{n} b_{i,0}^{2} = c_{n,0}^{-1}.\]
Taking the limit, we then see that 
\[
\lim\limits_{n\rightarrow\infty}\|x_{n}\|^{2} = \infty,\]
and hence $f_{\infty}$ doesn't exist in $\mathcal{H}_{\gamma}$. For other irreducibles to split would require finding an element of the form
\[
e_{1}e_{3}...e_{2i-1}\otimes f_{\infty}\]
in $\mathcal{H}_{\gamma}$, which by multiplicativity of the trace we see can't exist. Hence none of the irreducible components of $TL_{\infty}$ are closed in $\mathcal{H}_{\gamma}$, and therefore $\mathcal{H}_{\gamma}$ contains no closed invariant subspaces.
\end{proof}

\subsection{Generalized Regular Representations}
The decomposition of the regular representation of $TL_{\infty}$ is similar to the infinite symmetric group, whose regular representation is also topologically irreducible. In the case of the infinite symmetric group, alternative representations, called \textit{generalized regular representations}, \cite{Olshanski}, were introduced which had a more interesting decomposition. This raises the question of whether there are other $TL_{\infty}$ representations which have a more interesting decomposition.\\

Recall from Theorem \ref{thm: positive coeffs} that while the traces were positive definite for $0<\gamma\leq\frac{1}{4}$, there were also values of $\gamma$ that could be considered as corresponding to semi-definite traces. Taking a quotient of $TL_{\infty}$ to realize these values did not make sense algebraically, however, motivated by representations studied in \cite{MeTL} we can introduce an alternative construction that does realize these special values:
\begin{defin}
We define a \textbf{generalized infinite Temperley-Lieb diagram} to be a diagram consisting of finitely many through strands and infinitely many cups and caps, such that any through strand joins together points that are finitely far apart.
\end{defin}
We use $s(w)\in\mathbb{N}$ to denote the number of through strands of $w$, and $c(w)\in\mathbb{N}$ to denote the number of cups on one side of $w$.
\begin{defin}
Let $w$ be a generalized infinite Temperley-Lieb diagram. We define
\[
\mathcal{V}(w)\]
to be the representation generated by the two-sided action of $TL_{\infty}\otimes TL_{\infty}$ on $w$.
\end{defin}
We now want to define an inner product on $\mathcal{V}(w)$. Given an element $x\in\mathcal{V}(w)$, denote its restriction to $TL_{n}$ by $\downarrow_{n}x$. For example, taking $w = e_{1}e_{3}e_{5}...$, then $\downarrow_{1}w =1$, $\downarrow_{i}w=e_{1}e_{3}...e_{\lfloor\frac{i}{2}\rfloor}$.
\begin{defin}
We define the inner product $\langle\cdot,\cdot,\rangle_{w,\gamma}$ on $\mathcal{V}(w)$ as follows:
\[
\langle x,y\rangle_{w,\gamma}:=\lim\limits_{n\rightarrow\infty}\frac{\langle\downarrow_{n}x,\downarrow_{n}y\rangle_{\gamma}}{\gamma^{c(\downarrow_{n}w)}}\]
\end{defin}
As we are only considering elements generated from the $TL_{\infty}$ action on $w$ here, then any such $x,y$ will only differ from $w$ at finitely many points. Hence for large enough $n$, the sequence of terms $\langle\downarrow_{n}x,\downarrow_{n}y\rangle_{\gamma}$ will consist of some constant multiplied by increasing powers of $\gamma$. The denominator in the above definition just renormalizes the sequence so that the limit term is finite.
\begin{prop}
If $s(w)=n$, then the inner product $\langle\cdot,\cdot\rangle_{w,\gamma}$ on $\mathcal{V}(w)$ is positive definite when
\[
0<\gamma < \begin{cases}
	\infty & n=0,1 \\ \frac{1}{4}\sec^{2}(\frac{\pi}{n+2}) & n\geq 2
\end{cases} ,\]
and is positive semi-definite for $n\geq 2$ when
\[
\gamma = \frac{1}{4}\sec^{2}(\frac{\pi}{n+2}).\]
\end{prop}
\begin{proof}
As $w$ only has $n$ through strands, the inner product will only depend on coefficients $c_{k,p}$ with $k-2p\leq n$. The result then follows from Theorem \ref{thm: positive coeffs}.	
\end{proof}
\begin{lemma}\label{lemma: equivalent inner products}
For a fixed choice of $w$, the positive definite inner products on $\mathcal{V}(w)$ for different values of $\gamma$ are equivalent.
\end{lemma}
\begin{proof}
Let $\gamma>\gamma^{\prime}$, and denote the coefficients evaluated at $\gamma^{\prime}$ by $c_{k,p}^{\prime}$. Then as we only have to consider coefficients $c_{k,p}$ with $k-2p\leq n$, we only need to choose $a,b>0$ such that
\[
a c_{k,p}\leq c_{k,p^{\prime}}\leq bc_{k,p}.\]
As there are only finitely many coefficients to consider, we see that we can always make such a choice.
\end{proof}
We denote the Hilbert space completion of $\mathcal{V}(w)$ with respect to $\langle\cdot,\cdot\rangle_{w}$ by $\mathcal{H}(w)$.
\begin{defin}\label{defin: generalized standard reps}
Given a generalized infinite Temperley-Lieb diagram $w$ with $s(w)=n$, we define $\mathcal{I}_{k}(w)\subseteq \mathcal{V}(w)$, $0\leq k\leq n$, $k=n\text{ mod }2$, as the subspace of diagrams with at most $k$ through strands. We then define
\[
\mathcal{V}_{k}(w):= \mathcal{I}_{k}/ \mathcal{I}_{k-2}.\]
\end{defin}
\begin{prop}
	Let $w$ have $n$ through strands, and $w^{\prime}$ be given by replacing two adjacent through strands with a cup and cap. Then taking the positive definite inner product on $\mathcal{V}(w^{\prime})$, and the positive semi-definite inner product on $\mathcal{V}(w)$, we have
	\[
	\mathcal{V}(w^{\prime}) \simeq \frac{\mathcal{V}(w)}{\langle\cdot,\cdot\rangle_{w}=0}\]
\end{prop}
Hence for for a fixed choice of $w$, there is up to equivalence a unique positive definite inner product coming from the trace with $\gamma=\frac{1}{4}\sec^{2}(\frac{\pi}{n+2})$. 
\begin{thm}
As a $TL_{\infty}\otimes TL_{\infty}$ representation, 
\[
\mathcal{H}(w)\simeq\bigoplus_{\substack{0\leq k\leq n\\ k=n\text{ mod }2}}\mathcal{H}_{k}(w),\]	
where each of the $\mathcal{H}_{k}(w)$ are closed invariant subspaces in $\mathcal{H}(w)$.
\end{thm}
\begin{proof}
Let $\{f_{n,k}\}$, $0\leq k\leq \lfloor\frac{n}{2}\rfloor$ be the central idempotents for $TL_{n}$. Then replacing the $n$ through strands with each $f_{n,k}$, we can split $\mathcal{H}(w)$ into a direct sum of pairwise orthogonal subspaces that will contain the corresponding $\mathcal{V}_{k}(w)$. It follows from \cite{MeTL} that each of the $\mathcal{V}_{k}(w)$ is irreducible.
\end{proof}

\section{The matrix decomposition of $TL_{n}$ at roots of unity.}\label{section: Matrix decomposition}

Before proceeding to considering traces on $TL_{\infty}$ for $q$ a root of unity, for our purposes we will need to consider a matrix decomposition for $TL_{n}$ when $q$ is a root of unity, as well as a set of minimal idempotents, which appear as the diagonal terms in the decomposition. The first results in relation to this appeared in \cite{GW, MartinBook}, more recent work on projective idempotents corresponding to restricted $\bar{U}_{q}(\mathfrak{sl}_{2})$ appeared in \cite{MeProj, Ibanez, Blanchet}. In what follows, the results can be shown by generalizing the methods of \cite{MeProj}. We first recall the following:

\begin{thm}
The minimal idempotent corresponding to the projective cover of the trivial representation in $TL_{n}$ is given by $f_{n}$ if $n\leq l-1$ or $n=kl-1$, and otherwise for $0\leq i\leq l-2$ can be given by:
\begin{figure}[H]
	\centering
	\includegraphics[width=0.4\linewidth]{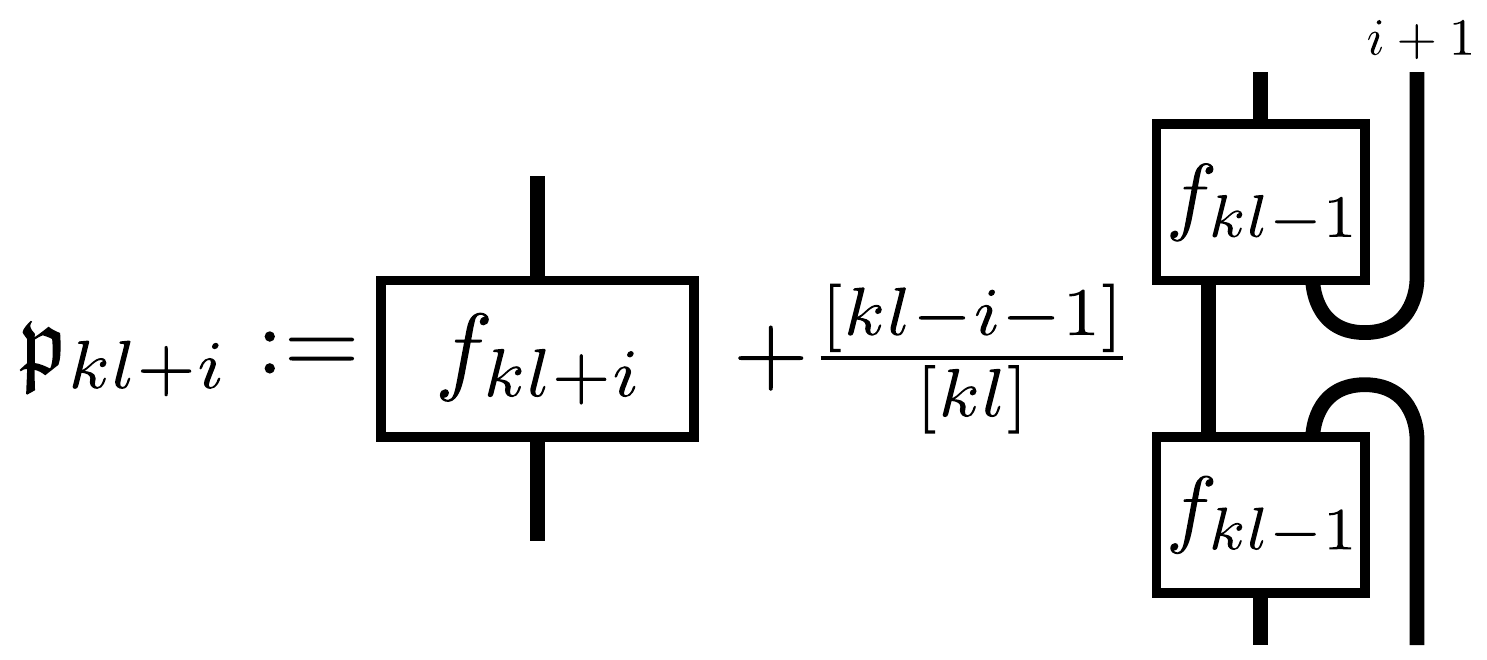}
\end{figure} 	
\end{thm}
Note that in particular, we can simplify to get
\[
\mathfrak{p}_{kl} = f_{kl-1}\otimes 1.\]
To give a matrix decomposition of $TL_{n}$, we generalize the method of the previous section and give fusion rules that the matrix elements can be built out of. These fusion rules will be vectors of half $TL_{n}$ diagrams that describes how a half $TL_{n}$ diagram (viewed as a path on the Bratteli diagram) is included into $TL_{n+1}$. The inclusions correspond to the inclusions of $TL_{n}$ irreducibles. We only need to specify the inclusions at the vertices $(n,0)$ for each $n$, as the inclusion at the $(n,p)$ vertex is given by applying the mapping for the inclusion at the $(n-2p,0)$ vertex. The matrix decomposition given using these fusion rules will be the decomposition of the semisimplification of $TL_{n}$, denoted $TL_{n}^{S}$. We will also give fusion rules, denoted in red, that describe how a fusion rule of $TL^{S}_{n}$ can be included into the Jacobson radical of $TL_{n+1}$. An orthogonal decomposition of the Jacobson radical of $TL_{n}$, denoted $J(TL_{n})$, can then be given by combining the fusion rules for the radical with the fusion rules for the semisimplification. Note that while certain fusion rules in what follows will have complex coefficients, when considering the matrix decomposition, we have to use the transpose $\dagger$ of fusion rule diagrams, not the involution $\ast$.\\ 

We start with the fusion rules for $l=2$. By Frobenius reciprocity, the inclusion rules for irreducibles follow from the restriction rules, which were given in Section \ref{section: TL introduction}. From that, we see there are two cases to consider for $l=2$. When $n$ is odd,
\[
\uparrow\mathcal{L}_{n,0}\simeq \mathcal{L}_{n+1,0}.\]
When $n$ is even
\[
\uparrow\mathcal{L}_{n,0}\simeq \mathcal{L}_{n+1,0}\oplus 2\mathcal{L}_{n+1,1}\oplus\mathcal{L}_{n+1,2}.\]
We can translate this into the corresponding vertices, and state the fusion rules:
\begin{prop}
	For $l=2$ the fusion rules are as follows:
	\begin{itemize}
		\item The fusion rule for $\uparrow(n,0)\simeq (n+1,0)$ for $n$ odd is given by $f_{n}\otimes 1$.
		\item The fusion rule for $\uparrow(n,0)\simeq (n+1,0)\oplus 2(n,1)\oplus(n,2)$ for $n$ even is given by
		\begin{figure}[H]
			\centering
			\includegraphics[width=0.25\linewidth]{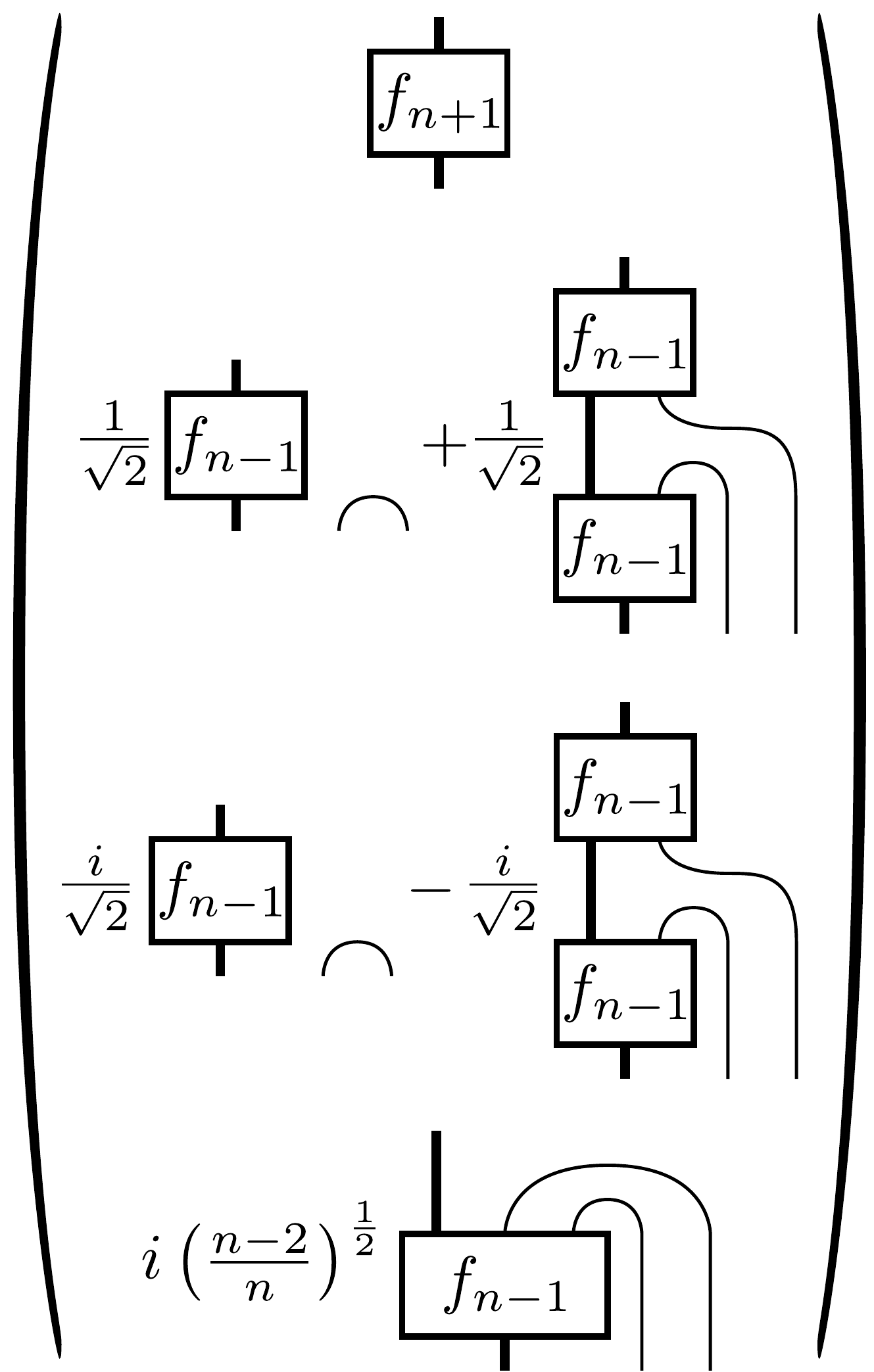}
		\end{figure} 
	\end{itemize}
\end{prop}
Note that for $n=3$, the bottom diagram is zero.\\

The fusion rules for other roots of unity are given by the following:
\begin{prop}
	For $l\geq 3$, the fusion rules are given as follows:
	\begin{itemize}
		\item For  $n<l$, the fusion rule for $\uparrow(n,0)\simeq (n+1,0)\oplus(n+1,1)$ is the same as the generic case.
		\item For $\uparrow(kl,0)\simeq (kl+1,0)\oplus2(kl+1,1)$, the fusion rule is given by 
		\begin{figure}[H]
			\centering
			\includegraphics[width=0.25\linewidth]{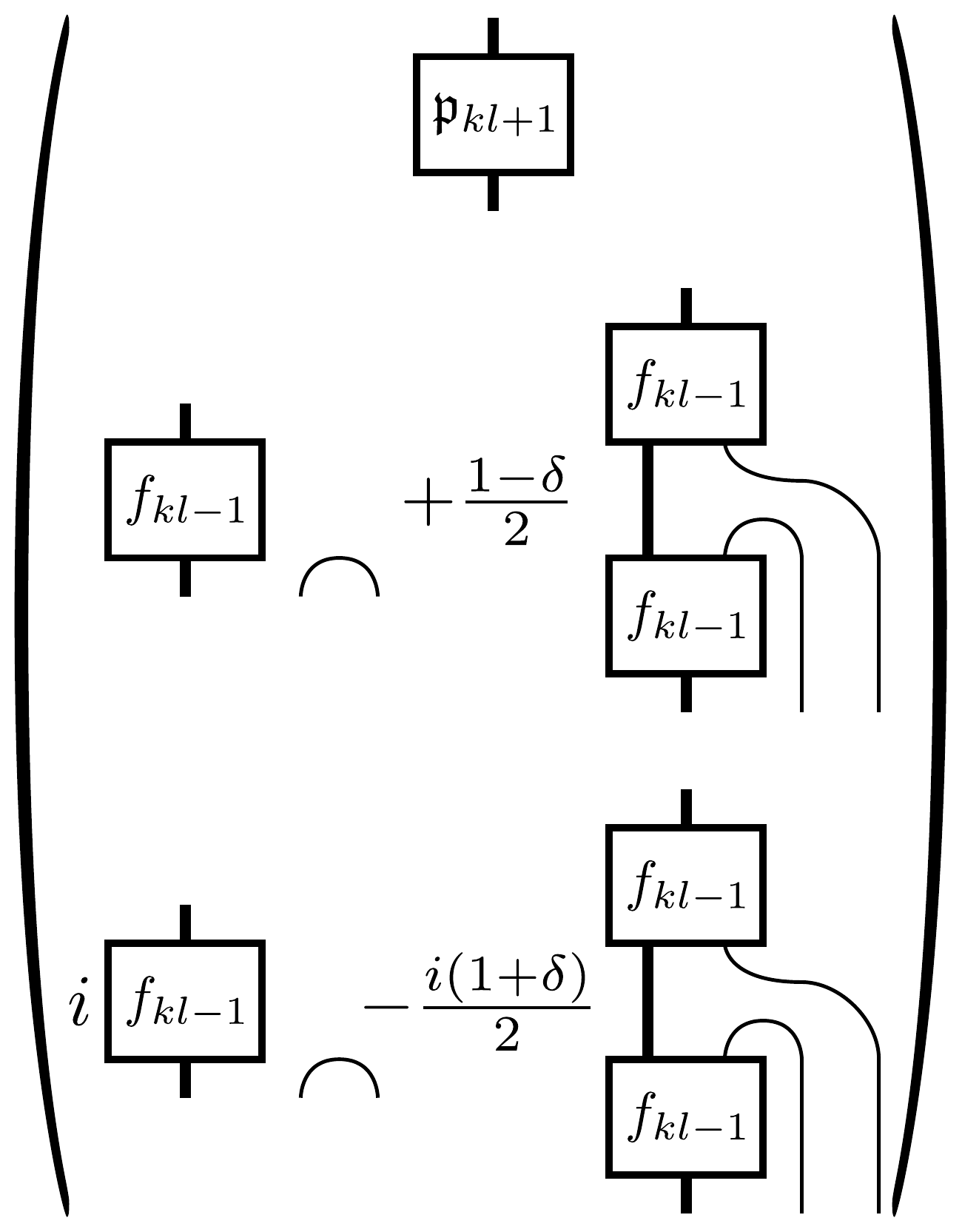}
		\end{figure} 
		\item For $1<i<l-2$, the fusion rule for $\uparrow(kl+i,0)\simeq(kl+i+1,0)\oplus(kl+i+1,1)$ is given by
		\begin{figure}[H]
			\centering
			\includegraphics[width=0.5\linewidth]{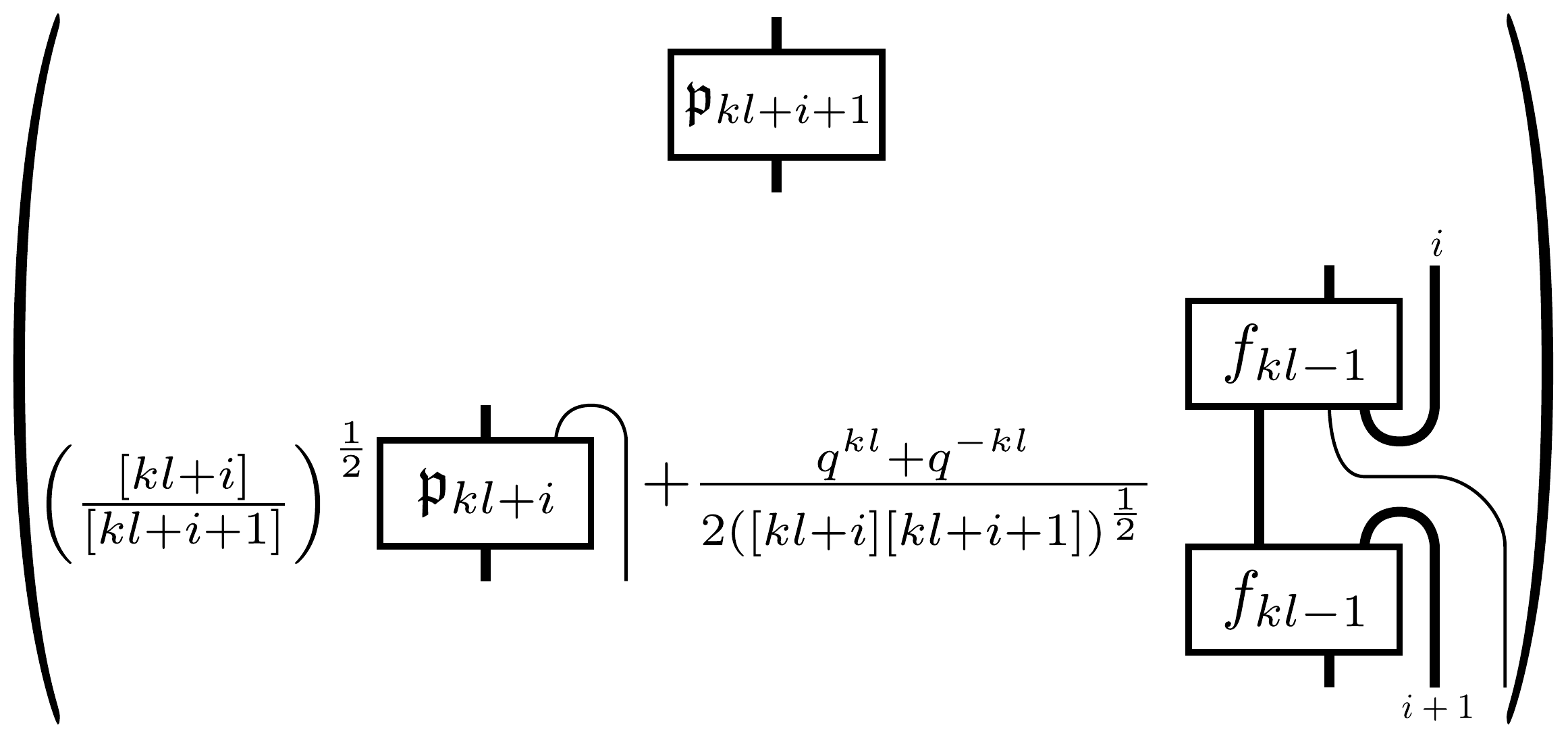}
		\end{figure} 
		\item The fusion rule for $\uparrow(kl-2)\simeq (kl-1,0)\oplus(kl-1,1)\oplus(kl-1,l)$ is given by
		\begin{figure}[H]
			\centering
			\includegraphics[width=0.5\linewidth]{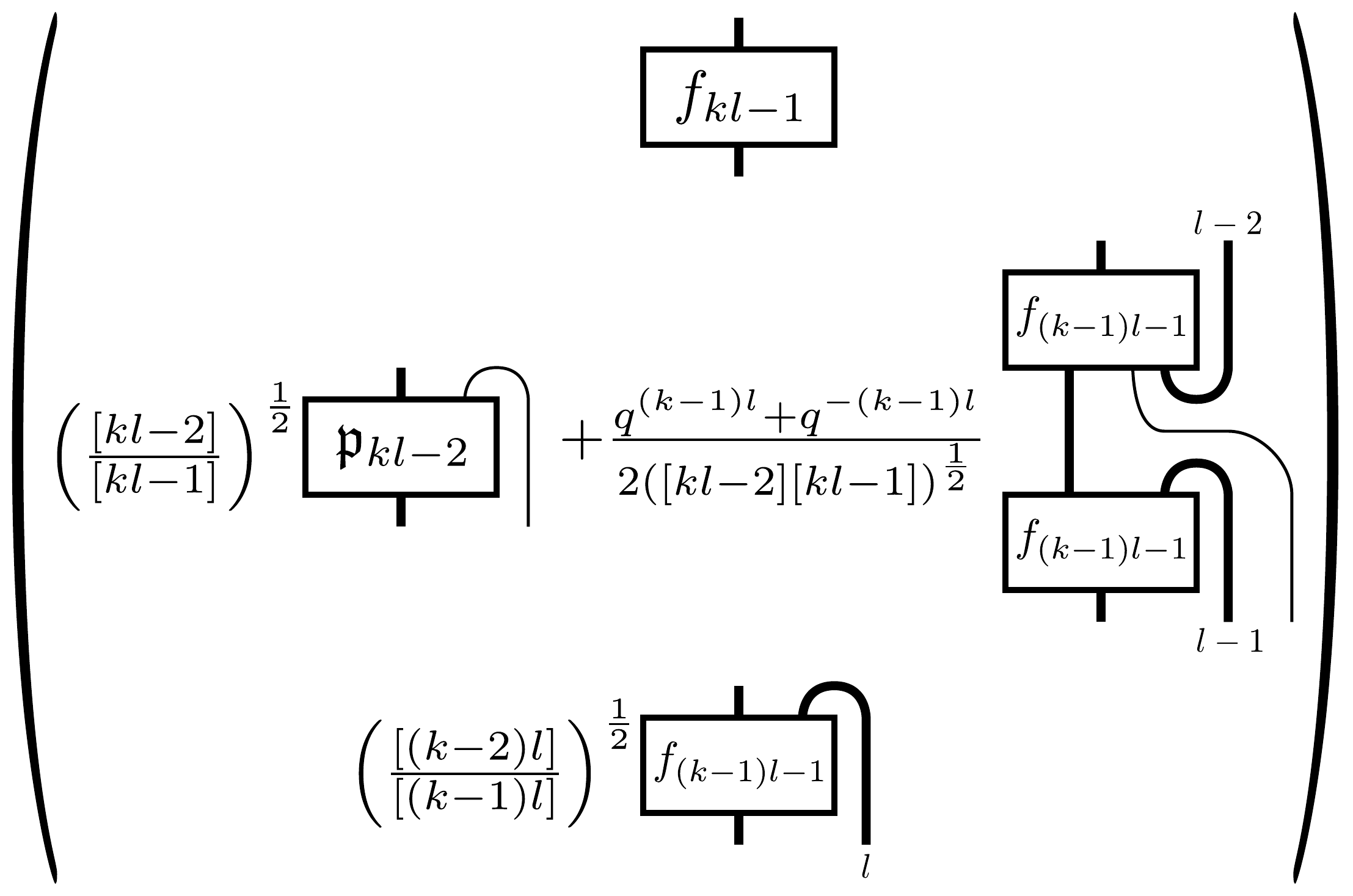}
		\end{figure} 
	\end{itemize}
\end{prop}
We note that in the last fusion rule, the bottom diagram is zero for $k=2$, and that its coefficient can be simplified to give a finite value.\\

The fusion rules corresponding to the Jacobson radical are given as follows:
\begin{prop}\label{prop: jacobson radical decomp}
	Given a diagram $D$ in a $TL_{n}^{S}$ fusion rule  that describes the inclusion into a critical vertex $(n,i)$, the diagram will give a fusion rule of the Jacobson radical for $n+1$ as follows:
	\begin{figure}[H]
		\centering
		\includegraphics[width=0.08\linewidth]{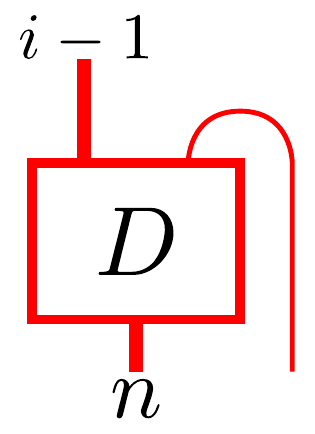}
	\end{figure}
	All other fusion rules for the Jacobson radical are obtained from such diagrams by applying the  $TL_{n}^{S}$ fusion rules for the corresponding vertex, unless the fusion rule would meet a critical line. 
\end{prop}
\begin{proof}
We first consider the fusion rule coming from a critical line diagram. In this case, as $D$ comes from the inclusion into the critical $(n,i)$ vertex, it will only be non-zero on the irreducible $\mathcal{L}_{n,i}$. Then acting the radical rule on $\mathcal{L}_{n+1,i}$, the result would have too many cups, and so give zero. Acting the radical rule on $\mathcal{L}_{n+1,i+1}$, to not have too many caps in the result, we must have a (Jones) partial trace form, however this will multiply the diagram by some $[kl]$ and so the result will again be zero. Hence the fusion rule is zero on all irreducibles, and must correspond to the Jacobson radical.
\begin{figure}[H]
	\centering
	\includegraphics[width=0.12\linewidth]{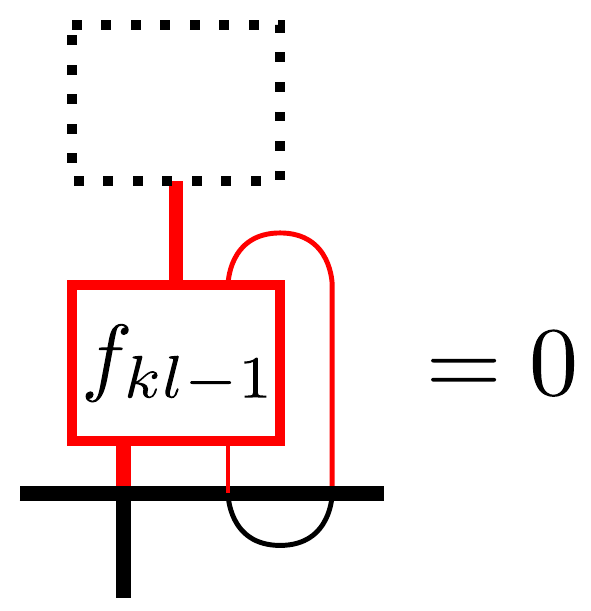}
	\caption{An example of a Jacobson radical fusion rule acting on an irreducible.}
\end{figure}
 For larger fusion rules constructed as described in the proposition, by the same argument we would again have the fusion rule acting as zero on all irreducibles. The only exception is if the resulting fusion rule would meet a critical line. By considering the number of through strands of diagrams, it follows that we can write the fusion rule hitting a critical line in terms of $TL_{n}^{S}$ fusion rules, modulo diagrams with less through strands. Hence the fusion rule hitting the critical line will no longer correspond to the radical.
\begin{figure}[H]
	\centering
	\includegraphics[width=0.3\linewidth]{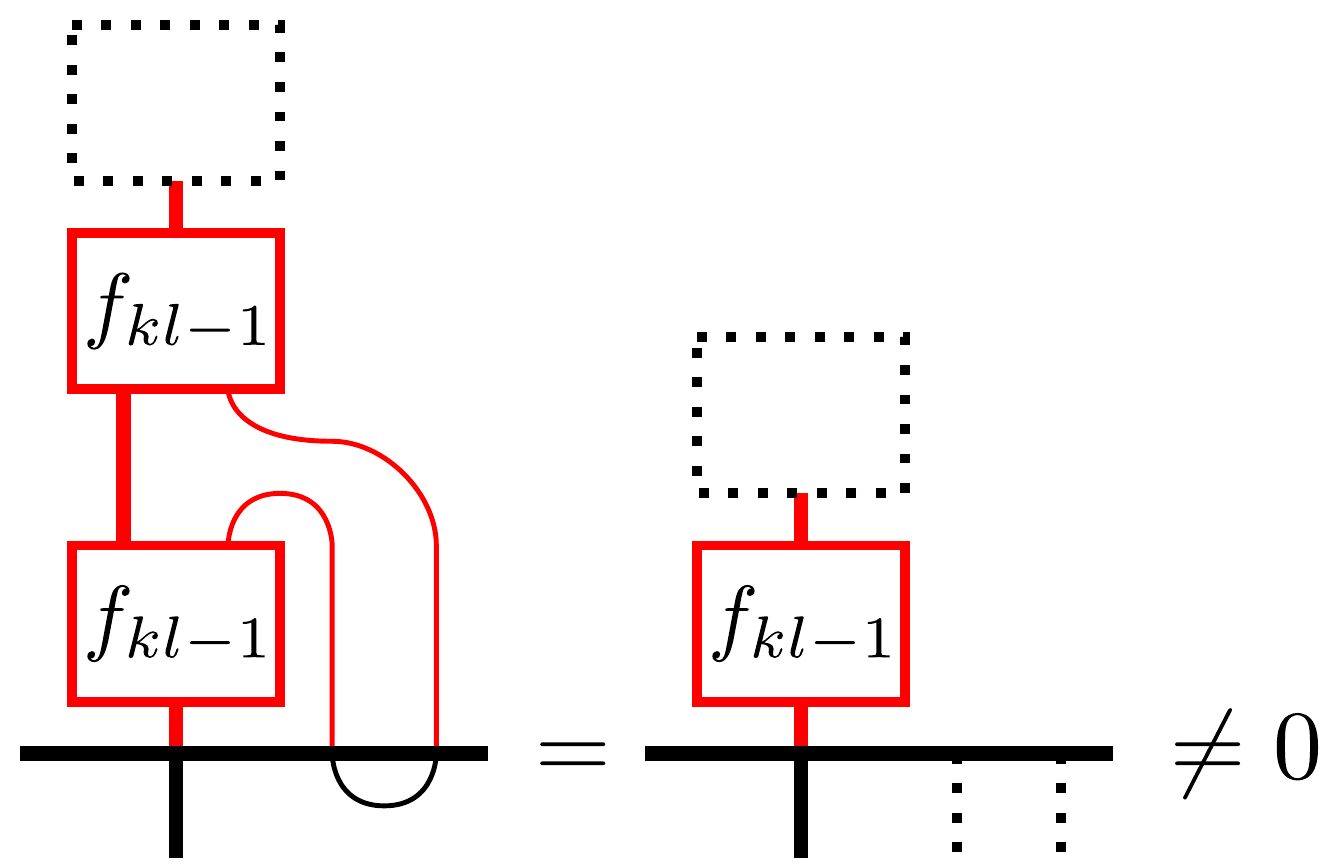}
	\caption{An example of a fusion rule hitting a critical line and no longer corresponding to the Jacobson radical.}
\end{figure} 
As we are constructing the radical fusion rules by essentially following induction paths on the Bratteli diagram of irreducibles, it follows that we must have obtained every element of the Jacobson radical by our method.
\end{proof}
We note that for $l=2$, every possible inclusion of the Jacobson radical fusion rule from $n\rightarrow n+1$ when $n$ is even would hit a critical vertex. Hence the $n\rightarrow n+1$ Jacobson radical fusion rule is zero for $l=2$ and $n$ even, which corresponds to $TL_{n+1}(0)$ being semisimple.\\

As an example, the Jacobson radical fusion rules for $l=4$ and $n=4,5,6$ are given by:
\begin{figure}[H]
	\centering
	\includegraphics[width=0.6\linewidth]{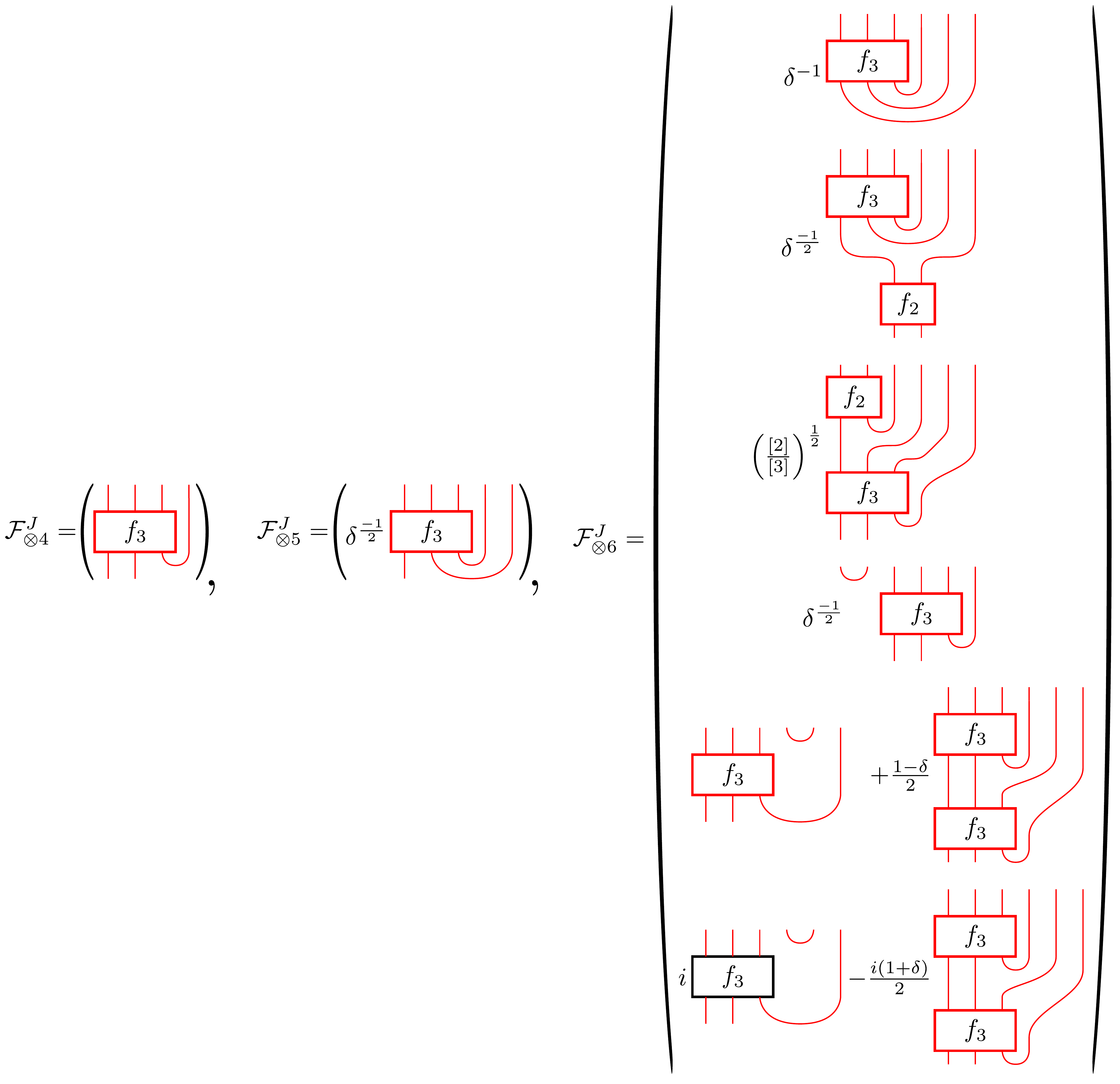}
\end{figure}
In the case of the $n=6$ rules, the first two diagrams come from the inclusion of the $n=5$ rule, and the other four diagrams come from the $(5,1)$ critical vertex.\\

As a more general example, the full fusion rules for $l=2$, $n=4$ are given by
\begin{figure}[H]
	\centering
	\includegraphics[width=0.25\linewidth]{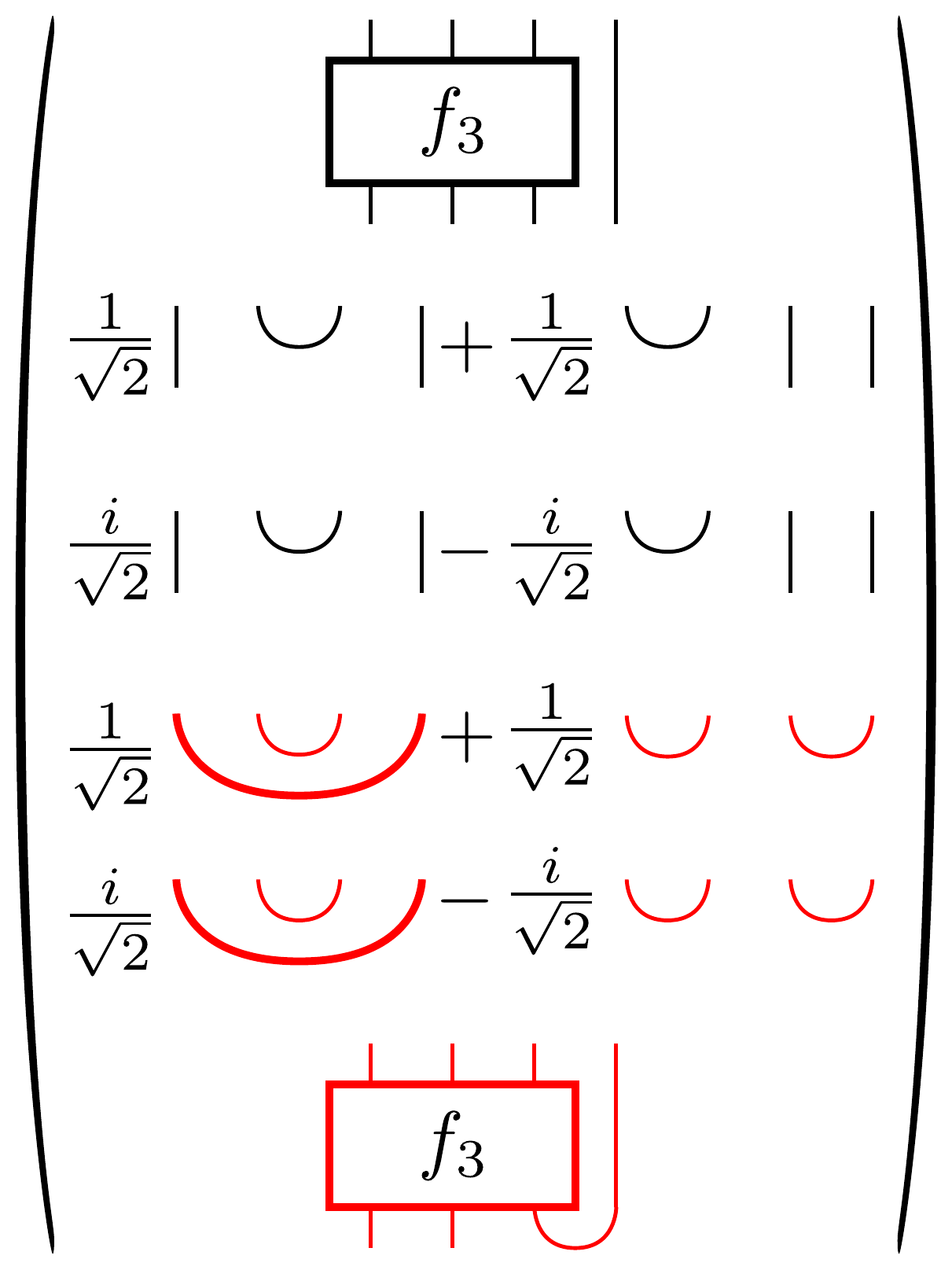}
\end{figure}
Given that
\[
TL_{4}^{S}(0)\simeq \mathbb{C}\oplus M_{2}(\mathbb{C}),\]
from the fusion rules we get a matrix decomposition of $TL_{4}^{S}(0)$ as follows:
\begin{figure}[H]
	\centering
	\includegraphics[width=0.75\linewidth]{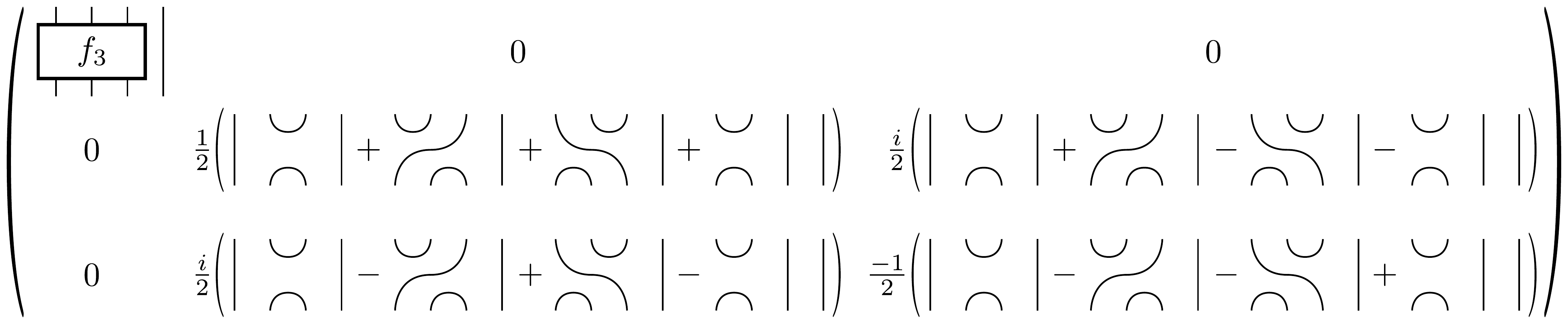}
\end{figure}
Further, we have an orthogonal basis of $J(TL_{4}(0))$ given by:
\begin{figure}[H]
	\centering
	\includegraphics[width=0.95\linewidth]{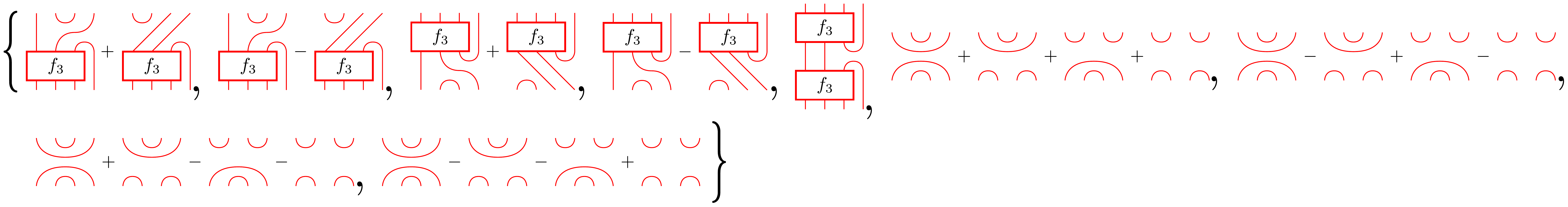}
\end{figure}
\noindent\textit{Remark:} As we had to use the transpose $\dagger$ of fusion rules in the construction of the matrix decomposition and not the involution $\ast$, some non-diagonal matrix elements will not behave as expected under the involution, and instead we will have
\[
e_{ij}^{\ast} = -e_{ji}\]
in certain cases. This can be seen in the above example for $TL_{4}(0)$. More specifically, a matrix element will have involution as above if it was constructed using an odd number of fusion rules with complex coefficients.

\section{Limits of traces for $q$ a root of unity}\label{section: traces rou}

We now proceed to considering traces on $TL_{\infty}$ for $q$ a root of unity. Our starting point is the following description of the traces on $TL_{n}$:
\begin{prop}\label{prop traces on tln}
	The space of traces on $TL_{n}$ is $\lfloor\frac{n}{2}\rfloor$ dimensional.
\end{prop}
\begin{proof}
	For $\delta\neq 0$, from the construction of the radical fusion rules in Proposition \ref{prop: jacobson radical decomp}, it can be seen that the trace of any element in the Jacobson radical will be zero. Hence the possible traces on $TL_{n}$ reduce to the matrix traces on $TL_{n}^{S}$, which must form a basis for all possible traces on $TL_{n}$.\\
	For $\delta=0$, as $TL_{n}$ is semisimple for $n$ odd, it follows from restriction that the number of possible traces on $TL_{2n}$ is either $2n-1$ or $2n$. The first $2n-1$ can be given by matrix traces on $\mathcal{V}_{2n,p}$, $0\leq p<2n$. It can then be verified that there is another trace $t^{0}_{2n,n}$ on $TL_{2n}$ defined (not necessarily uniquely) by
	\[
t^{0}_{2n,n}(e_{1}e_{3}...e_{2n-1})=1.\]
\end{proof}
Whilst extremal traces are generally considered for inclusions of semisimple algebras, the version of the Kerov-Vershik ring theorem we have used can be considered as applying to multiplicative graphs, and forgetting the semisimple algebras underlying it. As the standard Bratteli diagram and corresponding traces still exists in the root of unity case, then assuming positivity of coefficients, the traces will still be extremal under the same conditions as the generic case.

However, if we wanted to state anything about positivity of the inner product formed from the trace in the root of unity case, then using the trace with respect to the standard basis no longer makes sense algebraically, as the minimal projections differ in the generic and root of unity cases. Instead we will want to consider the trace $t_{\infty}$ with respect to the basis of traces on irreducibles $\mathcal{L}_{n,i}$. Our aim from now on is then to take the conditions of extremal traces with respect to the basis of traces of standard representations, and translate it into conditions on the new basis. 

We now consider the basis of matrix traces on the irreducibles $\mathcal{L}_{n,p}$, which we denote by 
\[
l_{n,p}.\]
These new traces can easily be related to the $t_{n,p}$ as follows. When $\mathcal{V}_{n,p}$ is irreducible, we just have
\[
l_{n,p} = t_{n,p}.\]
When $\mathcal{V}_{n,p}$ is indecomposable, with
\[
\mathcal{L}_{n,p}\rightarrow\mathcal{V}_{n,p}\rightarrow\mathcal{L}_{n,p^{\prime}},\]
we have
\[
t_{n,p} = l_{n,p}+l_{n,p^{\prime}}.\]
The restriction rules for the $l_{n,p}$ will be the same as the restriction rules for the $\mathcal{L}_{n,p}$. For the case of $\delta=0$, there is no $\mathcal{L}_{2n,n}$, however by Proposition \ref{prop traces on tln} we know there is an extra trace $t^{0}_{2n,n}$. We define our choice of $t^{0}_{2n,n}$ by 
\[
t^{0}_{2n,n}(e_{1}e_{3}...e_{j})=0\]
for $j<2n-1$ odd. Then it follows that the restriction of $t^{0}_{2n,n}$ to $TL_{n-1}$ is zero, and hence we can neglect it from now on.
\begin{defin}
For $q$ a root of unity with $q^{2l}=1$, we denote a trace on $TL_{\infty}$ by 
\[
\chi_{\infty,l}.\] 
We write the restriction of $\chi_{\infty,l}$ to $TL_{n}$ as
\[
\chi^{(n)}_{\infty,l} = \sum\limits_{i}\lambda_{n,i}l_{n,i},\]
for some coefficients $\lambda_{n,i}\in\mathbb{C}$.
\end{defin}
\begin{prop}\label{prop: extremal rou}
If $\chi_{\infty,l}$ is extremal, then for $l=2$,
\[
\lambda_{2n+1,n} = c_{2n+1,n}.\]
For $l\geq 3$,
\[
\lambda_{2n,n}=c_{2n,n}.\]
\end{prop}
\begin{proof}
As the traces $\{l_{n,i}\}$ are just a change of basis of the traces $\{t_{n,i}\}$, it follows that the extremal traces must satisfy the same conditions as the generic case coefficients. To translate the multiplicativity condition from the generic case into conditions for the coefficients of the $\{l_{n,i}\}$ basis, for $l\geq 3$ we consider the element 
\[
e_{1}e_{3}...e_{2n-1}.\]
It only acts as non-zero on the representations $\mathcal{V}_{2n,n}$, $\mathcal{L}_{2n,n}$. Comparing its trace with respect to the two bases, we have
\[
t_{\infty}(e_{1}e_{3}...e_{2n-1}) = c_{2n,n}t_{2n,n}(e_{1}e_{3}...e_{2n-1}) = c_{2n,n},\]
\[
t_{\infty}(e_{1}e_{3}...e_{2n-1}) = \lambda_{2n,n}l_{2n,n}(e_{1}e_{3}...e_{2n-1}) = \lambda_{2n,n}.\]
To get the condition for $l=2$, we instead consider the element
\[
e_{1}e_{2}e_{3}...e_{2n},\]
whose trace can be seen to only act as non-zero on the representations $\mathcal{V}_{2n+1,n}$ and $\mathcal{L}_{2n+1,n}$. Again comparing $t_{\infty}$ of the element with respect to the two bases, we get
\[
c_{2n+1,n} = \lambda_{2n+1,n}.\]
\end{proof}

The extremal traces on $TL_{\infty}$ for a $q$ a root of unity must then satisfy the following conditions:
\begin{prop}\label{prop: rou coeff conditions}
The coefficients of the extremal traces must satisfy
\begin{align}
	\lambda_{0,0}&=\begin{cases}
		0 & l=2\\
		1 & \text{otherwise }
	\end{cases},& \lambda_{1,0}&=1,& \lambda_{n,i}&=\gamma^{i}\lambda_{n-2i,0}
\end{align}
\begin{align}
	\lambda_{n,0}=\begin{cases}
		c_{n,0} & (n,0) \text{ critical }\\
		c_{n,0}+\gamma^{n+1\text{ Mod }l}c_{n-2(n+1\text{ Mod }l),0} & \text{ otherwise }
	\end{cases}
\end{align}
\end{prop}
\begin{proof}
We first note that from the conditions of Proposition \ref{prop: extremal rou}, combined with the Bratteli diagram repeating through levels, we must have
\[
\lambda_{n,i} = \gamma^{i}\lambda_{n-2i,0}.\]
Hence as with the generic case, we only need to determine the values of $\lambda_{n,0}$. We start by considering the case $\delta=0$ separately.\\

As the traces $t^{0}_{2n,n}$ restrict to zero, their coefficients must be zero in $\chi_{\infty,2}$. It follows that $\lambda_{2,0}=1$. We then need
\[
\lambda_{3,0}+\lambda_{3,1} = 1,\]
with Proposition \ref{prop: extremal rou} giving $\lambda_{3,1}=\gamma$. Then for the higher coefficients, we want
\[
\lambda_{2n+1,0} = \lambda_{2n,0}-2\lambda_{2n,+1,1}-\lambda_{2n+1,2}, ~ \lambda_{2n+2,0} = \lambda_{2n+1,0}.\]
The first condition becomes
\[
\lambda_{2n+1,0} = \lambda_{2n,0}-2\gamma\lambda_{2n-1,0}-\gamma^{2}\lambda_{2n-3,0}\]
As all $(k,p)$ are critical for $k$ odd when $\delta=0$, by induction on both conditions, we get
\[
\lambda_{2n+1,0} = c_{2n,0}+\gamma c_{n-2,0}-2\gamma c_{2n-1,0}-\gamma^{2}c_{2n-3,0}.\]
Recalling that
\[
c_{k,0} = c_{k-1,0}-\gamma c_{k-2,0},\]
this simplifies to give
\[
\lambda_{2n+1,0} = c_{2n+1,0}.\]
For the second condition, we just have
\[
\lambda_{2n+2,0} = \lambda_{2n+1,0} = c_{2n+1,0} = c_{2n+2,0}+\gamma c_{2n,0}.\]

We now assume $l\geq 3$. We first need to write down the relations between the different $\lambda_{n,i}$ coming from induction. Using the decomposition of $\mathcal{L}_{n,p}$ given in Section \ref{section: TL representations} we get the following conditions for the $\lambda_{n,0}:$
\[
\lambda_{n,0} = \begin{cases}
	\lambda_{n-1,0} & (n-1,0)\text{ critical }\\
	\lambda_{n-1,0}-2\gamma \lambda_{n-2,0} & (n-2,0)\text{ critical }\\
	\lambda_{n-1,0}-\gamma \lambda_{n-2,0}-\gamma^{l}\lambda_{n-2l,0} & (n,0)\text{ critical }\\
	\lambda_{n-1,0}-\gamma \lambda_{n-2,0} & \text{ otherwise }
\end{cases}\]
We proceed by induction. For $n=1$, we just have $\lambda_{1,0}=1=c_{1,0}$. As the Bratteli diagram is the same as the generic case until we reach the vertex $(l-1,0)$, we have 
\[
\lambda_{n,0} = c_{n,0}, \text{ for }n\leq l-1.\] 
If $(n-1,0)$ is critical, then $n\text{ Mod }l=0$, so by assumption 
\[
\lambda_{n,0} = \lambda_{n-1,0} = c_{n-1,0}.\]
Rewriting, we have
\[
\lambda_{n,0} = c_{n,0}+\gamma c_{n-2,0} = c_{n,0}+\gamma^{n+1\text{ Mod }l}c_{n-2(n+1\text{ Mod }l),0}.\]
Next, if $(n-2,0)$ is critical, then by assumption
\[
\lambda_{n-1,0} = c_{n-1,0}+\gamma^{n\text{ Mod }l}c_{n-1-2(n\text{ Mod }l),0}, ~ \lambda_{n-2,0} = c_{n-2,0}.\]
So we have
\[
\lambda_{n,0} = c_{n-1,0}+\gamma^{n\text{ Mod }l}c_{n-1-2(n\text{ Mod }l),0}-2\gamma c_{n-2,0}\]
\[
= c_{n-1,0}+\gamma c_{n-3,0}-2\gamma c_{n-2,0} = c_{n,0}+\gamma^{2}c_{n-4,0}\]
\[
= c_{n,0}+\gamma^{n+1\text{ Mod }l}c_{n-2(n+1\text{ Mod }l),0}\]
Next, if $(n,0)$ is critical, by assumption
\[
\lambda_{n-1,0} = c_{n-1,0}+\gamma^{l-1}c_{n+1-2l,0}, ~ 
\lambda_{n-2,0} = c_{n-2,0}+\gamma^{l-2}c_{n+2-2l,0}, ~ 
\lambda_{n-2l,0} = c_{n-2l,0}.\]
So we have
\[
\lambda_{n,0} = c_{n-1,0}+\gamma^{l-1}c_{n+1-2l,0}-\gamma(c_{n-2,0}+\gamma^{l-2}c_{n+2-2l,0})-\gamma^{l}c_{n-2l,0}\]
\[
= c_{n,0}+\gamma^{l-1}(c_{n+1-2l,0}-c_{n+2-2l,0}-\gamma c_{n-2l,0}) = c_{n,0}.\]
Finally, for other $n$, by assumption
\[
\lambda_{n-1,0} = c_{n-1,0}+\gamma^{n\text{ Mod }l}c_{n-1-2(n\text{ Mod }l),0} = c_{n-1,0}+\gamma^{n\text{ Mod }l}c_{n+1-2(n+1\text{ Mod }l),0}\]
\[
\lambda_{n-2,0} = c_{n-2,0}+\gamma^{n-1\text{ Mod }l}c_{n-2-2(n-1\text{ Mod }l,0)} = c_{n-2,0}+\gamma^{n-1\text{ Mod }l}c_{n+2-2(n+1\text{ Mod }l,0)}\]
which gives
\[
\lambda_{n,0} = c_{n-1,0}+\gamma^{n\text{ Mod }l}c_{n+1-2(n+1\text{ Mod }l),0}-\gamma(c_{n-2,0}+\gamma^{n-1\text{ Mod }l}c_{n+2-2(n+1\text{ Mod }l,0)})\]
\[
= c_{n,0}+\gamma^{n\text{ Mod }l}(c_{n+1-2(n+1\text{ Mod }l),0}-c_{n+2-2(n+1\text{ Mod }l,0)})\]
\[
= c_{n,0}+\gamma^{n+1\text{ Mod }l}c_{n-2(n+1\text{ Mod }l),0}\]
\end{proof}
The following will make clear our reasoning for the choice of defining $\chi_{\infty,l}$ in terms of $l_{n,p}$:
\begin{prop}\label{prop: coeffs of minimal idempotents rou}
	Let $\sum\limits_{i,j}p^{(l)}_{n,i,j}=1$ be a set of minimal idempotents for $TL_{n}$ at the corresponding $l$th root of unity. Where we have labelled the idempotents so that $TL_{n}p_{n,i,j}$ is the indecomposable projective cover of $\mathcal{L}_{n,i}$. Then
	\[
	\chi_{\infty,l}(p^{(l)}_{n,i,j}) = \lambda_{n,i}\]
\end{prop}
\begin{proof}
As the trace of radical elements is zero, it follows that the trace of $p_{n,i,j}^{(l)}$ is equal to the trace of the corresponding idempotent in $TL_{n}^{S}$. As we constructed $\chi_{\infty,l}$ using the irreducible Bratteli diagram, it follows that the idempotents in $TL_{n}^{S}$ will have trace equal to $\lambda_{n,i}$, and hence the result holds for the $TL_{n}$ idempotents.
\end{proof}
Whilst extremality of traces follows from the generic case for $c_{n,0}>0$, i.e. when $0<\gamma\leq\frac{1}{4}$, from the above we see that for positivity of the inner product coming from the trace, we will instead want to determine when $\lambda_{n,0}\geq 0$. Recall we denote
\[
s_{k}:= \frac{1}{4}\sec(\frac{\pi}{k})^{2}\]
Then we have:
\begin{prop}\label{prop: positive coeffs rou}
	The possible positive coefficients for $\chi_{\infty,l}$ are as follows:
	\begin{itemize} 
		\item $\lambda_{n,0}>0$ for all $n$ if $0<\gamma\leq \frac{1}{4}$.
		\item For $l=2$, $\lambda_{i,0}>0$ for $i<n$ and $\lambda_{n,0}=0$ if $n$ is odd and $\gamma=s_{n+1}$.
		\item For $l\geq 3$, $\lambda_{i,0}>0$ for $i<n$ and $\lambda_{n,0}=0$ if $n\leq l-1$ and $\gamma=s_{n+1}$. 
		\item For $l\geq 3$, $\lambda_{i,0}>0$ for $i<n$ and $\lambda_{n,0}=0$ if $n=kl-1$ and $\gamma=s_{kl}$ for $k\geq 2$.
		\item For $l\geq 3$, $\lambda_{i,0}>0$ for $i<n$ and $\lambda_{n,0}=0$ if $l+\lfloor\frac{l}{2}\rfloor\leq n\leq 2l-2$ and $\gamma=s_{2(n+1\text{ Mod }l)}$. 
	\end{itemize}
\end{prop}
\begin{proof}
	As $\lambda_{n,0}$ can be written in terms of $c_{i,j}$, the first three cases follow from the generic case of Theorem \ref{thm: positive coeffs}. Hence we only need to check the final two cases. Recall that we can write
	\[
	c_{n,0} = \frac{[n+1]_{r}}{([2]_{r})^{n}}, ~ \gamma = ([2]_{r})^{-2}.\]
	From Proposition \ref{prop: rou coeff conditions}, we can then write
	\[
	\lambda_{n,0} = \frac{[n+1]_{r}+[n+1-2j]_{r}}{([2]_{r})^{n}}\]	
	for $n\neq kl-1$, where we take $j:=n+1\text{ Mod }l$. We can further rewrite
	\[
	[n+1]_{r}+[n+1-2j]_{r} = \frac{(r^{n-j+1}-r^{j-n-1})(r^{j}+r^{-j})}{r-r^{-1}}\]
	Hence for $\lambda_{n,0}=0$, we require 
	\[
	r^{2n-2j+2}=1 \text{ or } r^{2j}=-1.\]
	However as $j=n+1\text{ Mod }l$, the first case reduces to $r^{kl}=1$, so we would also have $\lambda_{n-j+1,0}=0$. It follows that putting $r^{kl}=1$, via $\gamma=s_{kl}$, we get
	\[
	\lambda_{n,0}=0 \text{ for } kl-1\leq n<kl+l-1.\]
	If we instead consider 
	\[
	r^{2j}=-1, \text{ via } \gamma=s_{2(n+1\text{ mod }l)},\] 
	we will have $\lambda_{k,0}=0$ for all $k>l$ such that $k=n\text{ mod }l$. By the previous condition, we then only want to consider this case for $n<2l-1$. Further if 
	\[
	n<l+\lfloor\frac{l}{2}\rfloor, \text{ then } 2(n+1\text{ mod }l)<l,\]
	and so we would have some $\lambda_{k,0}=0$ for $k<n$. Hence we only consider 
	\[
	\gamma=s_{2(n+1\text{ mod }l)} \text{ for } l+\lfloor\frac{l}{2}\rfloor\leq n\leq 2l-2.\] 
	It can be seen that we have covered all possible $n$, and hence these are all the possible conditions for $\gamma$.
\end{proof}

\section{The Inner Product at roots of unity.}\label{section: inner product rou}

To determine which extremal traces are positive, we want to consider the inner product on $TL_{\infty}$ defined by 
\[
\langle x,y\rangle_{\chi} := \chi_{\infty,l}(xy^{\ast})\]
as we did in the generic case. However at roots of unity the situation becomes more complicated. We say that an inner product on a space $V$ is called \textit{indefinite} \cite{GLR} if the set
\[
\{x\in V: \| x\|^{2}=0\}\]
is non-empty and not linearly closed. We note that standard definitions require the inner product to be non-degenerate. However we will neglect this at the moment, as we will instead show later that it is non-degenerate when considered on $TL_{\infty}$. We then have the following:
\begin{thm}
	The inner product defined on $TL_{\infty}$ via $\chi_{\infty,l}, \ast$ is indefinite if $\gamma\neq\delta^{-2}$.
\end{thm}
\begin{proof}
	If $\gamma=\delta^{-2}$, then it is the Jones trace, and so from \cite{Jones1} we have a semi-definite inner product on $TL_{\infty}$. We now assume $\gamma\neq\delta^{-2}$. The case $\delta=0$ is easy to see; 
	\[
	\|e_{1}\|^{2}=\|e_{2}\|^{2}=\delta^{2}\gamma=0,\]
	but as $\chi_{\infty,l}(e_{1}e_{2})=\gamma$, we have
	\[
	\|e_{1}+e_{2}\|^{2}=2\gamma.\]
	For $\delta\neq 0$, consider the following elements:
	\begin{figure}[H]
		\centering
		\includegraphics[width=0.3\linewidth]{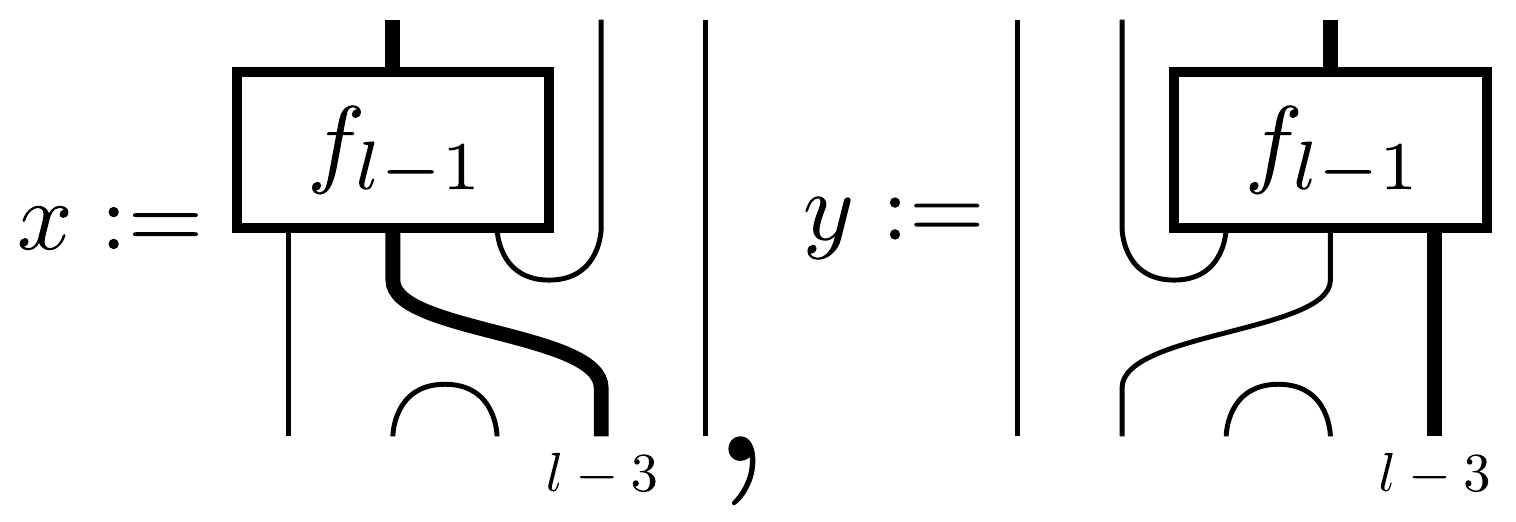}
	\end{figure}
	By a combination of cyclicity of the trace, that $f_{n}$ is an idempotent, and that $e_{i}f_{n}=0$, we have
	\[
	\| x\|^{2} = \|y\|^{2} =0.\]
	However we have
	\begin{figure}[H]
		\centering
		\includegraphics[width=0.3\linewidth]{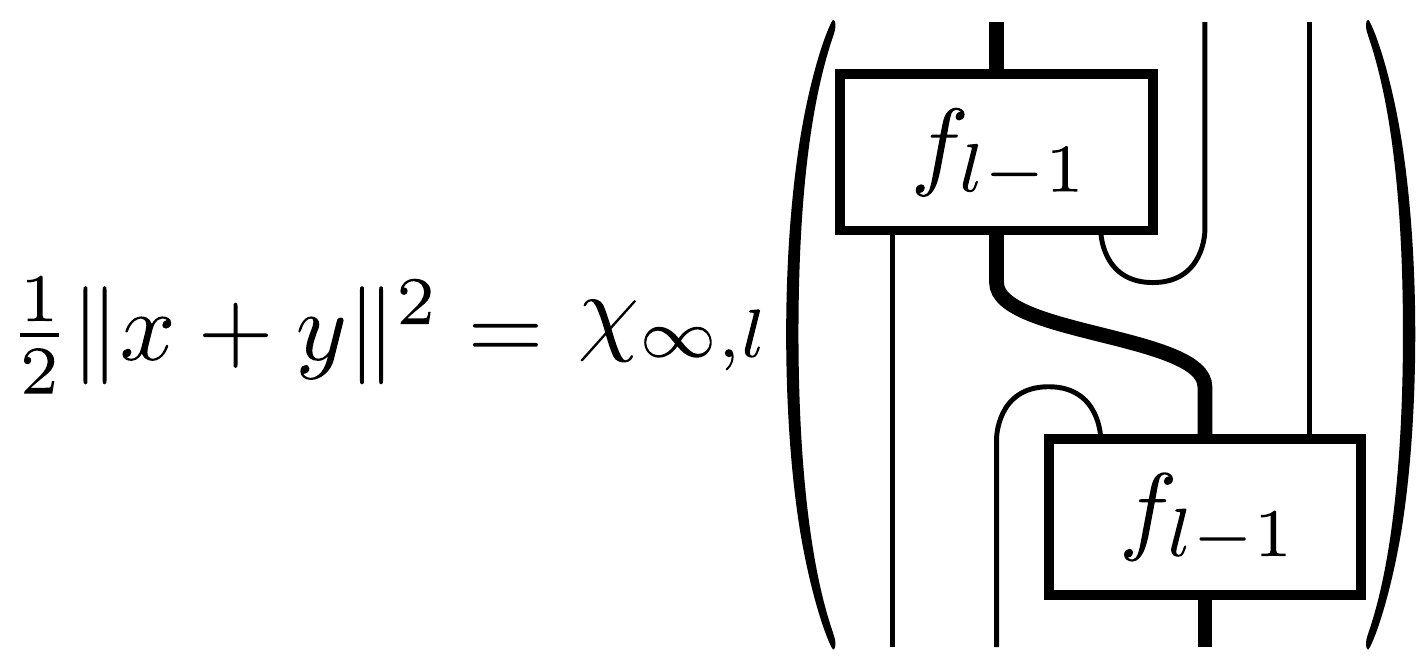}
	\end{figure}
We then want to show that this is non-zero. Consider $\chi$ in terms of the traces on $\mathcal{V}_{l+1,i}$, i.e.
\[
\chi^{(l+1)}_{\infty,l} = \sum\limits_{i}c_{l+1,i}t_{l+1,i}.\]
Just by considering cup numbers and the $f_{l-1}$ action, we see the only possible non-zero term will be $t_{l+1,1}$. Further, the only possible element in $\mathcal{V}_{l+1,1}$ that can contribute to a non-zero trace is
\begin{figure}[H]
	\centering
	\includegraphics[width=0.1\linewidth]{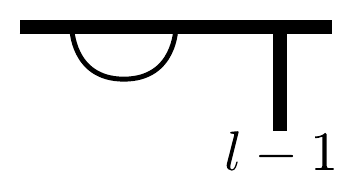}
\end{figure}
For the bottom copy of $f_{l-1}$, we then only need to consider the identity element acting. For the top copy of $f_{l-1}$, we need to consider which terms will act to return the same element of $\mathcal{V}_{l-1,1}$. It then follows that the value of
\begin{figure}[H]
	\centering
	\includegraphics[width=0.2\linewidth]{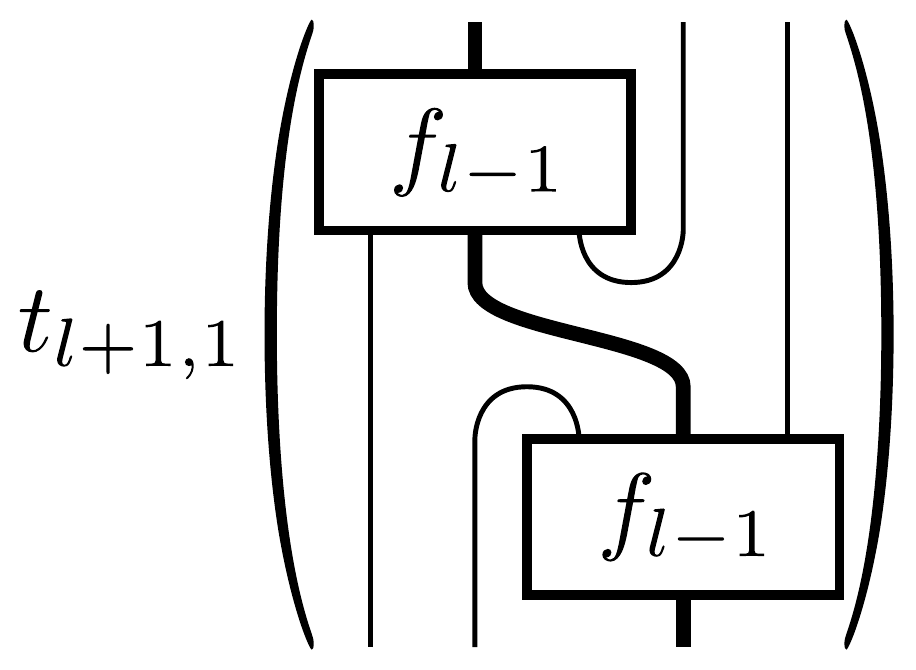}
\end{figure}
is given by the coefficient of the diagram
\begin{figure}[H]
	\centering
	\includegraphics[width=0.08\linewidth]{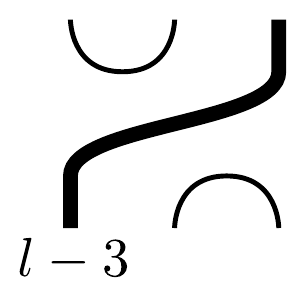}
\end{figure}
in $f_{l-1}$. From Proposition 3.3 of \cite{Morrison}, we see this coefficient is $\frac{(-1)^{l}}{[l-1]}$. Hence we have
\[
\|x+y\|^{2} = \frac{2(-1)^{l}}{[l-1]}c_{l+1,1}.\]
This is only zero if $\gamma = \delta^{-2}$, hence in general the inner product is indefinite.
\end{proof}

\noindent \textit{Non-degenerate} indefinite inner products can be thought of as coming from a positive definite inner product, twisted by an invertible Hermitian operator. However, for a given $n$, the inner product $\langle\cdot,\cdot\rangle_{\chi}$ restricts to a degenerate inner product, so it is not immediately obvious whether of not such an operator exists in our case. In what follows, we will take a slightly different point of view, and instead of considering such an operator, we will consider an alternative "corrected" involution. Recall from the decomposition of $TL^{S}_{n}$ given in Section \ref{section: Matrix decomposition} that for certain matrix elements we have
\[
e_{ij}^{\ast} = -e_{ji}\]
under the standard involution by reflection of diagrams. This in turn would cause the inner product restricted to $TL_{n}^{S}$ to be indefinite, so we can define a "corrected" involution that gives a positive definite inner product on $TL_{n}^{S}$ by correcting the sign in these cases:
\begin{defin}
Given the matrix decomposition of $TL_{n}^{S}$ with respect to $\dagger$, we define the involution $\diamond$ on $TL_{n}^{S}$ to be the involution such that
\[
e_{ij}^{\diamond} = e_{ji}\]
for all $i,j$.
\end{defin}
\begin{defin}
We denote by $(\cdot,\cdot)_{\gamma,l}$ the inner product defined by 
\[
(x,y)_{\gamma,l}:= \chi_{\infty,l}(xy^{\diamond}).\]
We denote the norm
\[
\| x\|^{2}_{\diamond,\infty,l}:=\chi_{\infty,l}(xx^{\diamond}).\]
\end{defin}
By our method of constructing the matrix basis of $TL_{n}^{S}$ via inclusion of fusion rules, it is clear that the corrected involution $\diamond$ on $TL_{n}^{S}$ will extend to an involution on $\bigcup\limits_{n}TL_{n}^{S}$. Hence to consider an inner product on $TL_{\infty}$ with respect to $\diamond$, then we need to determine whether this involution extends to elements of $J(TL_{n})$. Our aim is to obtain such an extension of the involution and inner product to $TL_{\infty}$, and in doing so obtain the following result:
\begin{thm}\label{thm: fixed inner product}
With respect to the involution $\diamond$, the positive extremal traces on $TL_{\infty}$ for $q$ a root of unity are given by
\[
0<\gamma\leq\frac{1}{4}.\]
\end{thm}
We will split the proof into a number of parts. We begin with the case $\delta=0$:
\begin{proof}
In this case, as $TL_{n}(0)$ is semisimple for all $n$ odd, given $x\in J(TL_{2n})$, we can view $x$ as an element of $TL_{2n+1}$ to find its matrix decomposition. It follows that $\diamond$ extends as an involution to $TL_{\infty}(0)$, then given Propositions \ref{prop: coeffs of minimal idempotents rou} and \ref{prop: positive coeffs rou} the resulting inner product will be positive definite.
\end{proof}
We now assume that $\delta\neq 0$. In this case we will not be able to write elements of $J(TL_{n})$ as matrix elements of $TL_{m}^{S}$ for some $m>n$. Instead we will write them as some infinite sum in $\bigcup\limits_{n}TL_{n}^{S}$. The remainder of the proof of Theorem \ref{thm: fixed inner product} will then be broken down into two parts: In the first part, we will show how to write an element of $J(TL_{n})$ as such an infinite sum. In the second part, we will show that the inner product converges and is positive definite.\\

\subsection{Writing elements of $J(TL_{n})$ in terms of $\bigcup\limits_{n}TL_{n}^{S}$.}
Before we explain the general construction, we want to give an example of the method. For $l=3$, consider the element of $J(TL_{3})$ as follows:
\begin{figure}[H]
	\centering
	\includegraphics[width=0.05\linewidth]{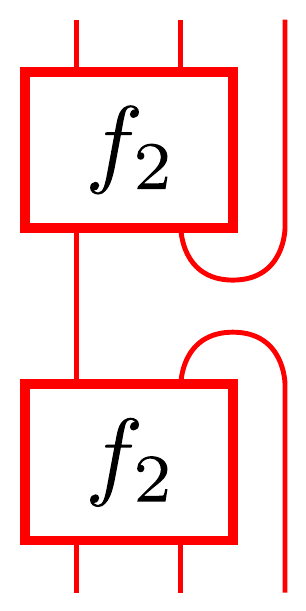}
\end{figure}
We want to write it in terms of $\bigcup\limits_{n}TL_{n}^{S}$. Considering the radical element as a pair of paths on the Bratteli diagram, it corresponds to two copies of the path
\[(1,0)\rightarrow(2,0)\rightarrow(3,1).\]
We see that we can extend this path by
\[
(3,1)\rightarrow (4,1)\]
to hit a critical line. Now consider the partition of unity for $TL_{2}(1)$, it is given by
\[
\{f_{2},\delta^{-1}e_{1}\}.\]
If we view the radical element as an element of $TL_{4}$, and insert the partition of unity into the middle of the diagram, we get the following:
\begin{figure}[H]
	\centering
	\includegraphics[width=0.3\linewidth]{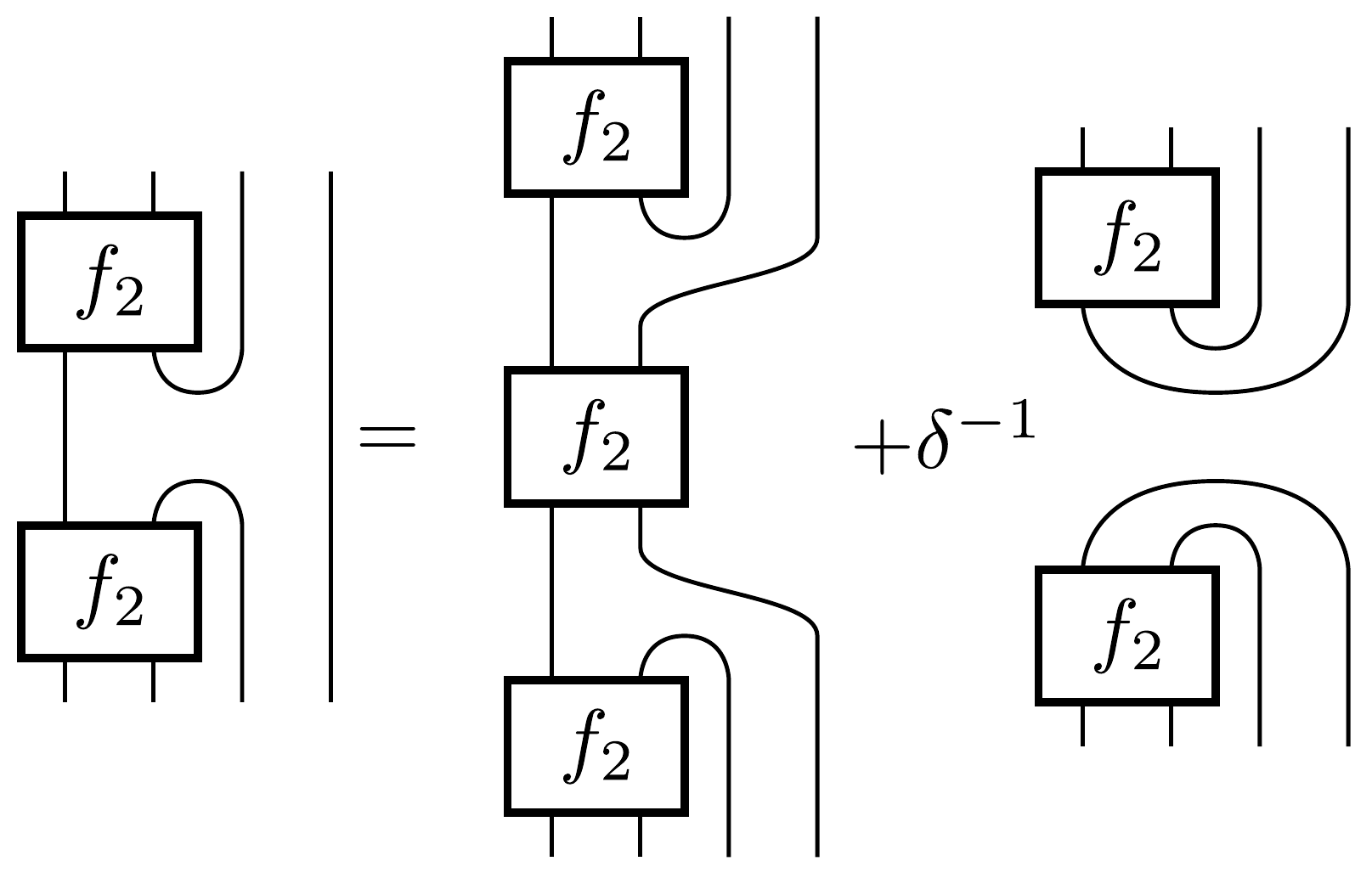}
\end{figure}
The first diagram on the right hand side is in $TL_{4}^{S}$, and the second diagram is in $J(TL_{4})$. In terms of the Bratteli diagram, the $TL_{4}^{S}$ element has been obtained as follows:
\begin{figure}[H]
	\centering
	\includegraphics[width=0.3\linewidth]{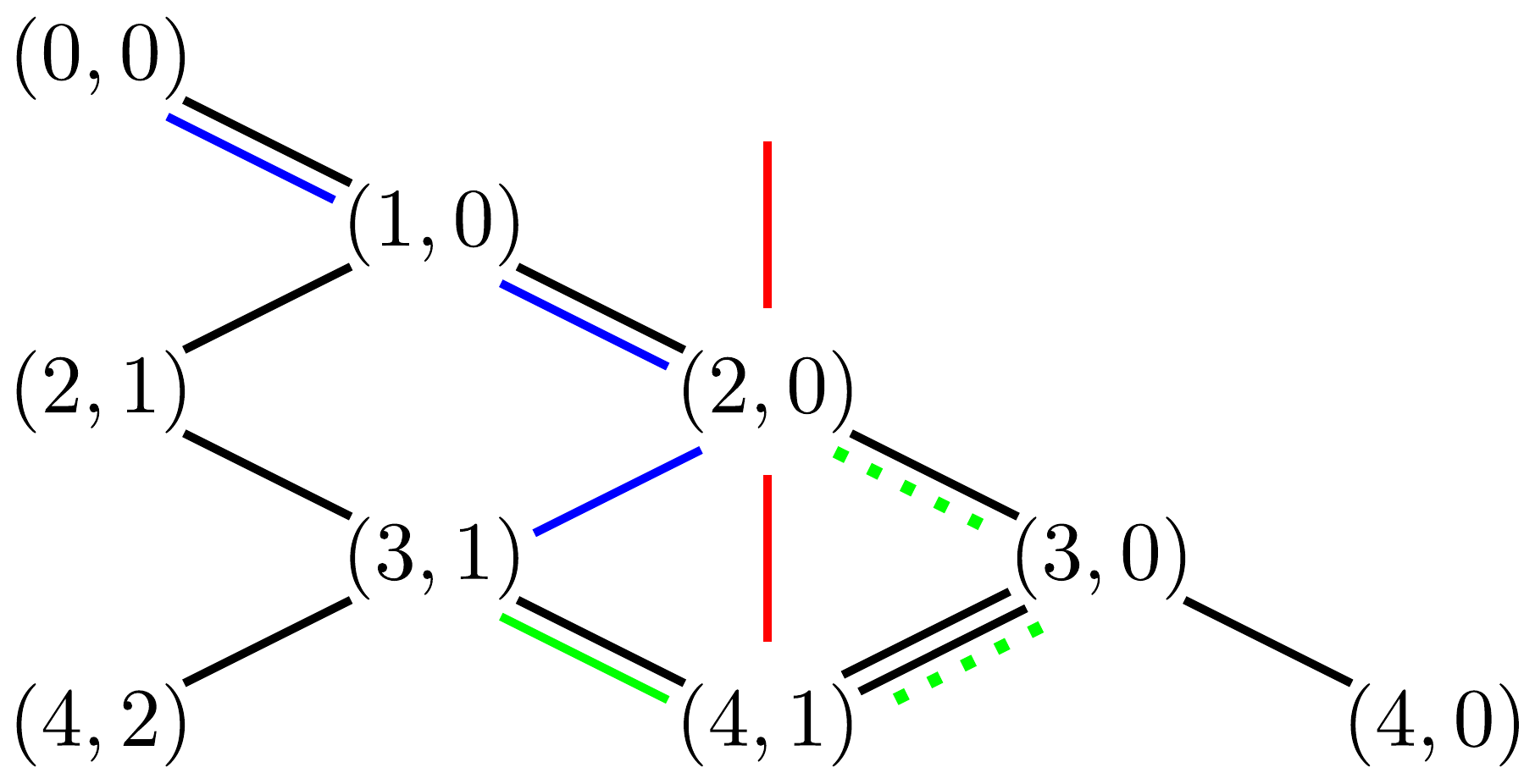}
\end{figure}
where the blue path is the original radical element, the green path is its extension to hit the critical line, and the dotted green path is the resulting $TL_{4}^{S}$ element. Note that the dotted green path is non-unique, as we pass through the double edge. Denoting the two paths through this edge by $v_{+}$ and $v_{-}$, by direct calculation we get that the $TL_{4}^{S}$ element is given by
\[
v^{(4)}_{+,+}+iv^{(4)}_{+,-}+iv^{(4)}_{-,+}-v^{(4)}_{-,-}.\]
The off-diagonal paths have involutions
\[
(v^{(4)}_{+,-})^{\ast} = -v^{(4)}_{-,+}.\]
Hence, using the corrected involution, $\diamond$, we get
\[
\chi_{\infty,3}\left((v_{+,+}+iv_{+,-}+iv_{-,+}-v_{-,-})(v^{\diamond}_{+,+}-iv^{\diamond}_{+,-}-iv^{\diamond}_{-,+}-v^{\diamond}_{-,-})\right) = 4c_{4,1}.\]
Now considering the $J(TL_{4})$ element we obtained, we see it corresponds to two copies of the path
\[(1,0)\rightarrow(2,0)\rightarrow(3,1)\rightarrow(4,2).\]
We can extend it to hit the critical vertex $(6,2)$. Hence viewing it as an element of $TL_{6}$, and applying the $TL_{2}$ partition of unity, then combining with the $TL_{4}^{S}$ diagram, we find that we have written our original radical element as follows:
\begin{figure}[H]
	\centering
	\includegraphics[width=0.6\linewidth]{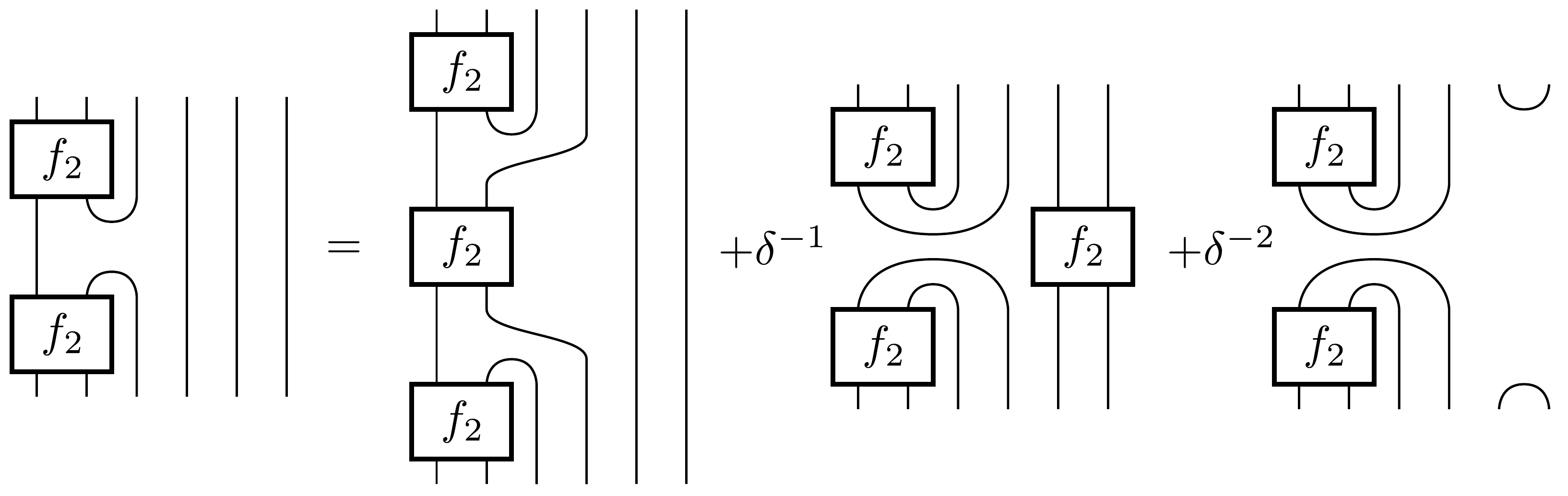}
\end{figure}
It is clear that we can repeat this construction to obtain a sequence of $\bigcup TL_{n}^{S}$ elements that sum to our original $J(TL_{3})$ element. As each of these terms correspond to paths passing through critical vertices at different levels, the inner product between different terms will be zero. Then considering the corrected norm of each of these terms, we get the corrected norm of the $J(TL_{3})$ element is
\[
4\sum\limits_{i=0}^{\infty}c_{4+2i,i+1} = 4\sum\limits_{i=0}^{\infty}\gamma^{i+1}c_{2,0} = 4c_{2,0}\frac{\gamma}{1-\gamma} = 4\gamma.\]
Hence the norm of this element is defined for 
\[
0\leq\gamma<1.\]

We now want to describe the general construction:\\

\begin{defin}
Given $x\in J(TL_{n})$, we associate with it an element $x^{\prime}\in TL_{n^{\prime}}^{S}$ for some $n^{\prime}>n$ as follows:\\
Take $x$ to be a basis element of the orthogonal decomposition of $J(TL_{n})$ given in Section \ref{section: Matrix decomposition}. Then view $x$ as a pair of paths on the Bratteli diagram, and let $(n_{0},p_{0})$ be the end point of each of the paths. Extending downwards from this vertex and to the right, we will hit a critical vertex
\[
(n_{1},p_{1}).\]
Take the partition of unity of $TL_{n_{1}-2p_{1}}$ and apply it to the middle of the through strands of $x$ (viewed as an element of $TL_{n_{1}})$. We denote the term given by applying $f_{n_{1}-2p_{1}}$ to $x$ by $x_{1}$.
\end{defin}
\begin{lemma}
Choose a path on the Bratteli diagram corresponding to the Jacobson radical, and consider the element $x=x^{\dagger}\in J(TL_{n})$ described by two pairs of the path. Using the above notation, $x_{1}\in TL_{n_{1}}^{S}$, $x-x_{1}\in J(TL_{n_{1}})$. Further, any term in $x-x_{1}$ will have the endpoint of its paths in the strip of the Bratteli diagram bounded by the critical lines $(n_{1}-2p_{1}-l,0)$ and $(n_{1}-2p_{1},0)$.
\end{lemma}
\begin{proof}
We first consider $x_{1}$. As it is obtained by inserting $f_{n_{1}-2p_{1}}$ into $x$, it will hit the critical vertex $(n_{1},p_{1})$ and so must correspond to a $TL_{n_{1}}^{S}$ element. It follows that $x_{1}$ must correspond to a pair of paths on the right of the critical line, and hence both paths must pass through a double edge, with both paths being equal before the double edge. Denote the two paths by $v_{+}$ and $v_{-}$, where $v_{-}$ is the fusion rule with complex coefficients. Then $x_{1}$ is determined by a sum of the form
\[a_{+,+}v_{+,+}+a_{+,-}v_{+,-}+a_{-,+}v_{-,+}+a_{-,-}v_{-,-}\]
for some coefficients $a_{\pm,\pm}$. We first note that as we constructed $x_{1}$ by inserting a diagram into the middle of a radical diagram, we must have 
\[
\chi_{\infty,l}(x_{1}) = 0,\]
and so 
\[
a_{+,+} = -a_{-,-}.\] 
Next, as $f_{n_{1}-2p_{1}}$ is $\dagger$ invariant, it follows that the double path contribution must also be $\dagger$ invariant. Hence
\[
a_{+,-} = a_{-,+}.\] Further, as $x_{1}$ is constructed from a radical element, we must have 
\[
x_{1}^{2} =0.\]
It follows that
\[
a_{+,+}^{2}+a_{+,-}^{2} = 0.\] Finally, by comparing $x_{1}$ directly with the double edge fusion rules, we see that the double edge contribution to $x_{1}$ is
\[
v_{+,+}+iv_{+,-}+iv_{-,+}-v_{-,-}.\]
To see that $x-x_{1}$ is a radical element, we only need to consider terms that will hit a critical line to the left of the $(n_{1},p_{1})$ line. If such a term appeared, then considering it as a $TL_{n}^{S}$ path, it is clear that to go from the $(n_{1},p_{1})$ line to a line on the left, then it must pass through one of the triple edges. As this edge is unique, it follows that the term would have non-zero trace. However, as every term in $x-x_{1}$ was constructed from the radical element, they must have zero trace. Hence any term in $x-x_{1}$ hitting a critical line to the left of $(n_{1},p_{1})$ must be zero.\\
Finally to see that any term in $x-x_{1}$ must have its path end in the critical strip, it is straightforward to note that any term must have its path pass through a vertex on the critical line $(n_{1}-2p_{1},0)$ to form the original radical element. From the Bratteli diagram, it can then be seen that there is no path that passes through this vertex and ends on a vertex to the left of the critical line $(n_{1}-2p_{1}-l,0)$.
\end{proof}
We note that a similar result holds if we only require one path of the element to be a radical path, and allow the other path to correspond to $TL_{n}^{S}$. However in this case extra semisimple elements may appear that end on the critical line $(n_{1}-2p_{1}-l,0)$. It follows that for a given radical element, we can iterate the above construction to rewrite $x$ as a sum of $TL_{n}^{S}$ elements 
\[
x\simeq x_{1}+x_{2}+...\in\bigcup\limits_{n}TL_{n}^{S}.\]
 Note that for $l>3$, vertices for radical fusion rules generally have inclusions of the form
\[
(n,p)\rightarrow(n+1,p)\oplus(n+1,p+1),\]
which means there are other radical terms to consider in the expansion. As we have shown that each radical element can be written as a sum of terms in $\bigcup TL_{n}^{S}$, it follows that we can apply $\diamond$ to each term to obtain an involution on radical elements, and hence on $TL_{\infty}$. However, the involution of radical terms will consist of an infinite sum of terms, which we have no guarantee will evaluate to something finite. We will address this in the second part of the proof of Theorem \ref{thm: fixed inner product}, by showing that the norms of radical terms are finite, and so that the involution of radical terms are in the completion of $TL_{\infty}$:
\begin{prop}
With respect to the involution $\diamond$, the inner product $(\cdot,\cdot)_{\gamma,l}$ on $TL_{\infty}$ is positive definite.
\end{prop}
\begin{proof}
By construction, the inner product is positive definite when restricted to $TL_{n}^{S}$. As we have constructed elements of $J(TL_{n})$ in terms of $\bigcup\limits_{n}TL_{n}^{S}$, it follows that all we need to show is that elements of $J(TL_{n})$ have finite norm under this construction. Similar to the generic case, we only need to consider diagonal elements, i.e. elements corresponding to two copies of the same path on the Bratteli diagram.\\

We first note the value of the norm for a fixed $x_{k}$ in the approximation of $x$. Let $(n_{k},p_{k})$ be the critical vertex that $x_{k}$ is defined from. Then in terms of the double edge matrix elements, the corrected norm of $x_{k}$ is given by
\[
\chi_{\infty,l}(x_{k}x_{k}^{\diamond}) = c_{n_{k},p_{k}}t_{n_{k},p_{k}}\left((e_{11}+ie_{12}+ie_{21}-e_{22})(e_{11}^{\diamond}-ie_{12}^{\diamond}-ie_{21}^{\diamond}-e_{22}^{\diamond})\right) = 4c_{n_{k},p_{k}}.\]
We now need to sum over all the terms appearing in the approximation of $x$. For a given $x_{k}$, it will contain a section of path starting at the original $(n_{0},p_{0})$ critical vertex from which the radical element was formed, and ending at the critical vertex $(n_{k},p_{k})$, without hitting a critical vertex in between. It follows from this that each of the $x_{k}$ in the approximation of $x$ will be pairwise orthogonal. Hence the corrected norm of $x$ will be given by the sum over the corrected norms of the $x_{k}$. In the simplest case, when $l=3$, this is just
\[
4\sum\limits_{i=0}^{\infty}c_{n_{0}+2i,p_{0}+i} = 4c_{n_{0}-2p_{0},0}\gamma^{p}\left(\sum\limits_{i=0}^{\infty}\gamma^{i}\right)=4c_{n_{0}-2p_{0},0}\frac{\gamma^{p}}{1-\gamma}.\]
This converges for $0\leq\gamma<1$, and considering any possible starting vertex, is positive for $0<\gamma\leq\frac{1}{4}$. For larger $l$, we need to consider the inclusions of vertices, as this will introduce extra radical terms that will also need approximated. Let $(n,p)$ be the final vertex of the fusion rule corresponding to the half-diagram whose norm we want to calculate. As all vertices appearing will be bound within a pair of critical lines, and by repetition of the Bratteli diagram, we can assume that $(n,p)$ lies to the leftmost critical line, with the calculation for other critical lines following similarly. Hence we can assume 
\[
\lfloor\frac{n-l+2}{2}\rfloor\leq p\leq \lfloor\frac{n}{2}\rfloor.\] 
We label the vertices by $v_{n,p}$, and $\hat{v}_{n,p}$ if critical. They will obey the inclusion rules
\[
v_{n,p}\mapsto \begin{cases} \hat{v}_{n+1,p}+v_{n+1,p+1} & n-2p=kl\\ v_{n+1,p} & n\text{ even and }p=\frac{n}{2}\\ v_{n+1,p}+v_{n+1,p+1} & \text{ otherwise }\end{cases}.\]
For our chosen starting vertex, considering repeated inclusions, we will have
\[
v_{n,p}\mapsto \hat{v}_{n_{1},p_{1}}+k_{1}\hat{v}_{n_{1}+2,p_{1}+1}+k_{2}\hat{v}_{n_{1}+4,p_{1}+2}+...\]
for some multiplicities $k_{i}$. The value of the $k_{i}$ will depend on the choice of starting vertex, however we can obtain a bound by noting that the inclusion rules for the $v_{n,p}$ correspond to some sub-Bratteli diagram of the Bratteli diagram to the left of all critical lines. In particular, the $k_{i}$ will be the number of paths to the $i+1$th vertex in a fixed column of the sub-Bratteli diagram. In the Bratteli diagram to the left of all critical lines, the number of paths to a given vertex $v_{n^{\prime},p^{\prime}}$ is just the dimension of $\mathcal{L}_{n^{\prime},p^{\prime}}$. Denote
\[
\text{ dim }\mathcal{L}_{n,p}:= L_{n,p}.\] 
By choosing the column containing the largest irreducibles, we can use their dimension to bound the $k_{i}$. It can be see that the column containing the largest sequence of irreducibles is given by 
\[
(l-2,0).\]
Returning to our calculation of the corrected norm of $x$, we have
\begin{align*}
\chi_{\infty,l}(xx^{\diamond}) =& 4(c_{n_{1},p_{1}}+k_{1}c_{n_{1}+2,p_{1}+1}+k_{2}c_{n_{2}+4,p_{1}+2}+...)\\
=&4c_{n_{1},p_{1}}\left(1+\gamma k_{2}+\gamma^{2}k_{3}+...\right)\\
\leq&4c_{n_{1},p_{1}}\left(1+\gamma L_{l,1}+\gamma^{2}L_{l+2,2}+...\right)\\
=&4c_{n_{1},p_{1}}\left(\sum\limits_{i=0}^{\infty}\gamma^{i}L_{l-2+2i,i}\right)
\end{align*}
We now note that, as is well known in subfactor theory, there is a bipartite graph with Frobenius-Perron eigenvalue $\delta$, such that the $n$th matrix power of its adjacency matrix give the values of the $L_{n,i}$ for $i$ large enough. From this, it follows that
\[
L_{l-2+2i,i}\leq \delta^{l+2i-2},\]
and so we have
\[
\chi_{\infty,l}(xx^{\diamond})\leq 4c_{n_{1},p_{1}}\left(\sum\limits_{i=0}^{\infty}\delta^{l+2i-2}\gamma^{i}\right).\]
Hence the corrected norm of $x$ is finite for
\[
0\leq\gamma<\delta^{-2},\]
and is bounded by
\[
\|x\|^{2}_{\diamond,\infty,l}\leq 4c_{n_{1},p_{1}}\left(\frac{\delta^{l-2}}{1-\delta^{2}\gamma}\right).\]
Combined with the positive-definiteness condition for $c_{n,i}$, we see that $(\cdot,\cdot)_{\gamma,l}$ is positive definite with respect to $\diamond$ for
\[
0<\gamma\leq\frac{1}{4}.\]
\end{proof}
This completes the proof of Theorem \ref{thm: fixed inner product}.\\

In the generic case, it was straightforward to calculate the norm of a diagram element via the trace. In the root of unity case, the use of $\diamond$ makes this more difficult. further, in the generic case, the multiplicativity condition on the trace resulted in a similar condition on the norm. Again the use of $\diamond$ means that this is no longer true in general. Instead, we can give weaker conditions for the norm:
\begin{corr}
For $x,y\in TL_{\infty}$,
\[
\|x\|^{2}_{\diamond,\infty,l} \geq \chi(xx^{\ast}), ~ \|x\otimes y\|^{2}_{\diamond,\infty,l}\geq \chi(xx^{\ast})\chi(yy^{\ast}).\]
\end{corr}
\begin{proof}
This follows from considering how $\diamond$ corrects the norm.
\end{proof}

\section{The Corrected involution for the generators of $TL_{n}(0)$.}\label{section: zero involution}

So far we have defined the involution $\diamond$ in terms of matrix elements and not the generators. In general, giving formulae for $e_{i}^{\diamond}$ may be quite difficult as it will generally involve infinite sums coming from radical elements. However in the case of $\delta=0$, $\diamond$ is an involution on $TL_{n}$ for $n$ odd, and so in terms of the generators should be more straightforward. Our aim for this section is to give formulae for $e_{i}^{\diamond}$ in the case of $\delta=0$. 

\begin{thm}
For $\delta=0$, the involution $\diamond$ is given in terms of the generators as follows:
\begin{itemize}
	\item $e_{1}^{\diamond}=e_{2}$.
	\item For $n>2$ even,
	\begin{figure}[H]
		\centering
		\includegraphics[width=0.4\linewidth]{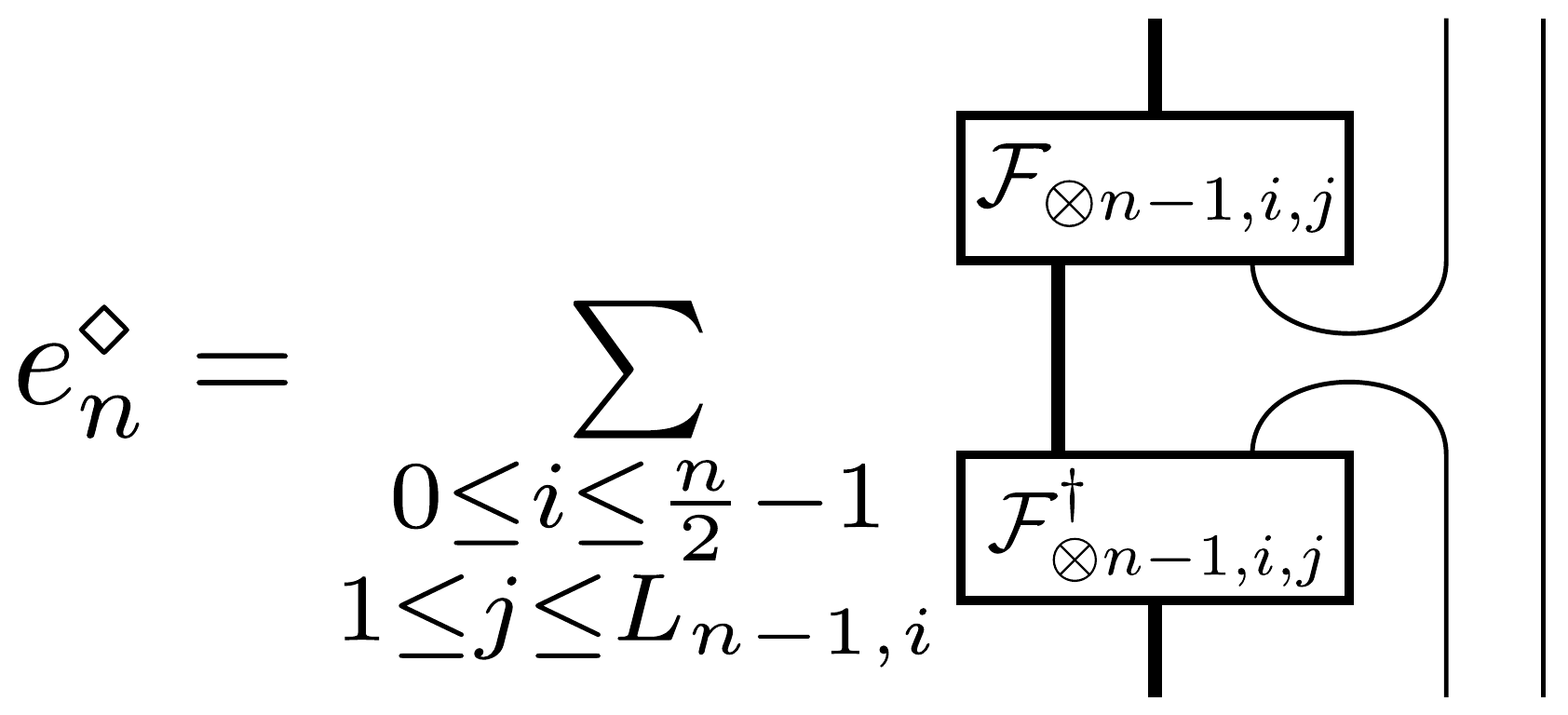}
	\end{figure} 
\item For $n>2$ odd,
\begin{figure}[H]
	\centering
	\includegraphics[width=0.95\linewidth]{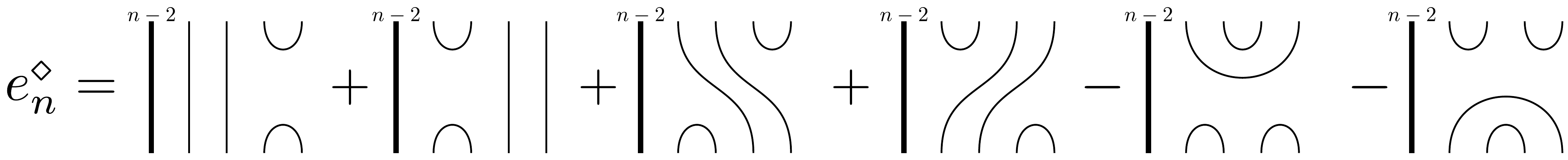}
\end{figure}
\end{itemize} 
\end{thm}

\begin{proof}
We first note that for $\delta=0$, inclusions from even to odd levels are trivial on the Bratteli diagram, so when considering paths on the Bratteli diagram we will only bother to label odd vertices. To simplify notation, we will sometimes denote a matrix element $v_{p_{1},p_{2}}$ in terms of a pair of paths as follows:
\[
v_{p_{1},p_{2}} = v_{p_{1}}v_{p_{2}}^{\dagger}\]
We will also need the following identity for odd Jones-Wenzl idempotents:
\begin{figure}[H]
	\centering
	\includegraphics[width=0.7\linewidth]{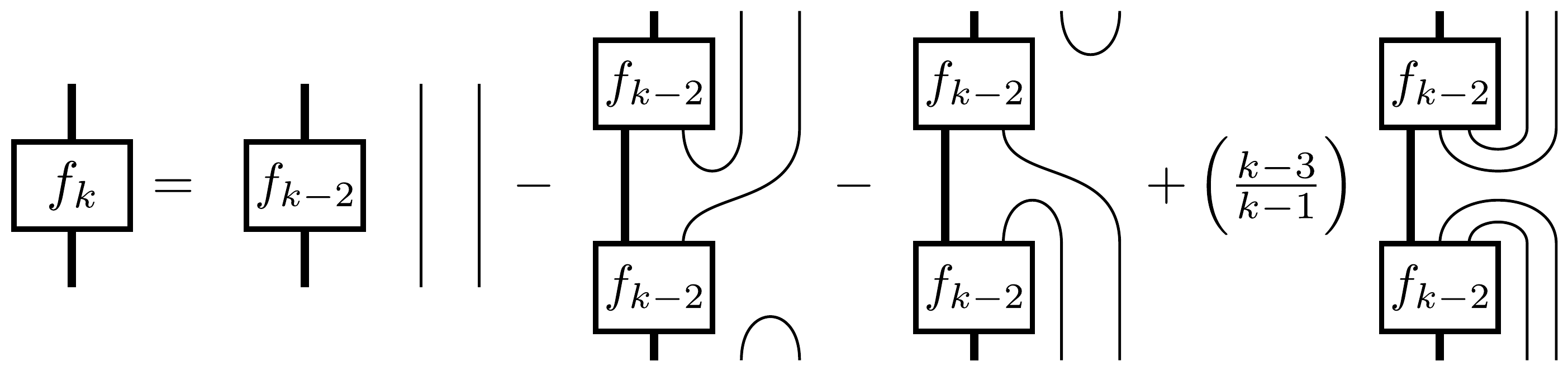}
\end{figure}
In terms of matrix elements, it is straightforward to check that
\[e_{1} = \frac{1}{2}(v_{1_{+},1_{+}}+iv_{1_{+},1_{-}}+iv_{1_{-},1_{+}}-v_{1_{-},1_{-}}), ~ e_{2} = \frac{1}{2}(v_{1_{+},1_{+}}-iv_{1_{+},1_{-}}-iv_{1_{-},1_{+}}-v_{1_{-},1_{-}}).\]
Under $\diamond$, we have $v_{1_{+},1_{-}}^{\diamond} = v_{1_{-},1_{+}}$, and so $e_{1}^{\diamond}=e_{2}$.
\subsection*{The case $n>2$ even.}
Using the $TL_{n}$ fusion rules of Section \ref{section: Matrix decomposition} we can write $e_{n}$ as
\begin{figure}[H]
	\centering
	\includegraphics[width=0.5\linewidth]{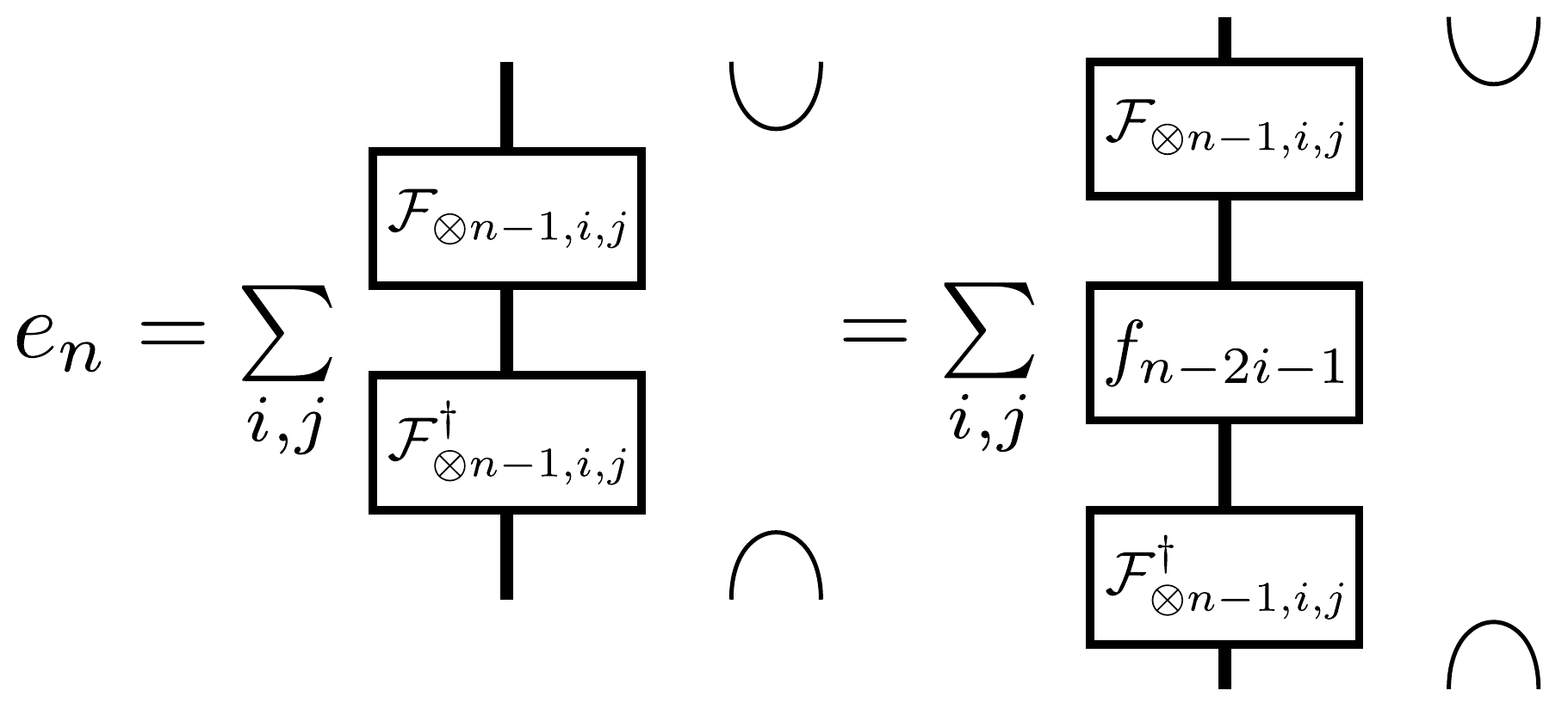}
\end{figure}
where the second equality comes from noting that we can pull a Jones-Wenzl idempotent from the bottom of each fusion rule in the case we are considering. The sums in each case are over all the possible fusion rule terms. Each diagram in the sum, when viewed as a path on the Bratteli diagram, will be of the form
\[
\frac{1}{2}\left(v_{p_{n-1}i_{+},p_{n-1}i_{+}}-iv_{p_{n-1}i_{+},p_{n-1}i_{-}}-iv_{p_{n-1}i_{-},p_{n-1}i_{+}}-v_{p_{n-1}i_{-},p_{n-1}i_{-}}\right),\]
for some path $p_{n-1}$ to the vertex $(n-1,i-1)$. Applying $\diamond$, we will then get
\[
\frac{1}{2}\left(v_{p_{n-1}i_{+},p_{n-1}i_{+}}+iv_{p_{n-1}i_{+},p_{n-1}i_{-}}+iv_{p_{n-1}i_{-},p_{n-1}i_{+}}-v_{p_{n-1}i_{-},p_{n-1}i_{-}}\right).\]
Diagrammatically, in terms of the end of the paths, this is given by:
\begin{figure}[H]
	\centering
	\includegraphics[width=0.35\linewidth]{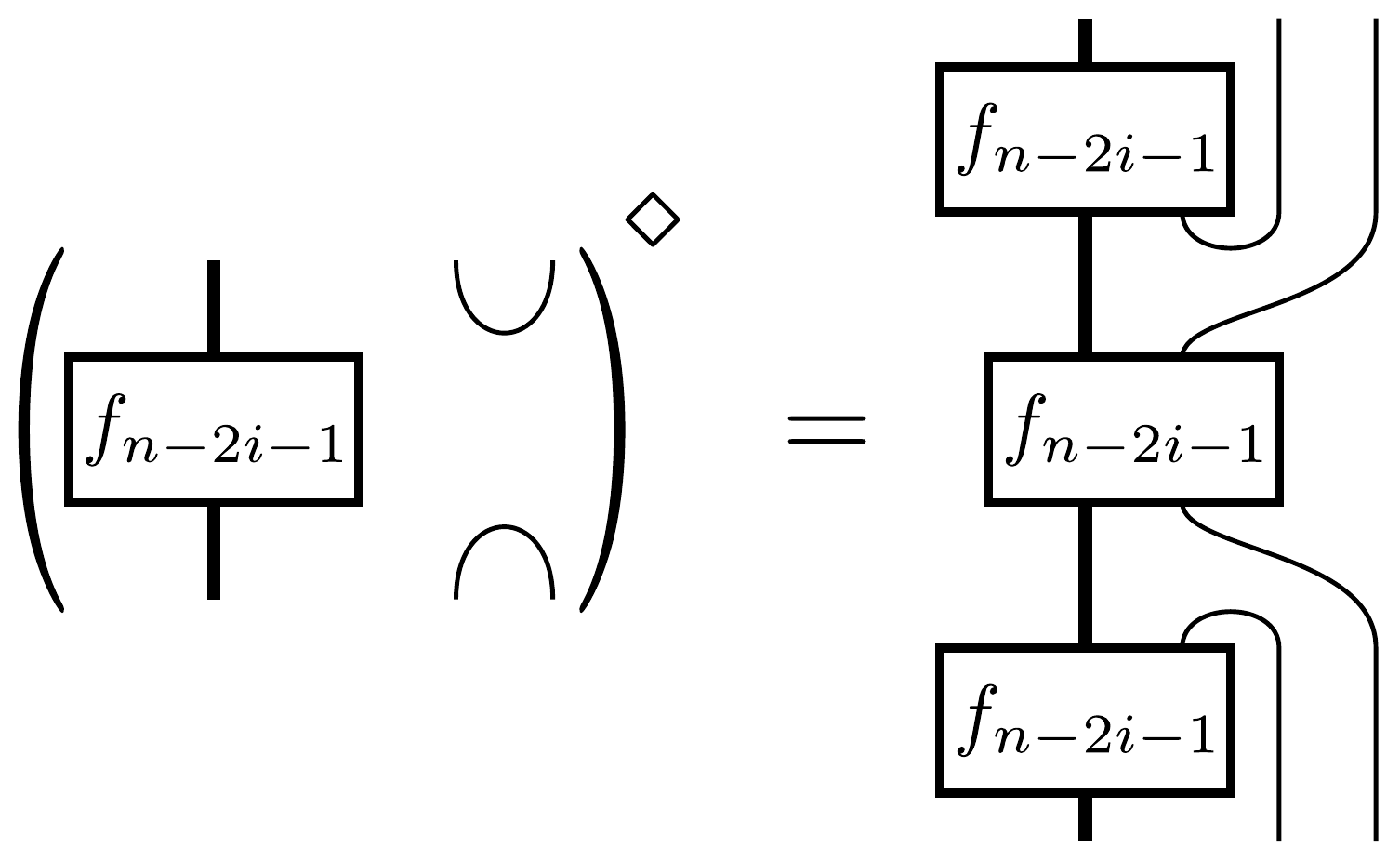}
\end{figure}
Hence we have
\begin{figure}[H]
	\centering
	\includegraphics[width=0.6\linewidth]{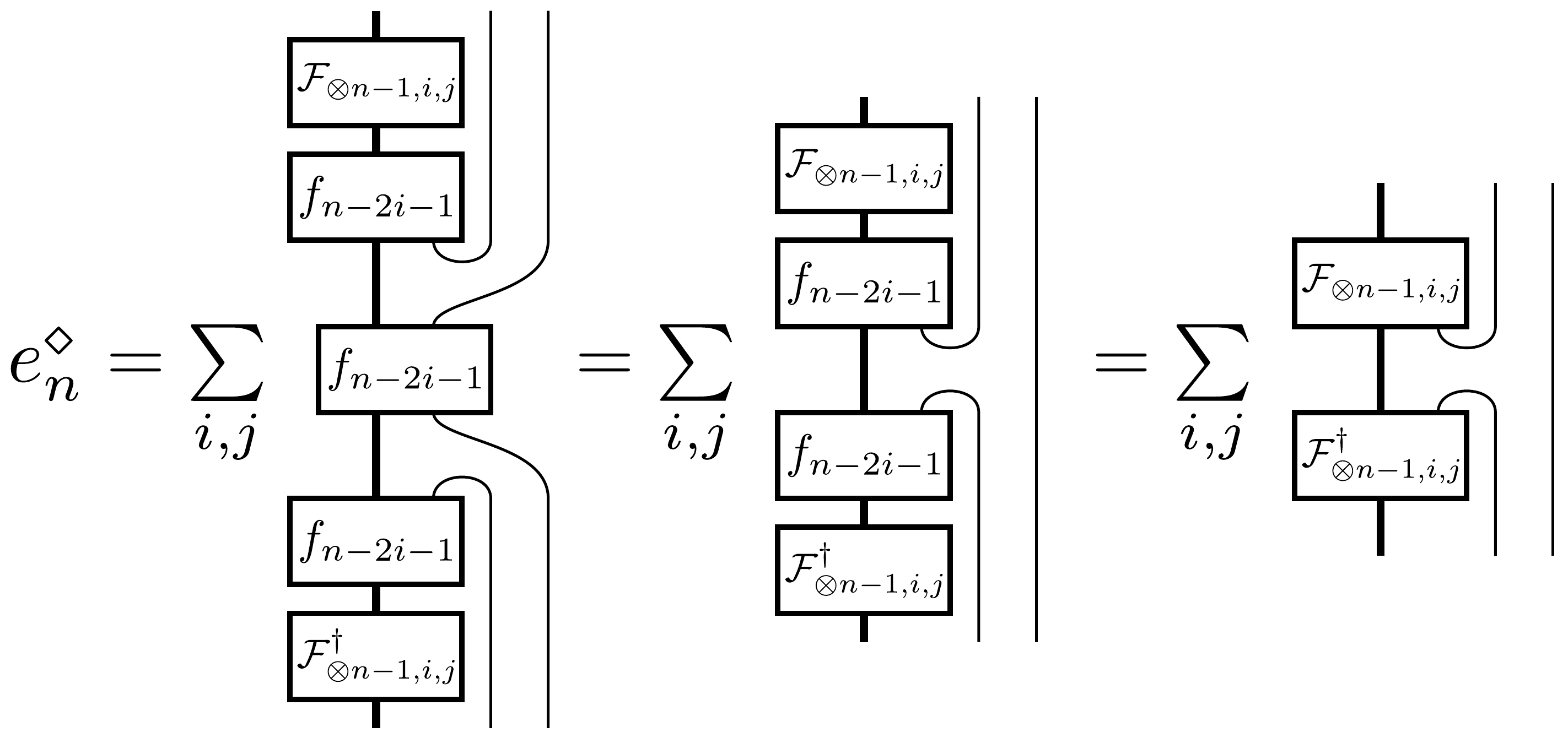}
\end{figure}
Where the second equality comes from using the expansion for $f_{k}$ given above.
\subsection*{The case $n>2$ odd.}
To start, we note that we will be viewing $e_{n}$ as an element of $TL_{n+2}$. We consider the following $TL_{n+2}$ elements given in terms of paths:
\begin{align}
A_{1}:=& \frac{1}{2}\left(v_{0_{n}1_{+}}+iv_{0_{n}1_{-}}+v_{0_{n-2}1_{+}1}+iv_{0_{n-2}1_{-}1}\right)\left(v_{0_{n}1_{+}}^{\dagger}+iv_{0_{n}1_{-}}^{\dagger}+v_{0_{n-2}1_{+}1}^{\dagger}+iv_{0_{n-2}1_{-}1}^{\dagger}\right)\\
A_{2}:=& \frac{1}{2}\left(-i\sqrt{\frac{n+1}{n-1}}v_{0_{n}2}+v_{0_{n-2}1_{+}2_{+}}+v_{0_{n-2}1_{-}2_{-}}+i\sqrt{\frac{n-3}{n-1}}v_{0_{n-2}22}\right)\times\\
&\times\left(v_{0_{n-2}1_{+}2_{+}}^{\dagger}+iv_{0_{n-2}1_{+}2_{-}}^{\dagger}+iv_{0_{n-2}1_{-}2_{+}}^{\dagger}-v_{0_{n-2}1_{-}2_{-}}^{\dagger}\right)\\
A_{3}:=& \frac{1}{2}\left(v_{0_{n-2}1_{+}3}+v_{0_{n-2}23_{+}}+iv_{0_{n-2}23_{-}}\right)\left(v_{0_{n-2}1_{+}3}^{\dagger}+v_{0_{n-2}23_{+}}^{\dagger}+iv_{0_{n-2}23_{-}}^{\dagger}\right)
\end{align}
In terms of diagrams, these elements are:
\begin{figure}[H]
	\centering
	\includegraphics[width=0.6\linewidth]{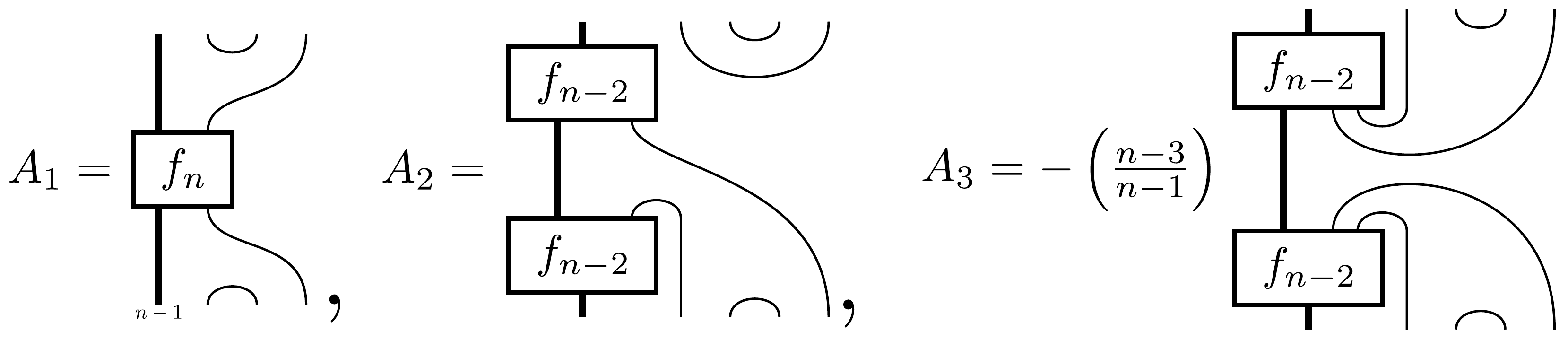}
\end{figure}
Combining the elements, we have
\begin{figure}[H]
	\centering
	\includegraphics[width=0.4\linewidth]{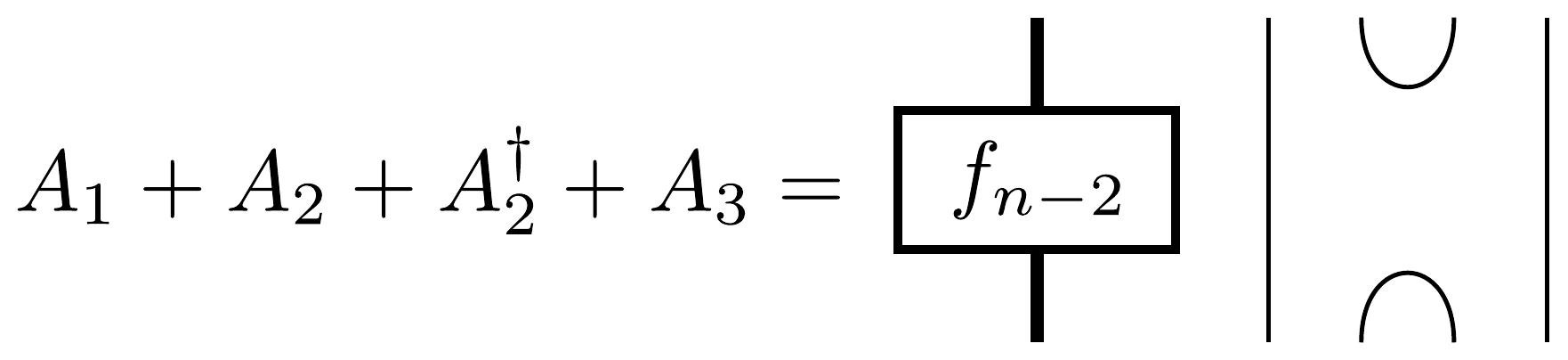}
\end{figure}
Hence we can rewrite $e_{n}$ as
\begin{figure}[H]
	\centering
	\includegraphics[width=0.5\linewidth]{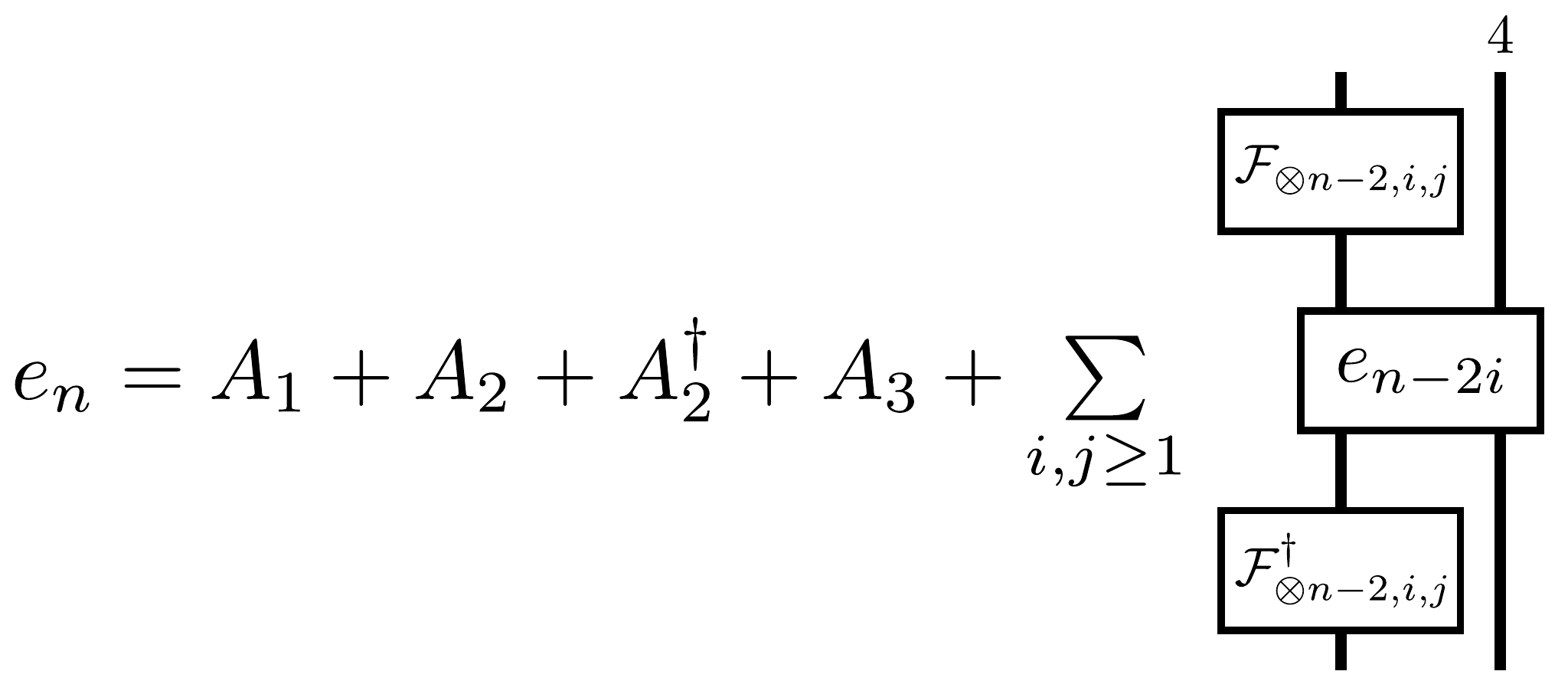}
\end{figure}
Applying $\diamond$ to $A_{1}$, $A_{2}$, and $A_{3}$, we obtain as follows:
\begin{figure}[H]
	\centering
	\includegraphics[width=0.6\linewidth]{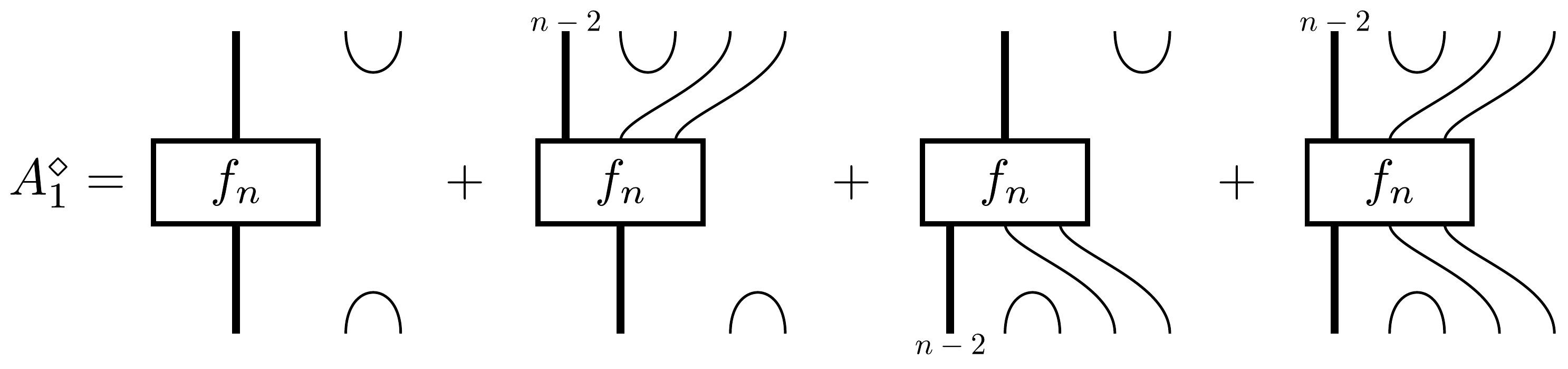}
\end{figure}
\begin{figure}[H]
	\centering
	\includegraphics[width=0.6\linewidth]{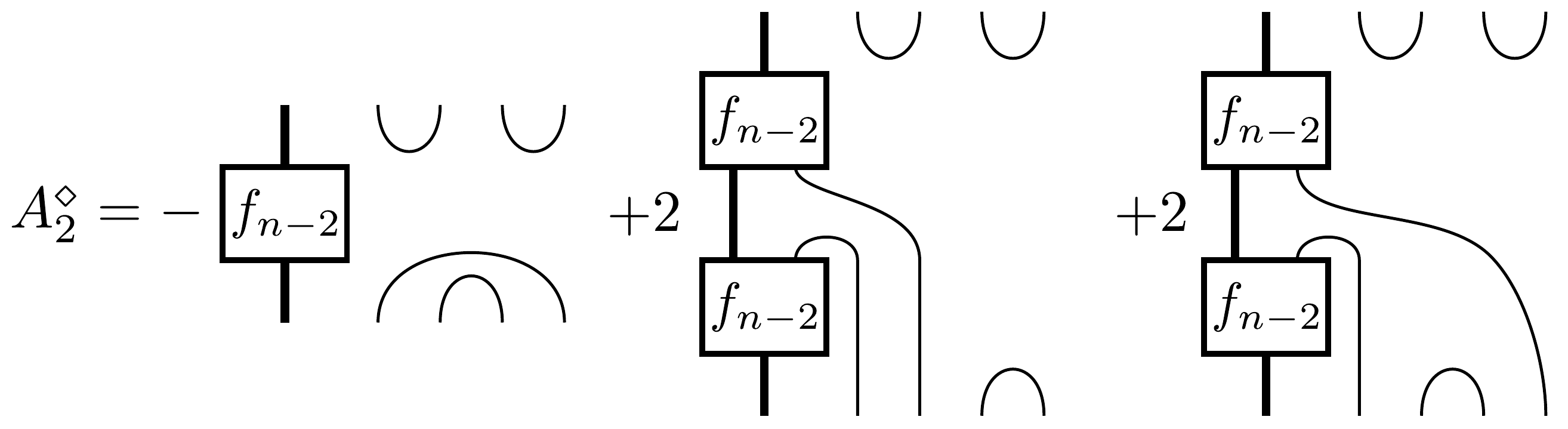}
\end{figure}
\begin{figure}[H]
	\centering
	\includegraphics[width=0.8\linewidth]{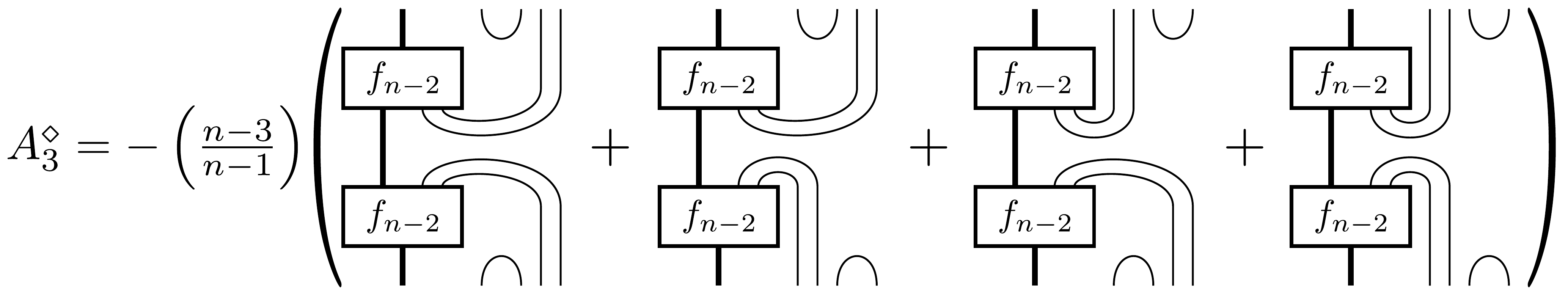}
\end{figure}
Taking the previous rewrite of $e_{n}$, substituting in the above formulae for $A_{i}^{\diamond}$, and simplifying, we then obtain the following:
\begin{figure}[H]
	\centering
	\includegraphics[width=0.7\linewidth]{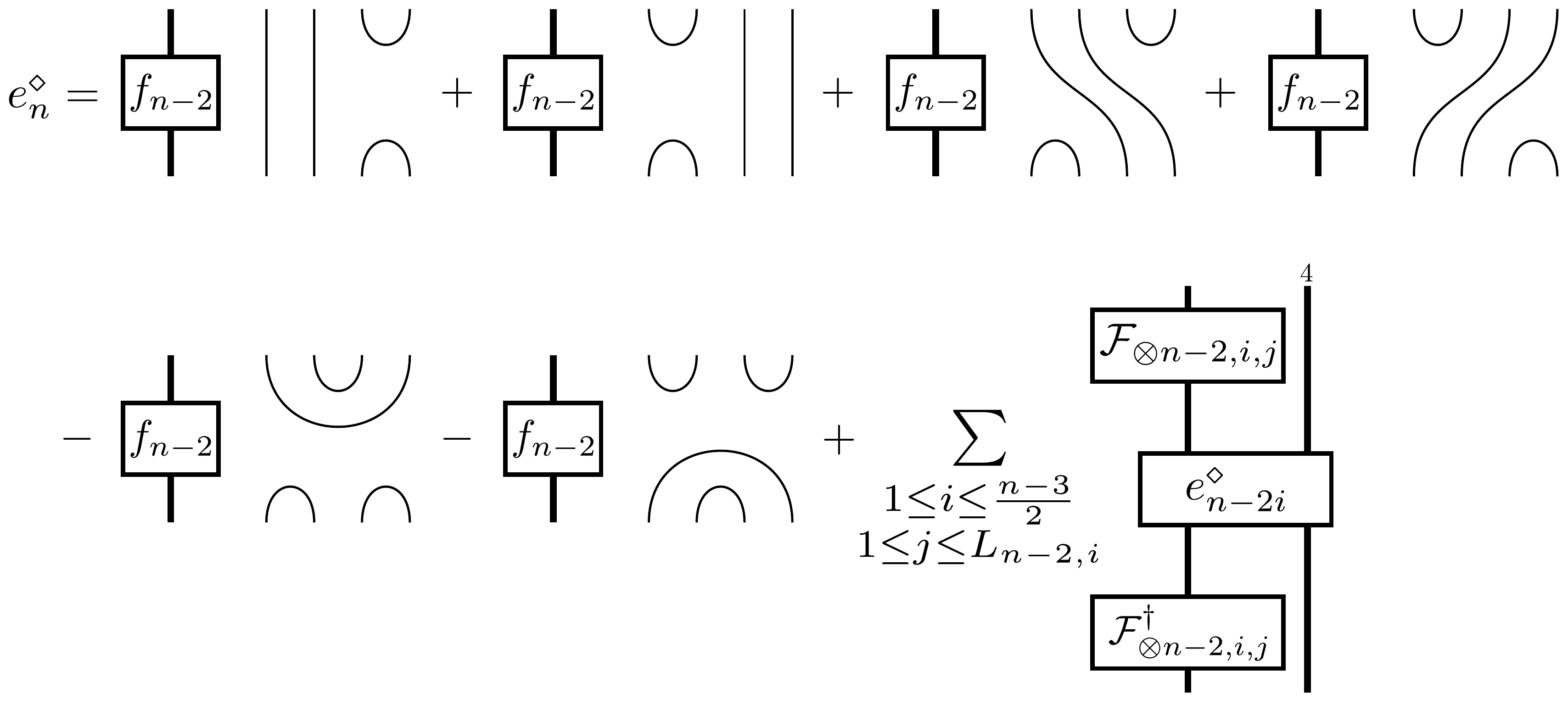}
\end{figure} 
We now note that for $n=3$, the fusion rule summation term will not appear, as we have
\[
A_{1}+A_{2}+A_{2}^{\dagger}=e_{3}.\] 
Hence in the case of $e_{3}^{\diamond}$ we already have the required formula. Next, we note that the formula we're trying to prove for $e_{n}^{\diamond}$ has vertical through strands for the first $n-2$ points. Proceeding by induction, we then get that $\mathcal{F}_{\otimes n-2,i,j}$ will commute with the $e_{n-2i}^{\diamond}$, and so we can rewrite:
 \begin{figure}[H]
 	\centering
 	\includegraphics[width=0.9\linewidth]{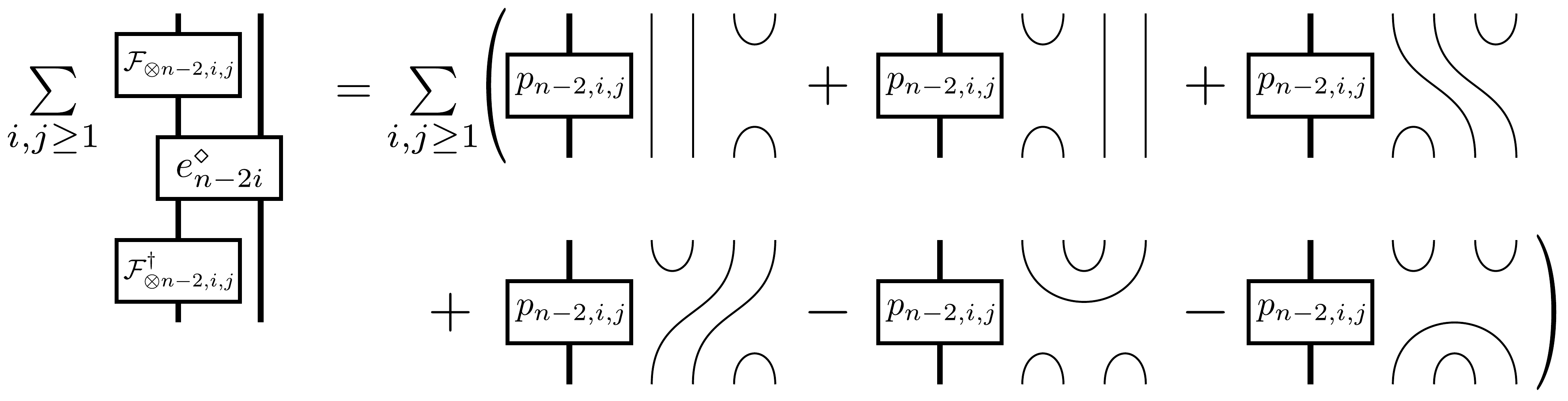}
 \end{figure} 
where the sum on the right is just a sum over the partition of unity, apart from the term $f_{n-2}$. Hence we can simplify the above formula for $e_{n}^{\diamond}$ as required.
\end{proof}

We note that while $e_{n}^{\diamond}$ is very simple in terms of basic diagrams for $n$ odd, for $n$ even, it gets more complicated as $n$ increases. For example, $e_{4}^{\diamond}$ and $e_{6}^{\diamond}$ are given as follows:
\[
e_{4}^{\diamond} = e_{1}+e_{3}-e_{1}e_{2}e_{3}-e_{3}e_{2}e_{1}+e_{1}e_{3}e_{2}+e_{2}e_{1}e_{3}\]
\begin{align*}
e_{6}^{\diamond} =& e_{4}^{\diamond}+e_{5}-e_{3}e_{4}e_{5}-e_{5}e_{4}e_{3}+e_{3}e_{5}e_{4}+e_{4}e_{3}e_{5}+e_{1}e_{2}e_{3}e_{4}e_{5}-e_{1}e_{2}e_{3}e_{5}e_{4}-e_{1}e_{2}e_{4}e_{3}e_{5}+e_{1}e_{2}e_{5}e_{4}e_{3}\\	
&-e_{1}e_{3}e_{2}e_{4}e_{5}+e_{1}e_{3}e_{2}e_{5}e_{4}+e_{1}e_{4}e_{3}e_{2}e_{5}-e_{1}e_{5}e_{4}e_{3}e_{2}-e_{2}e_{1}e_{3}e_{4}e_{5}+e_{2}e_{1}e_{3}e_{5}e_{4}+e_{2}e_{1}e_{4}e_{3}e_{5}\\
&-e_{2}e_{1}e_{5}e_{4}e_{3}+e_{3}e_{2}e_{1}e_{4}e_{5}-e_{3}e_{2}e_{1}e_{5}e_{4}-e_{4}e_{3}e_{2}e_{1}e_{5}+e_{5}e_{4}e_{3}e_{2}e_{1} 
\end{align*}
We can use the formulae for $e_{n}^{\diamond}$ to give the value of the norm of $e_{n}$ in the $\delta=0$ case:
\begin{corr}
For $n>2$ odd, 
\[
\| e_{n}\|^{2}_{\diamond,\infty,2}=4\gamma.\]
Otherwise,
\[
\| e_{n}\|^{2}_{\diamond,\infty,2}=\gamma.\]
\end{corr}
\begin{proof}
For $e_{1}$ we have
\[\| e_{1}\|^{2}_{\diamond,\infty,2} = \chi_{\infty,2}(e_{1}e_{1}^{\diamond}) = \chi_{\infty,2}(e_{1}e_{2})=\gamma.\]
With $e_{2}$ following similarly.
For $n>2$ odd, $e_{n}e_{n}^{\diamond}$ gives
 \begin{figure}[H]
	\centering
	\includegraphics[width=0.7\linewidth]{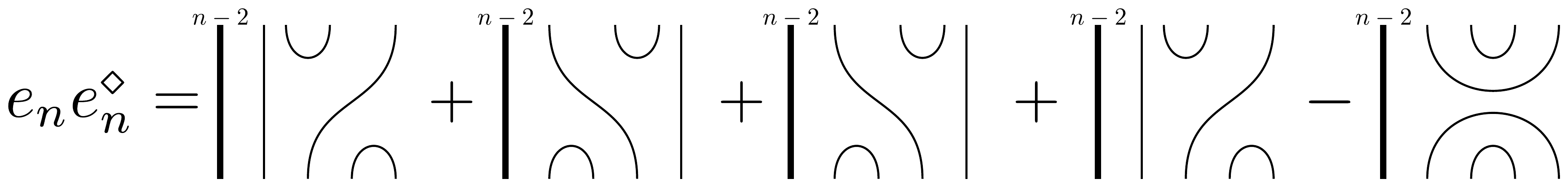}
\end{figure} 
The final term has trace equal to zero, and all the other terms have trace equal to $\gamma$, so we have
\[
\| e_{n}\|^{2}_{\diamond,\infty,2} = 4\gamma \]
when $n>2$ is odd.\\

For $n>2$ even, we have
 \begin{figure}[H]
	\centering
	\includegraphics[width=0.3\linewidth]{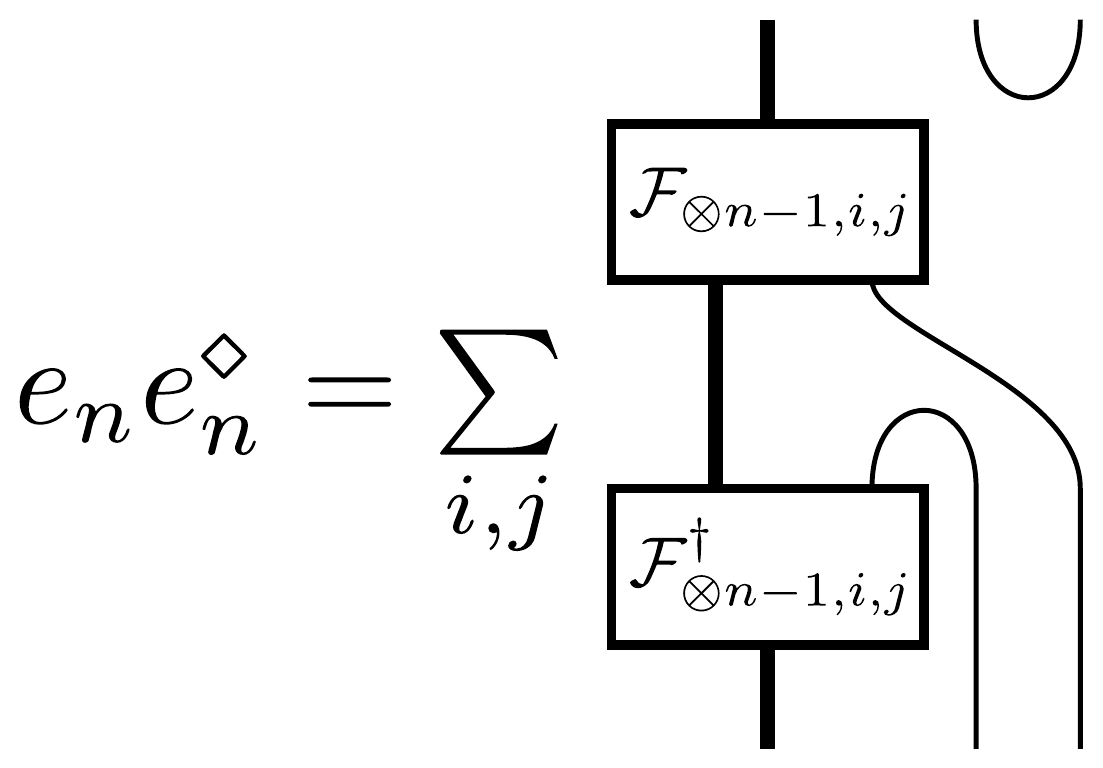}
\end{figure}
Consider the above diagram for a fixed choice of fusion rule corresponding to a path $p$ to $(n-1,i-1)$. If we add to it the diagram given by horizontal reflection, we obtain the idempotent corresponding to the paths
\[v_{pi_{+},pi_{+}}+v_{pi_{-},pi_{-}}.\]
This will have trace $2\lambda_{n+1,i}$, and so our original chosen diagram has trace $\lambda_{n+1,i}$. Taking the sum over all choices, we have
\[
\| e_{n}\|^{2}_{\diamond,\infty,2} = \sum\limits_{i=0}^{\frac{n}{2}-1}d_{n-1,i}\lambda_{n+1,i+1}=\sum\limits_{i=0}^{\frac{n}{2}-1}\gamma d_{n-1,i}\lambda_{n-1,i}=\gamma.\]
\end{proof}

\section{The Regular representation at roots of unity.}\label{section: Hilbert space rou}

We denote by $\mathcal{H}_{\gamma,l}$ the completion of $TL_{\infty}$ with respect to the inner product $(\cdot,\cdot)_{\gamma,l}$ and involution $\diamond$. As in the generic case, we can define an action of $TL_{\infty}\otimes TL_{\infty}$ on this by
\[
(x\otimes y)v = xvy^{\diamond}.\]
\begin{thm}
As a $TL_{\infty}\otimes TL_{\infty}$ representation, $\mathcal{H}_{\infty,l}$ has no closed invariant subspaces.
\end{thm}
This can be proven similarly to the generic case, however we omit a proof due to space.

\subsection{Generalized regular representations at roots of unity.}

As with the generic case, we want to find representations that realize the special values of $\gamma$ specified in Proposition \ref{prop: rou coeff conditions}. We again consider a generalized infinite Temperley-Lieb diagram $w$, and the representation $\mathcal{V}(w)$ generated by the $TL_{\infty}\otimes TL_{\infty}$ action. We can consider an inner product on $\mathcal{V}(w)$ by defining
\[
(x,y)_{w,\gamma,l}:=\lim\limits_{n\rightarrow\infty}\frac{(\downarrow_{n}x,\downarrow_{n}y)_{\gamma}}{\gamma^{c(\downarrow_{n}w)}}.\]
Note that this is well-defined even for $\delta=0$ as we are using the corrected inner product. 
Recall that we need
\[
\gamma<\delta^{-2}\]
for the norm of radical elements to be finite. Using Proposition \ref{prop: rou coeff conditions}, we can state the values of $\gamma$ giving positive definite inner products as follows:
\begin{prop}
Let $s(w)=n$. The positive-definite inner products $(\cdot,\cdot)_{w,\gamma,l}$ on $\mathcal{V}(w)$ are given by the following values of $\gamma$:
\begin{itemize}
\item If $n=0,1$, $0<\gamma<\delta^{-2}$. 
\item If $l=2$, $0<\gamma<s_{n+1}$ for $n$ odd, and $0<\gamma<s_{n}$ for $n$ even.
\item If $l\geq 3$ and $n\leq l-1$, $0<\gamma<s_{n+2}$.
\item If $l\geq 3$ and $l\leq n<l+\lfloor\frac{l}{2}\rfloor$, or $n>2l-2$, with $n=kl+i-1$ for $0\leq i\leq l-1$, $0<\gamma<s_{kl}$.
\item If $l\geq 3$ and $l+\lfloor\frac{n}{2}\rfloor\leq n\leq 2l-2$, $0<\gamma<s_{2(n+1\text{ Mod }l)}$.
\end{itemize}
Further, the inner products on $\mathcal{V}(w)$ are equivalent for different values of $\gamma$.
\end{prop}
Here we have included the positive-definite inner products coming from quotienting semi-definite ones. That the inner products are equivalent follows similarly to Lemma \ref{lemma: equivalent inner products}.
We denote by
\[\mathcal{H}_{\infty,l}(w)\]
the completion of $\mathcal{V}(w)$ with respect to $(\cdot,\cdot)_{w,\gamma,l}$.
We first need to describe how $\mathcal{V}(w)$ decomposes in the root of unity case. Recall the representations $\mathcal{V}_{n}(w)$ from Definition \ref{defin: generalized standard reps}. From \cite{MeTL} we see that these will in general be indecomposable if $n\neq kl-1$. These will not be the only representations that appear in the decomposition of the generalized regular representations. To discuss the other representations that appear, we first need to introduce the following notation:
\begin{defin}
We denote by $\langle i\rangle_{l}$, $0\leq i\leq l-1$, the set
\[
\{r(k)\in\mathbb{N}: r(k):=2kl-2-r(k-1), ~ r(0):=i\}.\]
We denote by $\langle i\rangle_{l}^{(n)}$ the subset
\[
\{j\in\langle i \rangle_{l}:j\leq n\}.\]
\end{defin}
If we rewrite $i=n-2p$ for $(n,p)$, then $\langle i\rangle_{l}$ is the orbit of $(n,p)$ reflected about the critical lines.
\begin{defin}
For a generalized Temperley-Lieb diagram $w$, we denote by $\mathcal{W}_{i}(w)$, $0\leq i\leq l-1$ the subspace of $\mathcal{V}(w)$ consisting of generalized Temperley-Lieb diagrams $w^{\prime}$ such that $s(w^{\prime})\in\langle i\rangle_{l}^{(s(w))}$. We define the $TL_{\infty}\otimes TL_{\infty}$ action on $\mathcal{W}_{i}(w)$ to be zero if the resulting diagram has string number not in $\langle i\rangle_{l}^{(s(w))}$.
\end{defin}
\begin{prop}
$\mathcal{W}_{i}(w)$ is an indecomposable $TL_{\infty}\otimes TL_{\infty}$ representation.
\end{prop}
\begin{proof}
For $s(w)$, consider a partition of unity for $TL_{n}$, $\{p_{n,j}\}$. Choose a $p_{n,j}$ such that $(n,j)$ is not critical. Then apply $p_{n,j}$ to the $n$ through strands in $w$, and denote the result by $w_{j}$. It is straightforward to see that we can either apply a $TL_{\infty}$ element to $w_{j}$ to obtain a radical element applied to $w$ with $j$ through strands, or to obtain a radical element applied to $w$ with $j^{\prime}<j$ through strands, where $j^{\prime}\in\langle j\rangle_{l}^{(s(w))}$. As we can apply $TL_{\infty}$ on both the top and bottom of $w$, it follows that any element generated by the action on $w_{j^{\prime\prime}}$, for $j^{\prime\prime}\in\langle j \rangle_{l}^{(s(w))}$, must be in the same indecomposable representation.
\end{proof}
We note that if $s(w)$ is small enough such that $\langle i\rangle_{l}^{(s(w))}$ contains a single element, then $\mathcal{W}_{i}(w)$ is just $\mathcal{V}_{i}(w)$. We can now state the decomposition of $\mathcal{H}_{\infty,l}(w)$:
\begin{thm}
$\mathcal{H}_{\infty,l}(w)$ decomposes in to $TL_{\infty}\otimes TL_{\infty}$ invariant closed subspaces as follows:
\[
\mathcal{H}_{\infty,l}(w)\simeq\bigoplus\limits_{i=0}^{\min\{l-2,s(w)\}}\mathcal{K}_{i,\infty,l}(w)\bigoplus\limits_{k: kl-1\leq s(w)}\mathcal{H}_{kl-1,\infty,l}(w),\]
where $\mathcal{K}_{i,\infty,,l}(w)$ is the closure of $\mathcal{W}_{i}(w)$, and $\mathcal{H}_{j,\infty,l}(w)$ is the closure of $\mathcal{V}_{j}(w)$.
\end{thm}
\begin{proof}
Let $s(w)=n$. Given the partition of unity $\{p_{n,i}\}$ for $TL_{n}$, applying each $p_{n,i}$ to the $n$ through strands of $w$ will split $\mathcal{H}_{\infty,l}(w)$ into $\lfloor\frac{n}{2}\rfloor$ subspaces, with the direct sum of subspaces corresponding to a single orbit being closed. We note that for a non-critical $(n,i)$, $\mathcal{W}_{i}(w)$ will be contained in multiple subspaces. To see that these spaces do not split in the closure, consider the element $w_{i}$ given by applying $p_{n,i}$ to the $n$ through strands of $w$. Assume there is some sum of terms 
\[
\sum c_{j}x_{j}\]
that gives a splitting of $w_{i}$, where $x_{j}$ are generalized Temperley-Lieb diagrams. As $w_{i}$ contains of infinitely many cups, we see that we can apply $\delta^{-1}e_{i}$ at each of these cups, (or $e_{i}e_{i+1}$ for $\delta=0$), so that the remaining $x_{j}$ in the sum only differ from $w_{j}$ at the original $n$ through strand points of $w$. Hence we have reduced the splitting to a finite sum of terms. However, we can then remove the extra cups from the sum to obtain a corresponding sum in $TL_{n}$, which would give a splitting of the representation generated by the two-sided action of $TL_{n}$ on the set of minimal idempotents in a fixed orbit. Since such a representation is indecomposable, it follows that $\mathcal{W}_{i}(w)$ does not split in the closure.
\end{proof}

\section*{Acknowledgements}
Research supported by Narodowe Centrum Nauki, grant number 2017/26/A/ST1/00189.

\bibliography{ref}

\begin{thebibliography}{10}

\bibitem{Inna}
M.~Balagovic, Z.~Daugherty, I.~Entova-Aizenbud, I.~Halacheva, J.~Hennig, M.~Im,
  G.~Letzter, E.~Norton, V.~Serganova, and C.~Stroppel.
\newblock Translation functors and decomposition numbers for the periplectic
  {L}ie superalgebra $\mathfrak{p}(n)$.
\newblock {\em Mathematical Research Letters}, 26(3):643--710, 2019.

\bibitem{Blanchet}
C.~Blanchet, M.~D. Renzi, and J.~Murakami.
\newblock {D}iagrammatic construction of representations of small quantum
  $\mathfrak{sl}_{2}$.
\newblock {\em Transformation Groups}, Sept. 2021.

\bibitem{Borodin-Olshanski}
A.~Borodin and G.~Olshanski.
\newblock Harmonic functions on multiplicative graphs and interpolation
  polynomials.
\newblock {\em The Electronic Journal of Combinatorics}, 7(1), May 2000.

\bibitem{BO}
A.~Borodin and G.~Olshanski.
\newblock {\em Representations of the Infinite Symmetric Group}.
\newblock Cambridge Studies in Advanced Mathematics. Cambridge University
  Press, 2016.

\bibitem{GS}
A.~M. Gainutdinov and H.~Saleur.
\newblock {Fusion and braiding in finite and affine Temperley-Lieb categories}.
\newblock {\em arXiv:1606.04530}, 2016.

\bibitem{GO}
A.~Gnedin and G.~Olshanski.
\newblock Coherent permutations with descent statistic and the boundary problem
  for the graph of zigzag diagrams.
\newblock {\em International Mathematics Research Notices}, Jan. 2006.

\bibitem{GLR}
I.~Gohberg, P.~Lancaster, and L.~Rodman.
\newblock {\em Indefinite Linear Algebra and Applications}.
\newblock Birkhäuser Basel, Oct. 2005.

\bibitem{GW}
F.~M. Goodman and H.~Wenzl.
\newblock The {T}emperley-{L}ieb algebra at roots of unity.
\newblock {\em Pacific J. Math.}, 161(2):307--334, 1993.

\bibitem{Ibanez}
E.~Ibanez.
\newblock {\em Evaluable Jones-Wenzl idempotents at root of unity and modular
  representation on the center of $\bar{U}_qsl(2)$}.
\newblock PhD thesis, University of Montpellier, 2016.

\bibitem{ILZ}
K.~Iohara, G.~I. Lehrer, and R.~B. Zhang.
\newblock Temperley{\textendash}lieb algebras at roots of unity, a fusion
  category and the jones quotient.
\newblock {\em Mathematical Research Letters}, 26(1):121--158, 2019.

\bibitem{Jimbo}
M.~Jimbo.
\newblock A q-analogue of {U}(gl({N}+1)), {H}ecke algebra, and the
  {Y}ang-{B}axter equation.
\newblock {\em Letters in Mathematical Physics}, 11(3):247--252, Apr. 1986.

\bibitem{Jones1}
V.~F.~R. Jones.
\newblock Index for subfactors.
\newblock {\em Invent Math}, 72(1):1--25, 1983.

\bibitem{Jones2}
V.~F.~R. Jones.
\newblock A polynomial invariant for knots via von {N}eumann algebras.
\newblock {\em Bulletin of the American Mathematical Society}, 12(1):103--112,
  1985.

\bibitem{Kauffman}
L.~H. Kauffman.
\newblock An invariant of regular isotopy.
\newblock {\em Transactions of the American Mathematical Society},
  318(2):417--471, 1990.

\bibitem{KV1}
S.~Kerov.
\newblock {\em Asymptotic Representation Theory of the Symmetric Group and its
  Applications in Analysis}.
\newblock American Mathematical Society, June 2003.

\bibitem{KV3}
S.~V. Kerov and A.~M. Vershik.
\newblock The characters of the infinite symmetric group and probability
  properties of the robinson–schensted–knuth algorithm.
\newblock {\em SIAM Journal on Algebraic Discrete Methods}, 7(1):116--124,
  1986.

\bibitem{MartinBook}
P.~Martin.
\newblock {\em Potts Models and Related Problems in Statistical Mechanics}.
\newblock World Scientific, 1991.

\bibitem{MartinSW}
P.~P. Martin.
\newblock On {S}chur-{W}eyl duality, ${A}_n$ {H}ecke algebras and quantum
  $sl({N})$ on $\otimes^{n+1}\mathbb{C}^{N}$.
\newblock {\em International Journal of Modern Physics A}, 07:645--673, 1992.

\bibitem{MeProj}
S.~T. Moore.
\newblock Diagrammatic morphisms between indecomposable modules of
  $\bar{U}_{q}(\mathfrak{sl}_{2})$.
\newblock {\em International Journal of Mathematics}, 31(02):2050016, Jan.
  2020.

\bibitem{MeTL}
S.~T. Moore.
\newblock On the representation theory of the infinite
  {T}emperley{\textendash}{L}ieb algebra.
\newblock {\em Journal of Algebra and Its Applications}, 20(11), Aug. 2020.

\bibitem{Morrison}
S.~Morrison.
\newblock A formula for the {J}ones-{W}enzl projections.
\newblock {\em arXiv:1503.00384}, Mar. 2015.

\bibitem{Neretin}
Y.~A. Neretin.
\newblock Some remarks on traces on the infinite-dimensional {I}wahori--{H}ecke
  algebra.
\newblock {\em arXiv:2101.02133}, 2021.

\bibitem{Olshanski}
G.~Olshanski.
\newblock {\em An introduction to harmonic analysis on the infinite symmetric
  group}, pages 127--160.
\newblock Springer Berlin Heidelberg, Berlin, Heidelberg, 2003.

\bibitem{RSA}
D.~Ridout and Y.~Saint-Aubin.
\newblock Standard modules, induction and the structure of the
  {T}emperley-{L}ieb algebra.
\newblock {\em Advances in Theoretical and Mathematical Physics},
  18(5):957--1041, 2014.

\bibitem{Sitaraman}
M.~Sitaraman.
\newblock Topological actions of {T}emperley-{L}ieb algebras and representation
  stability.
\newblock {\em arXiv:2008.09636}, Aug. 2020.

\bibitem{TL}
H.~N.~V. Temperley and E.~H. Lieb.
\newblock Relations between the 'percolation' and 'colouring' problem and other
  graph-theoretical problems associated with regular planar lattices: Some
  exact results for the 'percolation' problem.
\newblock {\em Proceedings of the Royal Society of London. Series A,
  Mathematical and Physical Sciences}, 322(1549):251--280, 1971.

\bibitem{Thoma}
E.~Thoma.
\newblock Die unzerlegbaren, positiv-definiten klassenfunktionen der
  abz{\"a}hlbar unendlichen, symmetrischen gruppe.
\newblock {\em Mathematische Zeitschrift}, 85:40--61, 1964.

\bibitem{KV2}
A.~M. Vershik and S.~V. Kerov.
\newblock Asymptotic theory of characters of the symmetric group.
\newblock {\em Functional Analysis and Its Applications}, 15(4):246--255, 1982.

\bibitem{Wahl2}
J.~Wahl.
\newblock Traces on diagram algebras {II}: Centralizer algebras of easy groups
  and new variations of the young graph.
\newblock {\em arXiv:2009.08181}, Sept. 2020.

\bibitem{Wahl1}
J.~Wahl.
\newblock Traces on diagram algebras {I}: Free partition quantum groups, random
  lattice paths and random walks on trees.
\newblock {\em Journal of the London Mathematical Society}, 105(4):2324--2372,
  Mar. 2022.

\bibitem{Wassermann}
A.~J. Wassermann.
\newblock {\em Automorphic actions of compact groups on operator algebras}.
\newblock PhD thesis, University of Pennsylvania, 1981.

\bibitem{Wenzl}
H.~Wenzl.
\newblock On sequences of projections.
\newblock {\em C. R. Math. Rep. Acad. Sci. Canada}, 9:5--9, 1987.

\end{thebibliography}
\bibliographystyle{abbrv}

\Addresses

\end{document}